\documentclass[a4paper,11pt]{scrartcl}
\addtokomafont{disposition}{\rmfamily}

\usepackage[left=2.5cm,right=2.5cm,top=2.5cm,bottom=2.5cm]{geometry}

\usepackage{amsthm}
\usepackage{amsmath}
\usepackage{amssymb}
\usepackage{mathrsfs} 
\usepackage{graphicx}

\usepackage[dvipsnames]{xcolor}
\usepackage{stackengine}
\usepackage{placeins}
\usepackage{arydshln}
\usepackage{bbm}
\usepackage{textcomp}
\usepackage{enumitem}
\usepackage{mathtools}
\usepackage{tabularx}
\usepackage{booktabs}
\usepackage{rotating}

\usepackage{color}
\usepackage{framed}
\usepackage{comment}
\definecolor{shadecolor}{gray}{0.9}

\usepackage{dsfont} 
\usepackage{caption} 
\usepackage{subcaption} 
\usepackage{graphicx}
\usepackage{pdfpages}
\usepackage{url}
\usepackage[english]{babel}
\usepackage{setspace}
\usepackage{dashrule}
\usepackage{tikz}
\usetikzlibrary{automata, positioning, arrows}
\tikzset{
        ->,  
        node distance=5.5cm, 
        every state/.style={thick, fill=gray!10}, 
        initial text=$ $, 
        }

\usepackage{xr-hyper} 
\usepackage[
breaklinks,
colorlinks=true,
pagebackref=false,
pdfauthor={},%
pdftitle={The Efficiency Gap},%
citecolor=blue,
linkcolor=blue,
urlcolor=blue,
filecolor=blue
]{hyperref}      
\usepackage{cleveref}
\usepackage[authoryear,round]{natbib}     

\specialcomment{extra}{\begin{shaded}}{\end{shaded}}

\theoremstyle{plain}  
\newtheorem{theorem}{Theorem}[section] 
 
\newtheorem{proposition}[theorem]{Proposition}

\theoremstyle{definition} 
\newtheorem{definition}[theorem]{Definition}

\newtheorem{rem}[theorem]{Remark}

\newtheoremstyle{assumption}
{3pt}
{3pt}
{}
{}
{\bf}
{.}
{.5em}
{\thmname{#1} (\thmnote{#3}\thmnumber{#2})}

\theoremstyle{assumption}
\newtheorem{ass}{Assumption}
\newtheorem{RunningExmp}{Running Example}

\theoremstyle{remark}

\newcommand{\dint}{\,\mathrm{d}}
\newcommand{\E}{\mathbb{E}}

\newcommand{\ES}{\operatorname{ES}}

\newcommand{\Var}{\operatorname{Var}}

\newcommand{\F}{\mathcal{F}}
\newcommand{\R}{\mathbb{R}}
\newcommand{\A}{\mathbb{A}}

\newcommand{\one}{\mathds{1}}
\newcommand{\interior}{\operatorname{int}}

\newcommand{\ph}{\varphi}

\DeclareMathOperator*{\argmin}{arg\,min}

\newcommand{\N}{\mathbb N}
\renewcommand{\P}{\mathbb P}

\renewcommand{\a}{\alpha}

\newcommand{\X}{\mathcal{X}}
\newcommand{\Y}{\mathcal{Y}}
\newcommand{\Z}{\mathcal{Z}}

\newcommand{\tod}{\overset{d}{\longrightarrow}}

\renewcommand{\rm}{\normalfont \rmfamily}
\renewcommand{\bf}{\normalfont \bfseries}

\newcommand{\as}{\text{a.s.}}

\def\be{\begin{equation} \label}
\def\ee{\end{equation}}

\numberwithin{equation}{section} 

\usepackage[colorinlistoftodos,textsize=tiny]{todonotes}
\newcommand{\Comments}{1}
\newcommand{\mynote}[2]{\ifnum\Comments=1\textcolor{#1}{#2}\fi}
\newcommand{\mytodo}[2]{\ifnum\Comments=1%
  \todo[linecolor=#1!80!black,backgroundcolor=#1,bordercolor=#1!80!black]{#2}\fi}

\ifnum\Comments=1               
\fi

\ifnum\Comments=1               
\fi

\ifnum\Comments=1               
\fi

\allowdisplaybreaks

\begin{document}

\title{The Efficiency Gap\footnote{
		We are very grateful to Jana Hlavinov\'a for a careful proofreading and valuable feedback on an earlier version of this paper and thank Andrew Patton and Sander Barendse for a fruitful discussion of the results of the paper.
		T.~Dimitriadis gratefully acknowledges support of the German Research Foundation (DFG) through grant number 502572912, of the Heidelberg Academy of Sciences and Humanities and the Klaus Tschira Foundation. J.~Ziegel gratefully acknowledges support of the Swiss National Science Foundation.
		}
	}
\author{Timo Dimitriadis\thanks{Heidelberg University, Alfred Weber Institute of Economics, Bergheimer Str.\ 58, 69115 Heidelberg, Germany and
		Heidelberg Institute for Theoretical Studies, 69118 Heidelberg, Germany, e-mail: \href{mailto:timo.dimitriadis@h-its.org}{timo.dimitriadis@awi.uni-heidelberg.de}}
	\and Tobias Fissler\thanks{Vienna University of Economics and Business (WU), Department of Finance, Accounting and Statistics, Welthandelsplatz 1, 1020 Vienna, Austria, 
		e-mail: \href{mailto:tobias.fissler@wu.ac.at}{tobias.fissler@wu.ac.at}} \and Johanna Ziegel\thanks{University of Bern, Department of Mathematics and Statistics, Institute of Mathematical Statistics and Actuarial Science, Alpeneggstrasse 22, 3012 Bern, Switzerland, 
		e-mail: \href{mailto:johanna.ziegel@stat.unibe.ch}{johanna.ziegel@stat.unibe.ch}}
}
\maketitle

\begin{abstract}
\noindent
\textbf{Abstract.}
Parameter estimation via M- and Z-estimation is equally powerful in semiparametric models for one-dimensional functionals due to a one-to-one relation between corresponding loss and identification functions via integration and differentiation. 
For multivariate functionals such as multiple moments, quantiles, or the pair (Value at Risk, Expected Shortfall), this one-to-one relation fails and not every identification function possesses an antiderivative. 
The most important implication is an \emph{efficiency gap}: The most efficient Z-estimator often outperforms the most efficient M-estimator. 
We theoretically establish this phenomenon for multiple quantiles at different levels and for the pair (Value at Risk, Expected Shortfall), and illustrate the gap numerically.
Our results further give guidance for pseudo-efficient M-estimation for semiparametric models of the Value at Risk and Expected Shortfall.
\end{abstract}

\section{Introduction}
\label{sec:intro}

\onehalfspacing

Given some real-valued response variable $Y_t$ and some $p$-dimensional vector of covariates $X_t$, one is often interested in modelling the effect of the covariates on the response variable through regression models.
E.g., one might be interested in the average effect of economic and financial conditions as e.g.\ inflation on GDP growth.
The classical mean regression technique captures the \textit{average} effect by modelling the \textit{expectation} of the conditional distribution of $Y_t$ given $X_t$, denoted by $F_{t}$.
However, researchers are often interested in different properties of this conditional distribution, e.g., in low quantiles if attention is focused on downside risks of GDP growth as in \citet{Adrian2019}.
This can be facilitated through quantile regression \citep{Koenker1978}, where one parametrically models the quantile of the conditional distribution $F_{t}$.

More generally, one is interested in a certain statistical \textit{functional} $\Gamma$ of the conditional distribution $F_{t}$, where the functional maps a (conditional) distribution to a real-valued outcome.
The functional of interest varies among disciplines: E.g., quantitative risk managers are specifically interested in models for risk measures such as conditional variances (volatility), quantiles (Value at Risk, VaR), expectiles and Expected Shortfall (ES) \citep{Bollerslev1986, Engle2004, Efron1991, Patton2019}.
Epidemiological forecasts, of particular importance due to the COVID-19 pandemic, often focus on prediction intervals, which commonly consist of two quantiles \citep{Bracher2021NatComm, Cramer2022}. 

It is common practice to model the functional as some parametric model  $\Gamma(F_{t}) = m(X_t,\theta_0)$ for some unique parameter $\theta_0 \in \Theta \subseteq \R^q$.
This specification is commonly referred to as \emph{semiparametric}: 
Even though the model $m$ itself is parametric, it does not specify the full conditional distribution $F_{t}$, but only a functional $\Gamma$ thereof \citep{Newey1990, BKRW1998book}.

While standard approaches often model every functional of interest separately, joint semiparametric models for \textit{multivariate} (or vector-valued) functionals have desirable advantages in many instances:
A joint treatment of two quantile levels is e.g.\ beneficial for prediction intervals \citep{Shrestha2006, Bracher2021NatComm}, it can impede quantile crossings \citep{GourierouxJasiak2008, WhiteKimManganelli2015, Catania2019}, and it generally improves efficiency.
More fundamentally, there are cases where  M- or Z-estimation of  univariate models is infeasible such as for the variance, ES and Range Value at Risk (RVaR, also called ``interquantile expectation'', which nests the trimmed mean), since suitable loss or identification functions for these functionals do not exist; see \cite{Osband1985}, \cite{Weber2006}, \cite{WangWei2020}, \cite{DFZ_CharMest}.
However, such objective functions exist for an appropriate multivariate functional; see \cite{FisslerZiegel2016} for the pair (VaR, ES), \cite{Osband1985} for the pair (mean, variance), and \cite{FisslerZiegel2019} for the triplet of the RVaR with two quantiles.
These examples motivate our consideration of joint estimation of multivariate models.

Estimation of the parameter $\theta_0$ in semiparametric models is regularly carried out by either minimum (M-) or zero (Z-) estimation \citep{NeweyMcFadden1994}.
Given these estimators are consistent and asymptotically normal, one favors an \textit{efficient} estimator with an associated covariance matrix which is as small as possible.
Besides more accurate estimates, this allows for more powerful inference through tests and confidence intervals.

In this article, we investigate the efficiency of M-estimators, based on some \emph{loss functions}, in particular in relation to Z-estimators, which are based on \emph{identification functions} (or moment conditions).
We show the existence of an ``efficiency gap'' for multivariate functionals in the sense that the semiparametric M-estimator cannot attain the Z-estimation or semiparametric efficiency bound in the sense of \cite{Stein1956}.
For this, we make use of a recent result of \citet{DFZ_CharMest} that fully characterizes the class of consistent, semiparametric M-estimators for general functionals through the classes of strictly consistent loss functions from the literature on forecast evaluation \citep{Gneiting2011, FisslerZiegel2016}.
For vector-valued functionals, theses latter classes are considerably smaller than the corresponding classes of  identification functions used in Z-estimation.
This is in stark contrast to the univariate case, where these classes are almost equivalent and M- and Z-estimation can be equally efficient.
As a stepping stone, we derive the novel result that the ``optimal instrument matrix" of \cite{Chamberlain1987} and \cite{Newey1993} is not only a sufficient, but also a necessary condition for efficient Z-estimation.

Throughout the article, we recurrently make use of the running example of a double quantile model---i.e., a semiparametric model for two quantiles at different levels---to illustrate and exemplify our general theoretical results.
In particular, we derive conditions for the occurrence of the efficiency gap and illustrate these in simulations.
Our results directly generalize to finitely many quantiles.
This model class arises naturally in the following fields of applications:
In quantitative risk management, one is interested in quantiles (VaR) of financial returns at two small probability levels, say 1\% and 2.5\%, which directly motivates the joint modelling of two quantiles \citep{Engle2004, WhiteKimManganelli2015, Catania2019}.
Furthermore, prediction intervals can naturally be defined as the interval spanned by two (conditional) quantiles with levels of e.g., $5\%$ and $95\%$ \citep{BrehmerGneiting2020, FisslerHlavinovaRudloff2019Theory, Bracher2021NatComm}.
Eventually, the entire conditional distribution can conveniently be approximated by multiple conditional quantiles; see e.g., \cite{Buchinsky1994}, \cite{Angrist2006}, \cite{Chernozhukov2010} for microeconomic and \cite{Adrian2019} for macroeconomic applications.
While models for individual quantile levels could be estimated separately, an important methodological demand on reasonable models is to impede quantile crossings \citep{Koenker2005}, which can be achieved through joint models as in \cite{GourierouxJasiak2008}, \cite{WhiteKimManganelli2015} and \cite{Catania2019}.
Moreover, joint estimation generally improves efficiency.

We further illustrate that the efficiency gap arises for the popular and recently proposed joint models for the VaR and ES \citep{Patton2019}.  
While the as yet common choices of loss functions used for their M-estimation are rather ad hoc, we provide two novel \textit{pseudo-efficient} loss functions, that is, choices which result in efficient M-estimation at least in specific (but realistic) situations.
We illustrate their superiority in simulations, especially for small probability levels that are of particular importance in risk management.
The first pseudo-efficient choice is ``surprisingly feasible'' in the sense that it requires very little pre-estimates compared to classical semiparametric models for the mean or quantiles.
This finding suggests an improved and practically relevant M-estimator for semiparametric VaR and ES models.
We anticipate that the efficiency gap generalizes to joint models for various other vector-valued functionals like multiple expectiles or the RVaR, jointly with corresponding quantiles.

The paper is organized as follows.
Section \ref{sec:Notation and Setting} formally introduces M- and Z-estimation and relates these to the literature on forecast evaluation that we quickly review in Section \ref{sec:LossIdFunctions}.
Section \ref{sec:EfficientEstimation} considers efficient M- and Z-estimation of general semiparametric models and attainability of the semiparametric efficiency bound.
In Section \ref{sec:ExamplesRegressionFrameworks}, we 
establish the efficiency gap for double quantile models and models for (VaR, ES), which is illustrated in simulations in Section \ref{sec:Simulation}.
The Supplementary Material contains all proofs in Section \ref{app:Proofs}, analyzes efficient estimation of the pair (mean, variance) in Section \ref{sec:DoubleMomentRegressionModel}, discusses the impact of the gap on equivariant estimation in Section \ref{sec:implications}, and contains further technical details in its subsequent sections.

\section{M- and Z-estimation}
\label{sec:Notation and Setting}

We consider a time series $Z_t = (Y_t,X_t)$, $t\in\N$, where $Y_t$ are real-valued response variables and $X_t$ are $\R^p$-valued regressors, that can potentially contain lagged values of $Y_t$, allowing for autoregressive models.
Let $\F_\Z$ be a class of possible joint distributions of $Z_t$ that formalizes the uncertainty about the distribution of our time series.
$\F_\Z$  induces a class $\F_\X$ of marginal distributions of $X_t$ and a class $\F_{\Y|\X}$ of conditional distributions, $F_t$, of $Y_t$ given $X_t$.
Whenever they exist, we denote the conditional density by $f_t$, the conditional expectation by $\E_t[\cdot] = \E[\cdot  \,|\, X_t]$ and the conditional variance by $\Var_t(\cdot) = \Var(\cdot  \,|\, X_t)$.
Equalities of random variables are meant to hold almost surely if not stated otherwise.

Let $\Gamma \colon\F_{\Y|\X}\to\Xi\subseteq\R^k$ be some $k$-dimensional and measurable functional of the conditional distributions $F_t$.
Standard examples for univariate functionals are the mean or quantiles.
Later on, we consider a pair of two quantiles and the pair consisting of the VaR and ES as examples for multivariate functionals.
Let $\Theta\subseteq \R^q$ be a parameter space with non-empty interior, $\interior(\Theta)$, and $m\colon \R^p\times \Theta\to \Xi$ a parametric and (in $\theta$) differentiable model for the functional $\Gamma$.
We denote the gradient of its $j$-th component by the column vector $\nabla_{\theta} m_j(X_t,\theta)\in\R^q$, $j= 1,\dots,k$.
We work under the following assumption of a correctly specified model with a unique parameter. 
\begin{ass}\label{ass:unique model}
	For all distributions $F_{Z_t} \in \F_\Z$ of $Z_t=(Y_t, X_t)$, there is a unique and time-independent parameter $\theta_0 = \theta_0(F_{Z_t}) \in \interior(\Theta)$ such that $m(X_t,\theta_0) = \Gamma(F_t)$ for all $t\in\N$.
\end{ass}

We dispense with a strong stationarity assumption on the time series $Z_t$, however, Assumption \eqref{ass:unique model} imposes a \emph{semiparametric stationarity} assumption in that the parameter $\theta_0$, and hence the functional $\Gamma(F_t)$ is time-independent, allowing e.g., for heteroskedasticity.

Following \citet{Huber1967} and \cite{NeweyMcFadden1994}, we consider M-estimators for $\theta_0$
\be{eq:M-est}
	\widehat \theta_{M,T} = \argmin_{\theta \in\Theta} \frac{1}{T} \sum_{t=1}^T \rho_t \big(Y_t , m(X_t,\theta)\big),
\ee
based on possibly time-varying loss functions $\rho_t$, which are the key ingredient of the M-estimator.
The core condition on $\rho_t$ for the consistency of $\widehat \theta_{M,T}$ is that 
\begin{align}
	\label{eq:model cons}
	\E\big[\rho_t \big(Y_t , m(X_t,\theta_0)\big)\big]<\E\big[\rho_t \big(Y_t , m(X_t,\theta)\big)\big]\qquad \text{for all } \theta\neq\theta_0,\quad \text{for all } t\in\N,
\end{align}
which we call \emph{strict $\F_{\Z}$-model-consistency} of $\rho_t$ for $m$ as in \citet{DFZ_CharMest}; also see \citet[Properties 3.3 and 3.4]{Gourieroux1987}.

A standard alternative to M-estimation are zero (Z-) or method of moments (MM-) estimators \citep{Hansen1982, NeweyMcFadden1994}, given by
\begin{align}
	\label{eq:Z-est}
	\widehat \theta_{Z,T} = \argmin_{\theta \in\Theta} \Big\|\frac{1}{T} \sum_{t=1}^T 
	\psi_t (Y_t , X_t,\theta)
	\Big\|^2.
\end{align}
The name arises since the minimization in \eqref{eq:Z-est} essentially sets the average of the $q$-dimensional, possibly time varying functions $\psi_t$ to zero.
Hence, consistency of the Z-estimator crucially relies on the \emph{strict unconditional $\F_{\Z}$-identification} condition
\begin{align}
	\label{eq:model ind}
	\Big(\E\big[\psi_t (Y_t , X_t,\theta)\big]=0 
	\quad \Longleftrightarrow \quad \theta=\theta_0\Big) \qquad \text{for all } \theta\in\Theta, \quad \text{for all } t\in\N.
\end{align}
The functions $\psi_t$ in \eqref{eq:model ind} are often the gradients of the losses $\rho_t$ in \eqref{eq:model cons}.
We do not consider the standard extension to generalized method of moments (GMM) estimation, where $\psi_t$ can be of larger dimension than $\theta$, as the exactly identified case in \eqref{eq:Z-est} suffices for efficient estimation; see Theorem \ref{theorem:GeneralEfficientA}  and Remark  \ref{rem:GMM} for details.

For semiparametric estimation, there exist a multitude of choices for the functions $\rho_t$ and $\psi_t$ that satisfy the conditions \eqref{eq:model cons} and \eqref{eq:model ind} respectively \citep{GourierouxMonfortTrognon1984, Komunjer2005}.
This opens up the possibilities to optimally choose $\rho_t$ and $\psi_t$, e.g., for efficient estimation \citep{Newey1993}.
To characterize such bounds, it is essential to characterize the entire classes of functions $\rho_t$ and $\psi_t$ such that \eqref{eq:model cons} and \eqref{eq:model ind} hold.
For this, \citet{DFZ_CharMest} formally connect these conditions to the notions of strictly consistent loss and strict identification functions from the literature on forecast evaluation \citep{Gneiting2011}, which we shortly review in the following.

\section{Strictly consistent loss and strict identification functions}
\label{sec:LossIdFunctions}

Throughout this section, let $Y \sim F \in \F$ be a real-valued random variable, where $\F$ is a generic class of probability distributions on $\R$.
We consider the single-valued functional $\Gamma \colon \F \to\Xi$ that attains values in the $k$-dimensional \emph{action domain} $\Xi\subseteq \R^k$.
A map $a\colon\R\times \Xi \to\R^\ell$, $\ell \in \mathbb{N}$, is called \emph{$\F$-integrable} if $\mathbb{E}[a(Y, \xi)]$ exists and is finite for all $Y \sim F\in\F$ and for all $\xi\in\Xi$.

\begin{definition}[Consistency and elicitability]\label{defn:consistency}
	An $\mathcal{F}$-integrable map $\rho\colon \R\times \Xi\to\R$ is an \emph{$\F$-consistent} loss function for a functional $\Gamma\colon \F\to\Xi$ if
	\be{eq:strict consistency}
		\E \big[ \rho \big(Y,\Gamma(F)\big)\big] \le  \E \big[ \rho (Y, \xi) \big] \qquad \text{for all } Y \sim F\in\F, \ \text{for all }\xi\in\Xi\,.
	\ee
	If equality in \eqref{eq:strict consistency} implies that $\xi = \Gamma(F)$, then $\rho$ is called \emph{strictly $\F$-consistent} for $\Gamma$.
	A functional $\Gamma$ is \emph{elicitable on $\F$} if there is a strictly $\F$-consistent loss function for it.
\end{definition}

The crucial difference to unconditional model consistency in  \eqref{eq:model cons} is that in \eqref{eq:strict consistency}, the expectation is only taken with respect to $Y$.
The whole classes of (strictly) consistent losses are characterized for many functionals \citep{Gneiting2011, FisslerZiegel2016}.
E.g., under richness conditions on the class $\F$, one can show that $\rho$ is (strictly) $\F$-consistent for the mean functional if and only if it is a \emph{Bregman loss}  $\rho(y,\xi) = \phi(y)- \phi(\xi) +\phi'(\xi)(\xi-y) + \kappa(y)$ where $\phi$ is a (strictly) convex function on $\R$ with subgradient $\phi'$, and the function $\kappa\colon\R\to\R$ is such that $\F$-integrability holds \citep{Savage1971, Gneiting2011}.
This class nests the omnipresent squared loss $\rho(y,\xi) = (y-\xi)^2$.
Likewise, under similar richness conditions and if $\F$ contains only distributions with a unique $\alpha$-quantile, a loss is strictly $\F$-consistent for the $\alpha$-quantile with $\alpha \in (0,1)$, if and only if $\rho$ is a generalized piecewise linear loss functions $\rho(y,\xi) = (\one_{\{y\le \xi\}} - \alpha )(g(\xi)-g(y)) + \kappa(y)$, where $g$ is (strictly) increasing, and $\kappa$ a function ensuring $\F$-integrability \citep{Gneiting2011b}.
This class nests the well known pinball loss $\rho(y,\xi) = (\one_{\{y\le \xi\}} - \alpha)(\xi - y)$.
The following running example is used illustratively throughout the paper.

\begin{RunningExmp} 
	\label{RunningExmp1}
	Consider the double quantile $\Gamma(\cdot) = \big(Q_\alpha(\cdot), Q_\beta(\cdot)\big)$ at probability levels $0<\alpha<\beta<1$ and a class $\F$ of strictly increasing distribution functions fulfilling the richness Assumption \eqref{ass:GPLclass} in Appendix \ref{sec:AddAssumptions}. 
	\citet[Proposition 4.2]{FisslerZiegel2016} characterizes the class of (strictly) $\F$-consistent losses $\rho\colon\R\times\Xi\to\R$, $\Xi\subseteq \R^2$, for $\Gamma$ as
	\begin{align}
	\begin{aligned}
	\label{eqn:DoubleQuantileGeneralLossFunctions}
	\rho(y, \xi_1, \xi_2) = &\big( \mathds{1}_{\{y \le \xi_1\}} -\alpha \big) \big(g_{1}(\xi_1) -  g_{1}(y) \big) \\
	+ \, &\big(\mathds{1}_{\{y \le \xi_2\}} - \beta \big) \big(g_{2}(\xi_2) - g_{2}(y) \big) + \kappa(y),
	\end{aligned}
	\end{align}
	where $g_{1}, g_2: \mathbb{R} \to \mathbb{R}$ are (strictly) increasing, and $\kappa\colon\R\to\R$ is such that $\rho$ is  $\F$-integrable. 
	Strikingly, this means that the whole class of (strictly) consistent losses for the double quantile coincides with the sum of (strictly) consistent losses for the individual quantiles.
\end{RunningExmp}

In forecast evaluation, identification functions are used to check (conditional) calibration of forecasts \citep{NoldeZiegel2017, DimiPattonSchmidt2019}, akin to goodness-of-fit tests.

\begin{definition}[Identification function and identifiability]\label{defn:identifiability}
	An $\mathcal{F}$-integrable map $\ph\colon \R \times \Xi\to\R^k$ is an \emph{$\F$-identification function} for a functional $\Gamma\colon \F\to\Xi \subseteq \R^k$ if 
	\(
	\mathbb{E}\big[ \ph\big(Y,\Gamma(F)\big) \big] =0
	\)
	for all $Y \sim F \in \F$.
	 If additionally
	$
	\mathbb{E} \big[ \ph(Y,\xi) \big] =0$  implies that  $\xi = \Gamma(F)
	$
	for all $F\in\F$ and for all $\xi\in\Xi$, it is a \emph{strict $\F$-identification function} for $\Gamma$.
	A functional $\Gamma$ is called \emph{identifiable on $\F$} if there is a strict $\F$-identification function for it.
\end{definition}


Given a strict $\F$-identification function $\ph\colon\R\times\Xi\to\R^k$ for a functional $\Gamma\colon\F\to\Xi\subseteq \R^k$, \citet[Theorem 4]{DFZ_OsbandID} shows that under some regularity conditions and richness conditions on $\F$, the full class of strict $\F$-identification functions for $\Gamma$ is given by
\begin{align}
	\label{eq:all identification functions}
	\big\{ h(\xi)\ph(y,\xi)\,|\, h\colon \Xi \to\R^{k\times k}, \ \det (h(\xi))\neq0 \ \text{for all }\xi\in\Xi \big\}.
\end{align}
This characterization result is valid for any identifiable functional.
In contrast, there is no such general characterization result available for the class of strictly consistent loss functions for a given elicitable functional. They need to be established on a case-by-case basis.


\begin{RunningExmp}
	\label{RunningExmp3}
	Let $\F$ be the class of continuous and strictly increasing distribution functions. The double quantile functional possesses a strict $\F$-identification function $\ph(y,\xi_1,\xi_2) = \big( \mathds{1}_{\{y \le \xi_1 \}} - \alpha, \; \mathds{1}_{\{y \le \xi_2  \}} - \beta \big)^\intercal$.
	Equation \eqref{eq:all identification functions} provides a rich family of further strict $\F$-identification functions, e.g., choosing 
	$h(\xi_1,\xi_2)=
	\big(\begin{smallmatrix}
		1 & 1\\
		0 & 1
	\end{smallmatrix}\big)$ 
	leads to $\ph'(y,\xi_1,\xi_2) = h(\xi_1,\xi_2) \ph(y,\xi_1,\xi_2)
	= \big( \mathds{1}_{\{y \le \xi_1 \}} - \alpha + \mathds{1}_{\{y \le \xi_2  \}} - \beta, \; \mathds{1}_{\{y \le \xi_2  \}} - \beta \big)^\intercal$.
\end{RunningExmp}

There is an intimate relationship between (strictly) consistent loss functions and strict identification functions for $\Gamma$ via differentiation and integration. 
For \emph{one-dimensional} functionals $\Gamma$, these two classes are essentially equivalent:
On the one hand, under sufficient smoothness and regularity conditions, first-order conditions yield that the derivative of any (strictly) consistent loss for $\Gamma$ is an identification function, whose strictness however requires some additional care.
On the other hand, Osband's principle \citep{Osband1985, Gneiting2011} implies that---under sufficient regularity conditions---if $\ph$ is an oriented identification function for $\Gamma$, then for any consistent loss $\rho$ there is a real-valued function $h$ such that 
\begin{align}
	\label{eq:Osband}
	\nabla_\xi \mathbb{E} \big[ \rho(Y,\xi) \big] = h(\xi) \mathbb{E} \big[  \ph(Y,\xi) \big]
	\qquad \text{ for all } \xi\in\Xi \quad \text{and}  \quad Y \sim F \in \F. 
\end{align}

The relation  between loss and identification functions is more involved for \emph{multivariate} functionals $\Gamma$, and it turns out that \emph{there are considerably more identification functions than consistent losses}.
This disparity proves to be consequential for efficient estimation of semiparametric models for vector-valued functionals, as discussed in the subsequent sections of this article.

In more detail, the gradient of any (strictly) consistent loss is still a (multivariate) identification function for $\Gamma$.
For the reverse direction, \eqref{eq:Osband} holds equivalently with $h$ being $(k\times k)$-matrix valued. However, $h(\xi) \mathbb{E} [ \ph(Y,\xi) ]$ can only have an antiderivative if the Hessian $\nabla_\xi^2 \mathbb{E} [ \rho(Y,\xi) ]$ is
symmetric; see \citet[Corollary 3.3]{FisslerZiegel2016} for a rigorous statement.
This result imposes strong conditions on $h$ as illustrated with the following running example.

\begin{RunningExmp}
	\label{exmp:illustration osband}
	\citet[Proposition 4.2(i)]{FisslerZiegel2016} yields that the derivative of any expected (strictly) $\F$-consistent loss function for the double quantile takes the form $h(\xi_1,\xi_2) \mathbb{E} \big[ \ph(Y,\xi_1,\xi_2) \big]$ where 
	$h(\xi_1,\xi_2) = \text{diag}\big(w_1(\xi_1), w_2 (\xi_2)\big)$ and $w_1$, $w_2$ are non-negative, subject to the richness Assumption \eqref{ass:GPLclass}.
	This constitutes the argument for the characterization of all (strictly) consistent loss functions in \eqref{eqn:DoubleQuantileGeneralLossFunctions} where clearly $w_j = g_j'$, $j=1,2$.
	On the other hand, there is evidently a considerably larger class of $\R^{2\times 2}$-valued functions $h$ such that $\det(h(\xi_1, \xi_2))\neq 0$ for all $(\xi_1, \xi_2)\in\Xi$.
	E.g., $\ph'$ in Running Example \eqref{RunningExmp3} cannot arise as the derivative of a strictly consistent loss for the double quantile functional as the corresponding 
	$h=\big(\begin{smallmatrix}
		1 & 1\\
		0 & 1
	\end{smallmatrix}\big)$  
	is not diagonal.
\end{RunningExmp}

We refer to Supplement Section \ref{sec:Details} for further remarks and technical details on the connection between loss and identification functions.


\section{Efficient Semiparametric Estimation}
\label{sec:EfficientEstimation}

Recall the M-estimator at \eqref{eq:M-est}, where the losses $\rho_t$ have to satisfy \eqref{eq:model cons}, which closely resembles the notion of strict consistency in Section \ref{sec:LossIdFunctions}.
\citet[Theorem 1]{DFZ_CharMest} shows that these two conditions are equivalent under Assumptions \eqref{ass:unique model} and \eqref{ass:CharMest}, i.e., a semiparametric M-estimator is consistent if and only if a (strictly) consistent loss function is used.

\begin{RunningExmp}
	Let $m(X_t,\theta) = \big( q_\alpha(X_t,\theta),  q_\beta(X_t,\theta) \big)^\intercal$ be some semiparametric model for the double quantile functional where $\theta \in \Theta \subseteq \mathbb{R}^q$.
	Then, \citet[Theorem 1] {DFZ_CharMest} yields that under our Assumptions \eqref{ass:unique model} and \eqref{ass:CharMest}, a loss $\rho\colon \R\times\Xi\to\R$, $\Xi\subseteq \R^2$, is $\F_\Z$-model-consistent for $m$ if and only if $\rho$ is of the form given in \eqref{eqn:DoubleQuantileGeneralLossFunctions}.
	This implies that the M-estimator for the double quantile model can only be consistent if $\rho$ is of the form given in \eqref{eqn:DoubleQuantileGeneralLossFunctions}.
\end{RunningExmp}

Such characterization results for the full class of consistent M-estimators allow to determine an asymptotically most efficient M-estimator.
To this end, it is helpful to relate the asymptotic distributions of M- and Z-estimators, which coincide if the identification functions $\psi_t$ of the latter match the derivative with respect to $\theta$ of the loss $\rho_t$ of the former; see e.g., Theorems 3.1, 3.2, and the discussion on p.\ 2145 in \cite{NeweyMcFadden1994} for details.
For non-differentiable losses, this rationale holds on the level of the differentiable conditional expectations \citep[Theorems 7.1 und 7.2]{NeweyMcFadden1994}.
Consequently, in the sequel we say that an M-estimator has an equivalent Z-estimator if the derivative of (the conditional expectation of) the loss function with respect to $\theta$ equals (the conditional expectation of) the identification function almost surely.
Also notice that the asymptotic covariance of M-estimators is invariant to rescaling by constants $c$ and additions of terms $\kappa_t(Y_t)$ with the consequence that we dispense with a discussion of these terms in the sequel.

Following \cite{Chamberlain1987}, \cite{Gourieroux1987}, \cite{Newey1990} among many others, we consider functions $\psi_t$ in (\ref{eq:Z-est}) based on \emph{conditional moment conditions} of the form
\begin{align}
	\label{eqn:UnconditionalIdentificationGeneralForm}
	\psi_{t} \big(Y_t, X_t, \theta \big) = A_t(X_t, \theta) \, \varphi \big(Y_t, m(X_t,\theta) \big),
\end{align}
where $\ph$ is a strict identification function for the functional $\Gamma$ and the $q \times k$ matrices $A_t(X_t, \theta)$ are often called instrument matrices.
We denote their sequence by $\A = (A_t)_{t\in\N}$ and the resulting Z-estimator at \eqref{eq:Z-est} by $\widehat \theta_{Z,T,\A}$.
This restriction is justified by three reasons:
First, moment conditions of the form \eqref{eqn:UnconditionalIdentificationGeneralForm} generally suffice to reach the semiparametric efficiency bound \citep{Chamberlain1987}.
Second, the derivatives of strictly consistent loss functions take that form, where $A_t(X_t, \theta)$ matches the model gradient.
Third, despite the convenient result \eqref{eq:all identification functions}, a characterization of all consistent Z-estimators in terms of their functions $\psi_t$ is not available; see e.g., \citet{Roehrig1988}, \citet{Komunjer2012} and the supplement of \citet{DFZ_CharMest}.

Henceforth, we assume that the considered M- and Z-estimators are consistent and asymptotically normal.
Primitive conditions for this are widely available, see e.g., \cite{Huber1967}, \cite{Weiss1991}, \cite{NeweyMcFadden1994}, \cite{Andrews1994}, \cite{davidson1994stochastic}.
These conditions include classical moment and dependence conditions on the process $(Y_t, X_t)_{t\in\N}$ together with smoothness assumptions on the conditional expectations of the employed loss and identification functions, and crucially, an identification condition for the model parameters.
For M-estimators, this identification condition is conveniently fulfilled through \citet[Theorem  1]{DFZ_CharMest} by employing strictly consistent loss functions.
However, the analogue condition for the Z-estimator that $\psi_t$ are strict $\F_\Z$-identification functions for $\theta_0$ is more difficult to establish and generally has to be verified on a case-by-case basis; see e.g., Section \ref{app:IDZestDQRModels} for specific results for our running example of the double quantile models.

Henceforth, we impose the following assumption that ensures the specific form of the matrix $\Sigma_{T,\A}$ given in (\ref{eqn:AsymptoticCovarianceZEstimator}), in particular the absence of ``HAC'' terms \citep{NeweyWest1987}.
\begin{ass}\label{ass:PsiUncorrelated}
	Suppose that the sequence $\big(\psi_{t} (Y_t, X_t, \theta_0) \big)_{t\in\mathbb N}$ is uncorrelated.
\end{ass}

Under the above conditions on the Z-estimator $\widehat \theta_{Z,T, \A}$, it holds that
\begin{align}
	\label{eqn:AsymptoticNormalityZEstimator}
	\Sigma_{T,\A}^{-1/2} \Delta_{T,\A} \sqrt{T} \big( \widehat \theta_{Z,T, \A} - \theta_0 \big)  \tod \mathcal{N}(0, I_q),
\end{align} 
where the asymptotic covariance is governed by the terms
\begin{align}
	\label{eqn:AsymptoticCovarianceZEstimator}
	\Sigma_{T,\A} &= \frac{1}{T} \sum_{t=1}^T \mathbb{E} \big[ A_t(X_t, \theta_0) S_t(X_t,\theta_0) A_t(X_t, \theta_0)^\intercal \big]\in\R^{q\times q} \qquad \text{ and } \\ 
	\label{eqn:AsymptoticCovarianceZEstimator2}
	\Delta_{T,\A} &= \frac{1}{T} \sum_{t=1}^T \mathbb{E} \big[ A_t(X_t, \theta_0) D_t(X_t,\theta_0) \big]\in\R^{q\times q},
\end{align}
where, for any $\theta\in\Theta$,
\begin{align}
	\label{eqn:DefSMatrix}
	S_t(X_t,\theta) &= \E_t \left[ \varphi \big(Y_t, m(X_t, \theta )\big) \varphi \big( Y_t,m(X_t,\theta) \big)^\intercal \right] \in\R^{k\times k} \qquad \text{ and } \\
	\label{eqn:DefDMatrix}
	D_t(X_t,\theta) &= \nabla_{\theta} \E_t \big[ \varphi \big( Y_t,m(X_t,\theta) \big) \big]^\intercal\in\R^{k\times q}.
\end{align}

We say that an asymptotically normal estimator is \emph{efficient} if there is no other asymptotically normal estimator with a smaller covariance matrix in the \emph{Loewner order}\label{p: Loewner order} $\succcurlyeq$. For two positive semi-definite matrices $A$ and $B$, we say that $A \succcurlyeq B$ if and only if $A - B$ is positive semi-definite.
Motivated by the discussion in \citet[p.\ 102]{Newey1990}, we deliberately omit an analysis of ``superefficient'' estimators.
The following theorem establishes necessary and sufficient conditions for efficient Z-estimation by extending the theory of \cite{Hansen1985}, \cite{Chamberlain1987} and \cite{Newey1993}.
Notice that the theorem also holds in the case $Y_t \in \R^d$, $d > 1$.

\begin{theorem}
	\label{theorem:GeneralEfficientA}
	Under Assumptions \eqref{ass:unique model} and \eqref{ass:PsiUncorrelated}, let $\ph$ be a strict $\F_{\Y|\X}$-identification function for $\Gamma$. 
	Let $\widehat{\theta}_{Z,T, \A^{\ast}}$ be the Z-estimator at \eqref{eq:Z-est} that is asymptotically normal and based on the strict unconditional $\F_{\Z}$-identification function at \eqref{eqn:UnconditionalIdentificationGeneralForm} with instrument matrices $A_{t,C}^\ast(X_t,\theta)$ such that 
	\begin{align}
	\label{eqn::GeneralEfficientAMatrix}
	A_{t,C}^\ast(X_t,\theta_0) = C D_t(X_t,\theta_0)^\intercal S_t(X_t,\theta_0)^{-1} \qquad  \text{for all } t \in \mathbb{N},
	\end{align}
	where $S_t(X_t,\theta_0)$ and $D_t(X_t,\theta_0)$ are given at \eqref{eqn:DefSMatrix} and \eqref{eqn:DefDMatrix}, assuming that $S_t(X_t,\theta_0)$ is invertible, and $C$ is any deterministic and invertible $q\times q$ matrix.
	Then: 
	\begin{enumerate}[label=\rm(\roman*)]
		\item 
		\label{theorem:GeneralEfficientAPart1}
		The asymptotic covariance matrix of the Z-estimator $\widehat{\theta}_{Z,T, \A^{\ast}}$ is the limit (for $T \to \infty$) of
		\begin{align}\label{eq:Lambda}
		\Lambda_T^{-1} := \left( \frac{1}{T} \sum_{t=1}^T \mathbb{E} \left[ D_t(X_t,\theta_0)^\intercal S_t(X_t,\theta_0)^{-1} D_t(X_t,\theta_0) \right] \right)^{-1}.
		\end{align}	
		
		\item 
		\label{theorem:GeneralEfficientAPart2}
		For any sequence of instrument matrices $\A = (A_t)_{t\in\N}$, and $\Delta_{T,\A}$, $\Sigma_{T,\A}$ as given at \eqref{eqn:AsymptoticCovarianceZEstimator} and \eqref{eqn:AsymptoticCovarianceZEstimator2}, it holds that 
		$\Delta_{T,\A}^{-1} \Sigma_{T,\A} \Delta_{T,\A}^{-1} \succcurlyeq  \Lambda_T^{-1}$  for all $T\ge1$.
		\item 
		\label{theorem:GeneralEfficientAPart3}
		If for some $t\in\{1, \ldots, T\}$ and for any non-singular and deterministic matrix $C$ it holds that 
		$\P\big(A_t(X_t,\theta_0) \not= A_{t,C}^\ast(X_t,\theta_0) \big)>0$,
		then $\Delta_{T,\A}^{-1} \Sigma_{T,\A} \Delta_{T,\A}^{-1} \succcurlyeq  \Lambda_T^{-1}$ and $\Delta_{T,\A}^{-1} \Sigma_{T,\A} \Delta_{T,\A}^{-1} \not=  \Lambda_T^{-1}$.
	\end{enumerate}
\end{theorem}

Parts \ref{theorem:GeneralEfficientAPart1} and \ref{theorem:GeneralEfficientAPart2} of Theorem \ref{theorem:GeneralEfficientA} are direct time series generalizations of the efficiency result of \cite{Hansen1985}, \cite{Chamberlain1987}, and \cite{Newey1993}.
Together, they state that $\Lambda_T^{-1}$ is an asymptotic efficiency bound for the general Z-estimator for semiparametric models and that the Z-estimator based on the choice $A_{t,C}^\ast(X_t,\theta)$ for all $t\in\N$ which fulfills (\ref{eqn::GeneralEfficientAMatrix}) attains this efficiency bound, and is consequently an efficient Z-estimator. 
Thus, parts \ref{theorem:GeneralEfficientAPart1} and \ref{theorem:GeneralEfficientAPart2} of Theorem \ref{theorem:GeneralEfficientA} can be understood as a sufficient condition for efficient semiparametric Z-estimation.

Conversely, part \ref{theorem:GeneralEfficientAPart3} can be interpreted as a necessary condition for efficient estimation and is novel to the literature. 
It states that efficient semiparametric estimation can only be carried out by choosing instrument matrices satisfying  (\ref{eqn::GeneralEfficientAMatrix}) almost surely.
Otherwise, there is some $v\in\R^q$ such that the asymptotic variance of the linear combination $v^\intercal \widehat{\theta}_{Z,T, \A}$ is larger than the asymptotic variance of $v^\intercal \widehat{\theta}_{Z,T, \A^{\ast}}$.
This necessary condition for efficient estimation is crucial for the following sections where we show that for certain functionals, the M-estimator of semiparametric models cannot attain the Z-estimation efficiency bound and consequently neither the semiparametric efficiency bound in the sense of \cite{Stein1956}, which is further discussed in Section \ref{sec:SemipEffBound}

\begin{RunningExmp}
	\label{RunningExmp5}
	For the double quantile model based on the identification function $\varphi$ from Running Example \eqref{RunningExmp3}, Theorem \ref{theorem:GeneralEfficientA} implies that efficient Z-estimation is based on the efficient instrument matrix $A^\ast_{t,C}(X_t,\theta_0) = C D_t(X_t,\theta_0)^\intercal S_t(X_t,\theta_0)^{-1}$, where $C$ is some deterministic and nonsingular matrix and where
	\begin{equation}
		\label{eqn:DQREfficientSMainText}
		S_t(X_t,\theta_0) = 
		\begin{pmatrix}
		\alpha (1-\alpha)  &
		\alpha (1-\beta) \\
		\alpha (1-\beta) &
		\beta (1-\beta)
		\end{pmatrix},  \qquad 
		D_t(X_t,\theta_0) = 
		\begin{pmatrix}
		f_t(q_\alpha(X_t, \theta_0)) \nabla_{\theta} q_\alpha(X_t,\theta_0)^\intercal  \\
		f_t(q_\beta(X_t, \theta_0)) \nabla_{\theta} q_\beta(X_t,\theta_0)^\intercal 
		\end{pmatrix}.
	\end{equation}
	The asymmetric roles of $\alpha$ and $\beta$ in $S_t(X_t,\theta_0) $ stem from the convention that w.l.o.g.\ $\alpha<\beta$.
\end{RunningExmp}

\begin{rem}
	\label{rem:EffChoiceIdFunctio}
	The efficient instrument matrix $A_{t,C}^\ast(X_t,\theta_0)$ in Theorem \ref{theorem:GeneralEfficientA} depends on the specific choice of an identification function $\varphi$
	However, invoking the characterization result \eqref{eq:all identification functions}, if we used a different identification function $\varphi'\big(Y,m(X_t, \theta)\big) = h\big(m(X_t, \theta)\big)\varphi\big(Y,m(X_t, \theta)\big)$ in \eqref{eqn:UnconditionalIdentificationGeneralForm}, where $h$ has full rank, the resulting matrices $S_t(X_t,\theta)$ and $D_t(X_t,\theta)$ in \eqref{eqn:DefSMatrix}, \eqref{eqn:DefDMatrix} would change, but the induced conditional moment conditions \eqref{eqn:UnconditionalIdentificationGeneralForm} would remain unchanged.	
	Hence, the efficiency bound $\Lambda_T^{-1}$ is invariant to the choice of $\varphi$, 
	and Theorem \ref{theorem:GeneralEfficientA} can be interpreted as global, $\varphi$-independent, necessary and sufficient conditions for efficiency.
\end{rem}

\begin{rem}
	\label{rem:GMM}
	While overidentified GMM-estimation
	\begin{align}
		\label{eq:GMM}
		\widehat \theta_{\text{GMM},T} = \argmin_{\theta \in\Theta} \left(\frac{1}{T} \sum_{t=1}^T 
		\psi_t  (Y_t , X_t,\theta)
		\right)^\intercal W_T \left(\frac{1}{T} \sum_{t=1}^T 
		\psi_t  (Y_t , X_t,\theta)
		\right),
	\end{align}
	with  $s$-dimensional ($s > q$) functions $\psi_t$ and a positive definite weighting matrix  $W_T$ can generally improve efficiency compared to Z-estimation \citep{Hansen1982, HallBook2005}, when employing the efficient instrument choice in \eqref{eqn::GeneralEfficientAMatrix}, there is no additional efficiency gain through using overidentifying moment restrictions.
	For details, see e.g., \cite{Newey1993}, and notice that the proofs of Theorem \ref{theorem:GeneralEfficientA}, \ref{theorem:GeneralEfficientAPart1} and \ref{theorem:GeneralEfficientAPart2} work identically when including overidentifying moment restrictions together with a weighting matrix $W_T$ as in \eqref{eq:GMM}.
	Consequently, we restrict attention to efficient instrument Z-estimation in Theorem \ref{theorem:GeneralEfficientA}.
\end{rem}

\section{Semiparametric Models for Vector-Valued Functionals}
\label{sec:ExamplesRegressionFrameworks}


\subsection{Semiparametric Double Quantile Models}
\label{sec:DoubleQuantileRegressionModel}

Consider the double quantile model $m(X_t,\theta) = \big( q_\alpha(X_t,\theta), \,  q_\beta(X_t,\theta) \big)^\intercal$ at fixed levels $0 < \alpha < \beta < 1$ from our Running Examples \eqref{RunningExmp1}--\eqref{RunningExmp5}, whose importance is motivated in the Introduction.
The results of this section hold equivalently for multiple quantiles at different levels.
Let $\widehat \theta_{Z,T,\A}$ be the Z-estimator defined via \eqref{eq:Z-est} and \eqref{eqn:UnconditionalIdentificationGeneralForm} based on some sequence of instrument matrices $\A$ and the strict $\F_{\Y|\X}$-identification function $\ph(y,\xi_1,\xi_2) = \big( \mathds{1}_{\{y \le \xi_1 \}} - \alpha, \; \mathds{1}_{\{y \le \xi_2  \}} - \beta \big)^\intercal$ assuming that all distributions in $\F_{\Y|\X}$ are differentiable at their $\alpha$- and $\beta$-quantiles with strictly positive derivatives.
Recall from Remark \ref{rem:EffChoiceIdFunctio}
that the initial choice of $\ph$ is irrelevant.
The exact form (for this choice of $\ph$) of the efficient instrument matrix $A^\ast_{t,C}(X_t,\theta_0) = C D_t(X_t,\theta_0)^\intercal S_t(X_t,\theta_0)^{-1}$ is given in Running Example \eqref{RunningExmp5}. 

Under Assumptions \eqref{ass:CharMest} and \eqref{ass:GPLclass} in Appendix \ref{sec:AddAssumptions}, \citet[Theorem 1]{DFZ_CharMest} yields that the full class of consistent M-estimators at \eqref{eq:M-est} is given by the class of (strictly) $\F_{\Y|\X}$-consistent loss functions for two quantiles in \eqref{eqn:DoubleQuantileGeneralLossFunctions}.
For any sequence $G= (g_{1,t},g_{2,t})_{t\in\N}$ of such functions, we denote the corresponding M-estimators defined via \eqref{eq:M-est} by $\widehat \theta_{M,T,G}$.


We assume that the M- and Z-estimators are consistent and asymptotically normal.
Primitive conditions for this are discussed before Assumption \eqref{ass:PsiUncorrelated}.
Strict unconditional $\F_\Z$-model consistency of $\rho_t$ for the M-estimator is guaranteed by Theorem \citet[Theorem 1 (i) and (iii)]{DFZ_CharMest} for the strictly consistent losses at \eqref{eqn:DoubleQuantileGeneralLossFunctions}.
For the strict unconditional identification of the Z-estimator, we refer to Proposition \ref{prop:UniqueIdentificationDQR} in Section \ref{app:IDZestDQRModels} which shows strict identification for the efficient Z-estimator in linear models.
While generalizations of these conditions are desirable, their derivation is known to be ``quite difficult'' \cite[p.\ 2127]{NeweyMcFadden1994}.

The following theorem establishes that, under certain conditions, the M-estimator of the double quantile model is subject to the \textit{efficiency gap}, i.e., it cannot attain the Z-estimation efficiency bound, and consequently neither the semiparametric efficiency bound.

\begin{theorem}
	\label{theorem:DoubleQuantileRegressionEfficiencyBound}
	Suppose that Assumptions \eqref{ass:unique model}, \eqref{ass:PsiUncorrelated} together with Assumptions \eqref{ass:CharMest} and \eqref{ass:GPLclass} in Appendix \ref{sec:AddAssumptions} hold for the double quantile model at levels $0<\alpha<\beta<1$, $\widehat \theta_{M,T,G}$ is asymptotically normal and the following further regularity conditions hold:
	\begin{enumerate}[label=\rm{(DQ\arabic*)}, leftmargin=*]
		\item
		\label{cond:DQREfficiencyGapSeparatedParameterModel}
		The parameters of the individual models are separated,
		$m(X_t,\theta) = \big( q_\alpha(X_t,\theta^\alpha),  q_\beta(X_t,\theta^\beta) \big)^\intercal$,
		where $\theta = \big( \theta^\alpha, \theta^\beta \big) \in \interior{\Theta} \subseteq \mathbb{R}^q$, with $\theta^\alpha \in \mathbb{R}^{q_1}$ and $\theta^\beta \in \mathbb{R}^{q_2}$ and $q_1 + q_2 = q$.
		
		\item 
		\label{cond:LinearIndependentImageValues}
		For all $t\in\N$, and for all $A\in \mathcal A$ with $\P(A)=1$ there are $q_1+1$ mutually different $v_1, \ldots, v_{q_1+1} \in \big\{\nabla_{\theta^\alpha} q_\alpha(X_t(\omega),\theta_0^\alpha) \in \R^{q_1} \colon \omega\in A\big\}\subseteq \R^{q_1}$, such that any subset of cardinality $q_1$ of $\{ v_1,\dots,v_{q_1+1} \}$ is linearly independent.
		The analogue assertion holds for the gradient $\nabla_{\theta^\beta}  q_\beta(X_t,\theta_0^\beta)$, replacing $q_1$ by $q_2$.
		\item
		\label{cond:DQREfficiencyGapPositiveDensity}
		For all $t\in\N$, $F_t$ is differentiable at $q_\alpha(X_t,\theta_0^\alpha)$ and $q_\beta(X_t,\theta_0^\beta)$ and the derivatives satisfy 
		$f_t\big(q_\alpha(X_t,\theta_0^\alpha)\big) > 0$ and $f_t\big(q_\beta(X_t,\theta_0^\beta)\big) >0$, and $g'_{1,t}(\xi_1)>0,$ $g'_{2,t}(\xi_2)>0$ for all $\xi_1,\xi_2$. 
	\end{enumerate}
	Then, the following statements hold:
	\begin{enumerate}[label = \rm{(\Alph*)}]
		\item
		\label{statement:DQREfficientChoice}
		Let $\nabla_{\theta^\alpha} q_\alpha(X_t,\theta_0^\alpha) = \nabla_{\theta^\beta} q_\beta(X_t,\theta_0^\beta)$ for all $t\in\N$.
		The M-estimator $\widehat \theta_{M,T,G}$ attains the Z-estimation efficiency bound in \eqref{eq:Lambda} if and only if the following three conditions hold:
		\begin{align}
		\label{eqn:CondEffGapDQRConditions1}
		\exists c_1>0 \ \forall t\in\N: \quad 
		f_t\big(q_\alpha(X_t,\theta_0^\alpha)\big) &= c_1 f_t\big(q_\beta(X_t,\theta_0^\beta)\big)
		\quad \as, \\
		\label{eqn:CondEffGapDQRConditions2}
		\exists c_2>0 \ \forall t\in\N: \quad 
		g_{1,t}'\big(q_\alpha(X_t,\theta_0^\alpha)\big) &= c_2 f_t\big(q_\alpha(X_t,\theta_0^\alpha)\big)
		\quad \as, \\
		\label{eqn:CondEffGapDQRConditions3}
		\exists c_3>0 \ \forall t\in\N: \quad 
		g_{2,t}'\big(q_\beta(X_t,\theta_0^\beta)\big) &= c_3 f_t\big(q_\beta(X_t,\theta_0^\beta)\big)
		\quad \as
		\end{align}		
		
		\item 
		\label{statement:DQREfficiencyGap}
		Furthermore, if \eqref{eqn:CondEffGapDQRConditions2} or \eqref{eqn:CondEffGapDQRConditions3} is violated, then $\widehat \theta_{M,T,G}$
		does not attain the Z-estimation efficiency bound in \eqref{eq:Lambda}.
	\end{enumerate}	
\end{theorem}


A discussion of the conditions of Theorem \ref{theorem:DoubleQuantileRegressionEfficiencyBound}
is in order.
Assumptions \eqref{ass:CharMest} and \eqref{ass:GPLclass} are required to characterize the class of M-estimators; see the previous Running Examples and \citet{DFZ_CharMest} and \citet{FisslerZiegel2016} for a discussion.
The separated parameter condition \ref{cond:DQREfficiencyGapSeparatedParameterModel} contains a large class of possible models. 
E.g., it nests classically used individual quantile models for separate probability levels $\alpha$ and $\beta$. These parameters can also be restricted through inequality relations, e.g., to impede quantile crossings.
While models with joint parameters would also be interesting, completely different methods of proof are required to generalize Theorem \ref{theorem:DoubleQuantileRegressionEfficiencyBound} along these lines.
Our simulation results in Section \ref{sec:DQRSimStudy} indicate that the efficiency gap carries over  to joint parameter models, and is numerically even more severe.

Condition \ref{cond:LinearIndependentImageValues} concerns the variability of the model gradient and is slightly stronger than the classical assumption on univariate models $m$ that the matrix $\mathbb{E} \big[ \nabla_\theta m(X_t,\theta_0)  \nabla_\theta m(X_t,\theta_0)^\intercal \big]$ is of full rank for all $t \in \mathbb{N}$. E.g., consider a linear model with explanatory variable $X_t = (1,V_t)^\intercal$, where $V_t$ attains only $0$ and $1$ with positive probability.
Then, $\mathbb{E} \big[ \nabla_\theta m(X_t,\theta_0)  \nabla_\theta  m(X_t,\theta_0)^\intercal \big] = \mathbb{E} \big[ X_t X_t^\intercal \big]$ is positive definite whereas condition \ref{cond:LinearIndependentImageValues} is not fulfilled.
However, if $V_t$ attains at least three different values with positive probability (or if its distribution is absolutely continuous), \ref{cond:LinearIndependentImageValues} holds.
Condition \ref{cond:DQREfficiencyGapPositiveDensity} is standard for asymptotic normality in quantile regressions.

The gradient condition $\nabla_{\theta^\alpha}  q_\alpha(X_t,\theta_0^\alpha) = \nabla_{\theta^\beta} q_\beta(X_t,\theta_0^\beta)$ in \ref{statement:DQREfficientChoice} is mainly motivated through models that are linear in the parameters, 
where these gradients are simply $X_t$.
In contrast, statement \ref{statement:DQREfficiencyGap} holds independent of this gradient condition for general semiparametric models with separated parameters, but does not provide sufficient conditions for efficient M-estimation.

Section \ref{sec:GapSepModels} shows that the efficiency gap indeed affects the important diagonal entries of the covariance matrix, which is not immediate from Theorem \ref{theorem:DoubleQuantileRegressionEfficiencyBound}.

For the remainder of this subsection, we assume for simplicity that the gradient condition $\nabla_{\theta^\alpha} q_\alpha(X_t,\theta_0^\alpha) = \nabla_{\theta^\beta} q_\beta(X_t,\theta_0^\beta)$ holds, putting us in the situation of \ref{statement:DQREfficientChoice}.
Then, the core condition of this theorem on the underlying process is  (\ref{eqn:CondEffGapDQRConditions1}).
Given that (\ref{eqn:CondEffGapDQRConditions1}) holds, the remaining conditions (\ref{eqn:CondEffGapDQRConditions2}) and (\ref{eqn:CondEffGapDQRConditions3}) are fulfilled by using the obvious choices 
\begin{align}
\label{eqn:DRQChoicesg1g2}
g_{1,t}(\xi_1) = F_t(\xi_1), \qquad \text{ and } \qquad 
g_{2,t}(\xi_2) = F_t(\xi_2), \qquad 
\forall \xi_1,\xi_2 \in \mathbb{R}, \quad  \forall t \in \mathbb{N}.
\end{align}
These conditions coincide with classical efficient semiparametric quantile estimation (for one quantile only) in \cite{Komunjer2010a, Komunjer2010b}.
We refer to (\ref{eqn:DRQChoicesg1g2}) as the \textit{pseudo-efficient} choices as they attain the Z-estimation efficiency bound only in certain situations.

We now analyze the validity of the core condition (\ref{eqn:CondEffGapDQRConditions1}) for double quantile models of the form
\begin{align}
		Y_t = q_\alpha(X_t, \theta_0^\alpha) + u_t^\alpha 
		\qquad \text{and} \qquad
		Y_t = q_\beta(X_t, \theta_0^\beta) + u_t^\beta,
\end{align}
where the two quantile-innovations $(u_t^\alpha)_{t\in\N}$ and $(u_t^\beta)_{t\in\N}$ satisfy the quantile-stationarity conditions $Q_\alpha(u_t^\alpha \,|\, X_t ) = 0$ and $Q_\beta(u_t^\beta \,|\, X_t) = 0$,  such that Assumption \eqref{ass:unique model} is satisfied.
Apart from this assumption, these innovations can be heterogeneously distributed.
Clearly, $u_t^\alpha$ and $u_t^\beta$ are generally dependent. 

Such correctly specified double quantile models can for instance be generated through a process 
\begin{align}
	\label{eqn:GeneralLocScaleDGP}
	Y_t = \zeta(X_t) + \eta(X_t) \varepsilon_t,
\end{align}
for functions $\zeta\colon \R^p\to\R$, $\eta\colon \R^p\to (0,\infty)$,
where the innovations $(\varepsilon_t)_{t\in\N}$ are themselves independent,  independent of $(X_t)_{t\in\N}$, and where $z_\alpha = F_{\varepsilon_t}^{-1}(\alpha)$ and $z_\beta = F_{\varepsilon_t}^{-1}(\beta)$ are time-independent.
Then, the conditional quantiles at level $\alpha \in (0,1)$ (and equivalently for $\beta$) are given by $q_\alpha(X_t, \theta_0^\alpha)  = Q_\alpha(Y_t|X_t) =  \zeta(X_t)  + \eta(X_t) z_\alpha$.
E.g., if $\zeta(X_t)$ and $\eta(X_t)$ are linear in $X_t$, as in the simulation setup in Section \ref{sec:DQRSimStudy}, we also get linear conditional quantile models $q_\alpha(X_t, \theta_0^\alpha)$ and $q_\beta(X_t, \theta_0^\beta)$.
While the process in \eqref{eqn:GeneralLocScaleDGP} resembles the ubiquitous class of location-scale processes, the quantities $\zeta(X_t)$ and $\eta(X_t)$ possibly lose their interpretation as \textit{location} and \textit{scale} for sufficiently heterogeneously distributed innovations $\varepsilon_t$.

For a process in \eqref{eqn:GeneralLocScaleDGP}, the density transformation formula yields that
(\ref{eqn:CondEffGapDQRConditions1}) is equivalent to
\begin{align}
\label{eqn:DQRLocScaleModelRatioDensities}
\frac{	f_t \big( q_\alpha(X_t,\theta_0^\alpha) \big) }{	f_t \big( q_\beta(X_t,\theta_0^\beta) \big) }
= \frac{f_{\varepsilon_t}(z_\alpha)}{f_{\varepsilon_t}(z_\beta)} \stackrel{!}{=} c_1 \qquad  \forall  t \in \mathbb{N}.
\end{align}
This implies that  for processes of the form (\ref{eqn:GeneralLocScaleDGP}), the M-estimator $\widehat \theta_{M,T,G}$ of the double quantile model is able to attain the efficiency bound (based on the choices in (\ref{eqn:DRQChoicesg1g2})), if and only if the density ratio in (\ref{eqn:DQRLocScaleModelRatioDensities}) is constant in $t$.
Consequently, for any i.i.d.\ innovations $(\varepsilon_t)_{t\in\N}$, the M-estimator based on the choices (\ref{eqn:DRQChoicesg1g2}) attains the Z-estimation efficiency bound.

However, one can easily construct examples where condition (\ref{eqn:DQRLocScaleModelRatioDensities}) is violated, e.g., by considering Student's $t$-distributed innovations $\varepsilon_t \sim t_{\nu_t}(\mu_t,\sigma_t^2)$ with time-varying degrees of freedom $\nu_t$, and where the time-varying means and standard deviations are given by
\begin{align}
\label{eqn:DQRChoiceMuSigma}
\mu_t =  Q_\beta(t_{\nu_1}) - \sigma_t Q_\beta(t_{\nu_t}) \qquad \text{ and } \qquad \sigma_t = \frac{Q_\alpha(t_{\nu_1}) - Q_\beta(t_{\nu_1}) }{Q_\alpha(t_{\nu_t}) - Q_\beta(t_{\nu_t}) },
\end{align}
where $t_\nu = t_\nu(0,1)$.
These choices ensure that for $\alpha, \beta \in (0,1)$, $\alpha < \beta$, we have $Q_\alpha \left( t_{\nu_t} \big( \mu_{t}, \sigma_{t}^2 \big) \right) = z_\alpha$ and $Q_\beta \left( t_{\nu_t} \big( \mu_{t}, \sigma_{t}^2 \big) \right) = z_\beta$ for all $t \in \mathbb{N}$, and hence, the quantile-stationarity condition is satisfied while simultaneously condition (\ref{eqn:CondEffGapDQRConditions1}) is violated for all quantile levels such that $\alpha \not= 1-\beta$.

For centered or equal-tailed prediction intervals with $\alpha= 1-\beta < 0.5$, we can choose skewed normally distributed innovations \citep{Azzalini1985}  $\varepsilon_t \sim \mathcal{SN} (\mu_t,\sigma_t^2, \gamma_t)$  with time-varying skewness $\gamma_t$, where the means $\mu_t$ and the standard deviations $\sigma_t$ are given by
\begin{align}
\begin{aligned}
\label{eqn:DQRSkewedNormalChoiceMuSigma}
\mu_t =  Q_\beta \big( \mathcal{SN} (\gamma_1) \big) - \sigma_t Q_\beta \big( \mathcal{SN} (\gamma_t) \big), \qquad
\sigma_t = \frac{Q_\alpha \big( \mathcal{SN} (\gamma_1) \big) - Q_\beta \big( \mathcal{SN} (\gamma_1) \big) }{Q_\alpha \big( \mathcal{SN} (\gamma_t) \big) - Q_\beta \big( \mathcal{SN} (\gamma_t) \big) }, 
\end{aligned}
\end{align}
where $\mathcal{SN} (\gamma_1) := \mathcal{SN} (0,1,\gamma_1)$.
Then, $Q_\alpha \left( \mathcal{SN} (\mu_t, \sigma_t^2, \gamma_t) \right) = z_\alpha$ and $Q_\beta \left( \mathcal{SN} (\mu_t, \sigma_t^2, \gamma_t) \right) = z_\beta$  for all $t \in \mathbb{N}$ and for all $\alpha, \beta \in (0,1)$, $\alpha < \beta$.
We employ these models in the simulations in Section \ref{sec:DQRSimStudy}, where we numerically confirm the theoretical results of this section.

Constructing further processes where the M-estimator cannot attain the Z-estimation efficiency bound can be carried out along these lines, where the crucial condition is that the data generating mechanism must go beyond the class of simple location-scale processes with i.i.d.\ residuals.
Interesting candidates are GAS models of \cite{Creal2013}, and specifically for quantiles, the CAViaR specifications of \cite{Engle2004} and \citet{WhiteKimManganelli2015}.

In summary, there exists an efficiency gap for the double quantile model.
Its presence mainly depends on the underlying process through the key condition in \eqref{eqn:DQRLocScaleModelRatioDensities}. 
The elementary reason for this efficiency gap is the relatively narrow class of strictly consistent loss functions for quantiles at different levels in (\ref{eqn:DoubleQuantileGeneralLossFunctions}), which only consists of the sum of strictly consistent losses for the individual quantiles. 
In particular, this class is much smaller than the corresponding class of strict identification functions; see the Running Example \eqref{exmp:illustration osband} for details.

\subsection{Semiparametric Joint Quantile and ES Models}
\label{sec:QuantileESRegressionModel}

%

Consider a joint model for the quantile (or VaR) and ES at level $\alpha \in (0,1)$, given by $m(X_t,\theta) = \big( q_\alpha(X_t,\theta), e_\alpha(X_t,\theta) \big)^\intercal$, where $q_\alpha(X_t,\theta)$ is a model for the $\alpha$-quantile and $e_\alpha(X_t, \theta)$ denotes a model for the $\ES_\alpha$ at level $\alpha$.
For a random variable $Z$ with quantiles $Q_u(Z)$, the $\ES_\alpha(Z)$ is defined as  $ \frac{1}{\alpha} \int_0^\alpha Q_u(Z) \mathrm{d}u$ that simplifies to $\ES_\alpha(Z) = \mathbb{E} \left[ Z \; | \; Z \le Q_\alpha(Z)  \right]$ if $\mathbb{P}\big(Z \le Q_\alpha(Z)\big)=\alpha$.

As shown by \cite{Gneiting2011} and \cite{Weber2006}, ES is generally neither elicitable nor identifiable and thus,
Theorem 1 (ii) and (iv) \citet{DFZ_CharMest} and Propositions S1 and S3 in its supplementary material provide formal evidence that both M- and Z-estimation of semiparametric models for the conditional ES stand-alone are infeasible.
However, \cite{FisslerZiegel2016} show that under mild conditions, the pair $(Q_\alpha, \ES_\alpha)$ is jointly elicitable and identifiable, and further characterize the class of strictly consistent loss functions.
Due to the recent introduction of ES into the Basel framework as the standard risk measure in banking regulation \citep{Basel2016}, there is a fast-growing interest in semiparametric models for ES (jointly with the quantile) and \cite{Patton2019}, \cite{DimiBayer2019}, \cite{Taylor2019}, \cite{DimiSchnaitmann2019}, \cite{GuillenETAL2021}, among many others, utilize these losses for M-estimation of joint semiparametric models.


Suppose that $\F_{\Y|\X}$ contains only continuous and strictly increasing distribution functions with an integrable lower tail.
Consider the strict $\F_{\Y|\X}$-identification function 
\begin{align}
\label{eqn:QESIdFunction}
\varphi(y, \xi_1, \xi_2) = 
\begin{pmatrix}
\mathds{1}_{\{ y \le \xi_1 \}} - \alpha \\
\xi_2 - \xi_1 + \frac{1}{\alpha} \big( \xi_1 - y \big) \mathds{1}_{\{ y \le \xi_1 \}}
\end{pmatrix},
\end{align}
and define the Z-estimator $\widehat \theta_{Z,T,\A}$  via \eqref{eq:Z-est} and \eqref{eqn:UnconditionalIdentificationGeneralForm} based on some sequence of instrument matrices $\A$.
From Theorem \ref{theorem:GeneralEfficientA}, we get that the efficient estimator has to fulfil the condition $A^\ast_{t,C}(X_t,\theta_0) = C D_t(X_t,\theta_0)^\intercal S_t(X_t,\theta_0)^{-1}$ for some deterministic and nonsingular matrix $C$, where
\begin{align}
\label{eqn:QESREfficientMatrixD}
 D_t(X_t,\theta_0) 
 &= 
\begin{pmatrix}
f_t \big(q_\alpha(X_t, \theta_0) \big) \nabla_\theta q_\alpha(X_t,\theta_0)^\intercal \\
\nabla_\theta e_\alpha(X_t,\theta_0)^\intercal
\end{pmatrix}  \qquad \text{ and } \\[0.5em]
\label{eqn:QESREfficientMatrixS}
S_t(X_t,\theta_0) &= 
\begin{pmatrix}
\alpha (1-\alpha) &
(1-\alpha) \big( q_\alpha(X_t,\theta_0) - e_\alpha(X_t,\theta_0) \big) \\
(1-\alpha) \big( q_\alpha(X_t,\theta_0) - e_\alpha(X_t,\theta_0) \big) &
S_{t,22}
\end{pmatrix},\\[0.5em]
\nonumber
S_{t,22} &= \frac{1}{\alpha} \Var_t \big(  Y_t  \big| Y_t \le q_\alpha(X_t,\theta_0) \big) 
+ \frac{1-\alpha}{\alpha} \big( e_\alpha(X_t,\theta_0) - q_\alpha(X_t,\theta_0) \big)^2\,.
\end{align}

Under Assumptions \eqref{ass:unique model} and \eqref{ass:CharMest}, \citet[Theorem 1]{DFZ_CharMest} shows that M-estimation can be carried out if and only if a (strictly) $\F_{\Y|\X}$-consistent loss functions for the pair $(Q_\alpha, \ES_\alpha)$ is used.
\citet[Theorem 5.2, Corollary 5.5]{FisslerZiegel2016} show that under Assumption \eqref{ass:FZclass}, this whole class is given by
\begin{align} 
\begin{aligned}
\label{eqn:QESRGeneralLossFunctions}
\rho_t(y,\xi_1,\xi_2) = & \left(\mathds{1}_{\{y \le \xi_1\}} -\alpha \right) g_t(\xi_1) - \mathds{1}_{\{y \le \xi_1\}} g_t(y)  + \kappa_t(y)                \\
& + \phi_t'(\xi_2) \left( \xi_2 - \xi_1 + \frac{1}{\alpha} \big( \xi_1 - y \big) \mathds{1}_{\{ y \le \xi_1 \}}  \right) - \phi_t(\xi_2),
\end{aligned}	
\end{align}	
where $\xi_1\mapsto g_t(\xi_1) + \xi_1\phi_t'(\xi_2)/\alpha$ is (strictly) increasing for each $\xi_2$, $\phi_t$ is (strictly) convex
and $\rho_t$ is $\F_{\Y|\X}$-integrable.
For sequences $G = (g_t)_{t\in\N}$ and $\Phi = (\phi_t)_{t\in\N}$ of such functions, we denote the  M-estimator defined via \eqref{eq:M-est} and \eqref{eqn:QESRGeneralLossFunctions} by $\widehat \theta_{M,T,G, \Phi}$.

The following theorem establishes that, under certain conditions, the M-estimator of the joint quantile and ES regression model is subject to the \textit{efficiency gap}.

\begin{theorem}
	\label{theorem:QuantileESRegressionEfficiencyGap}
	Suppose that Assumptions \eqref{ass:unique model}, \eqref{ass:PsiUncorrelated} together with Assumptions \eqref{ass:CharMest} and \eqref{ass:FZclass} in Appendix \ref{sec:AddAssumptions} hold for the joint quantile and ES model at level $\alpha\in(0,1)$, $\widehat \theta_{M,T,G, \Phi}$ is asymptotically normal and the following further regularity conditions hold:
	\begin{enumerate}[label=\rm{(QES\arabic*)}, leftmargin=*]
		\item
		\label{cond:QESREfficiencyGapSeparatedParameterModel}
		The parameters of the individual models are separated,
		$m(X_t,\theta) = \big(q_\alpha(X_t,\theta^q),  e_\alpha(X_t,\theta^e) \big)^\intercal$,
		where $\theta = \big( \theta^q, \theta^e \big) \in \Theta \subseteq \mathbb{R}^q$, with $\theta^q \in \mathbb{R}^{q_1}$ and $\theta^e \in \mathbb{R}^{q_2}$ and $q_1 + q_2 = q$.
		
		\item 
		\label{cond:QESRLinearIndependentImageValues}
		For all $t\in\N$, and for all  $A\in \mathcal A$ with $\P(A)=1$ there are $q_1+1$ mutually different $v_1, \ldots, v_{q_1+1} \in \big\{\nabla_{\theta^q} q_\alpha(X_t(\omega),\theta_0^q) \in \R^{q_1} \colon \omega\in A\big\}\subseteq \R^{q_1}$, such that any subset of cardinality $q_1$ of $\{ v_1,\dots,v_{q_1+1} \}$ is linearly independent.
		The analogue assertion holds for the gradient $\nabla_{\theta^e}  e_\alpha(X_t,\theta_0^e)$, replacing $q_1$ by $q_2$.
		\item 
		\label{cond:QESRPositiveStuff}
		For all $t\in\N$, $F_t$ is differentiable at $q_\alpha(X_t,\theta_0^q)$ with 
		$f_t\big(q_\alpha(X_t,\theta_0^q)\big) > 0$ and 
		$ g_t'(\xi_1) + \phi_t'(\xi_2)/\alpha >0$ and $\phi_t''(\xi_2) >0$ for all $\xi_1, \xi_2$.		
	\end{enumerate}
	Then, the following statements hold:
	\begin{enumerate}[label = \rm{(\Alph*)}]
		\item
		\label{statement:QESREfficientChoice}
		Let $\nabla_{\theta^q} q_\alpha(X_t,\theta_0^q) = \nabla_{\theta^e} e_\alpha(X_t,\theta_0^e)$ for all $t \in \mathbb{N}$.
		The M-estimator $\widehat \theta_{M,T,G, \Phi}$ attains the Z-estimation efficiency bound in \eqref{eq:Lambda} if and only if the following five conditions hold:	
		\begin{align}
		\label{eqn:CondEffGapQESRConditions1}
		\exists c_1>0 \ \forall t\in\N :\quad 
		\operatorname{Var}_t \big(  Y_t  \big| Y_t \le q_\alpha(X_t,\theta^q_0) \big) = c_1 \big( q_\alpha(X_t,\theta^q_0) - e_\alpha(X_t,\theta^e_0) \big)^2\quad \as, \\
		\label{eqn:CondEffGapQESRConditions2}
		\exists c_2>0 \ \forall t\in\N :\quad 
		f_t\big(q_\alpha(X_t,\theta^q_0)\big) = \frac{c_2}{q_\alpha(X_t,\theta^q_0) - e_\alpha(X_t,\theta^e_0)} \quad \as,\\
		\label{eqn:CondEffGapQESRConditions3}
		\exists c_3>0 \ \forall t\in\N :\quad 
		\phi_t''\big(e_\alpha(X_t,\theta^e_0)\big) =\frac{c_3}{\operatorname{Var}_t \big(  Y_t  \big| Y_t \le q_\alpha(X_t,\theta^q_0) \big)} \quad \as, \\
		\label{eqn:CondEffGapQESRConditions4}
		\exists c_4 \in\R \ \forall t\in\N \ \exists c_{5,t} \in \R:\quad 
		g_t'\big(q_\alpha(X_t,\theta^q_0)\big) = c_4 f_t\big(q_\alpha(X_t,\theta^q_0)\big) + c_{5,t}\quad \as, \\
		\label{eqn:CondEffGapQESRConditions5}
		\forall t\in\N :\quad 
		\phi_t'\big(e_\alpha(X_t,\theta^e_0)\big) = \frac{c_3}{c_1 c_2} f_t\big(q_\alpha(X_t,\theta^q_0)\big) - \alpha c_{5,t} \quad \as
		\end{align}	
		%
		\item 
		\label{statement:QESREfficiencyGap}
		Furthermore, if \eqref{eqn:CondEffGapQESRConditions1}, or \eqref{eqn:CondEffGapQESRConditions3}, or 	
		\begin{align}
		\label{eqn:CondEffGapQESRConditionAdditional}
		\exists c_6>0 \ \forall t\in\N:\quad 
		g_t'\big(q_\alpha(X_t,\theta^q_0)\big) + \phi_t'\big(e_\alpha(X_t,\theta^e_0)\big)/\alpha = c_6 f_t\big(q_\alpha(X_t,\theta^q_0)\big) \quad\as
		\end{align}
		is violated, then 
		$\widehat \theta_{M,T,G, \Phi}$ does not attain the Z-estimation efficiency bound in \eqref{eq:Lambda}.	
	\end{enumerate}
\end{theorem}
The general structure of Theorem \ref{theorem:QuantileESRegressionEfficiencyGap}  is similar to Theorem \ref{theorem:DoubleQuantileRegressionEfficiencyBound}: Statement \ref{statement:QESREfficientChoice} provides necessary and sufficient conditions as to when the M-estimation and Z-estimation efficiency bounds coincide, using the additional assumption on the model gradients $\nabla_{\theta^q}  q_\alpha(X_t,\theta_0^q) = \nabla_{\theta^e} e_\alpha(X_t,\theta_0^e)$.
Dispensing with the latter condition, \ref{statement:QESREfficiencyGap} provides necessary conditions only.
Also, the conditions \ref{cond:QESREfficiencyGapSeparatedParameterModel}\,--\,\ref{cond:QESRPositiveStuff}
resemble the conditions  \ref{cond:DQREfficiencyGapSeparatedParameterModel}\,--\,\ref{cond:DQREfficiencyGapPositiveDensity} and are satisfed for a large class of processes and estimators, see the discussion after Theorem \ref{theorem:DoubleQuantileRegressionEfficiencyBound} and in \cite{Patton2019}.

For the remainder for this section, we assume that the gradient condition $\nabla_{\theta^q} q_\alpha(X_t,\theta_0^q) = \nabla_{\theta^e} e_\alpha(X_t,\theta_0^e)$ holds, putting us in the situation of \ref{statement:QESREfficientChoice}.
Then, the core conditions for efficiency of the joint quantile and ES models are given in (\ref{eqn:CondEffGapQESRConditions1})\,--\,(\ref{eqn:CondEffGapQESRConditions5}),
where the conditions (\ref{eqn:CondEffGapQESRConditions1}), (\ref{eqn:CondEffGapQESRConditions2}) only depend on the underlying process and do not involve $g_t$ and $\phi_t$, resembling condition
(\ref{eqn:CondEffGapDQRConditions1}).
These two conditions result from the rather restrictive shape of the class of (strictly) consistent loss functions in (\ref{eqn:QESRGeneralLossFunctions}), see \cite{FisslerZiegel2016} for details.
Section \ref{sec:QESModelExamples} further analyzes the validity of (\ref{eqn:CondEffGapQESRConditions1}),  (\ref{eqn:CondEffGapQESRConditions2}) with results resembling the ones for double quantile models from the previous section.
Given that these conditions hold, efficient M-estimation can be performed by employing suitable choices of $g_t$ and $\phi_t$ satisfying \eqref{eqn:CondEffGapQESRConditions3}\,--\,\eqref{eqn:CondEffGapQESRConditions5}, which are further discussed in Section \ref{sec:QESREfficientEstimation} and which resemble conditions (\ref{eqn:CondEffGapDQRConditions2}) and (\ref{eqn:CondEffGapDQRConditions3}).

Conditions (\ref{eqn:CondEffGapQESRConditions1})\,--\,(\ref{eqn:CondEffGapQESRConditions5}) and \eqref{eqn:CondEffGapQESRConditionAdditional} illustrate the concordance with mean and quantile regression models.
Condition \eqref{eqn:CondEffGapQESRConditionAdditional} (which can be split into \eqref{eqn:CondEffGapQESRConditions4} and \eqref{eqn:CondEffGapQESRConditions5} under the equality of the model gradients) is closely related to the efficient choice for semiparametric quantile models, see \cite{Komunjer2010a, Komunjer2010b}, and Section \ref{sec:DoubleQuantileRegressionModel} of this article.
However, in contrast to classical quantile regression, it is important to notice that given (\ref{eqn:CondEffGapQESRConditions1}) and (\ref{eqn:CondEffGapQESRConditions2}) hold, the choice $g_t(z) = 0$ (resulting from $c_4=0$ and $c_{5,t}=0$)  facilitates efficient estimation through a suitable choice of the function $\phi_t$.
Moreover, condition (\ref{eqn:CondEffGapQESRConditions3}) resembles the classical condition of efficient least squares estimation of \cite{GourierouxMonfortTrognon1984}, where the second derivative of $\phi_t$ is proportional to the reciprocal of the conditional variance.
As ES is a \textit{tail} expectation, one also needs to consider the \emph{tail variance} in (\ref{eqn:CondEffGapQESRConditions3}).


\cite{Barendse2020} considers two-step estimation and a related \textit{two-step efficiency bound} for semiparametric quantile and ES models that we discuss and relate to our results 
in Section \ref{sec:TwoStepEfficiency}.

\subsubsection{Efficient Estimation of Joint Semiparametric Quantile and ES Models}
\label{sec:QESREfficientEstimation}

%

Here, we discuss feasible choices for $g_t$ and $\phi_t$ satisfying (\ref{eqn:CondEffGapQESRConditions3})\,--\,(\ref{eqn:CondEffGapQESRConditions5}) and \ref{cond:QESRPositiveStuff} to facilitate efficient M-estimation for semiparametric joint quantile and ES models based on Theorem \ref{theorem:QuantileESRegressionEfficiencyGap}.
To this end, we assume that (\ref{eqn:CondEffGapQESRConditions1}) and (\ref{eqn:CondEffGapQESRConditions2}) hold for the underlying process and defer a discussion of these conditions to Section \ref{sec:QESModelExamples}.	
An obvious solution satisfying (\ref{eqn:CondEffGapQESRConditions3})\,--\,(\ref{eqn:CondEffGapQESRConditions5}) is
\begin{alignat}{3}
\begin{aligned}
\label{eqn:QESREfficientChoice1}
g_t^\text{eff1}(\xi_1) &= d_1 F_t(\xi_1), \qquad & \text{for } \xi_1 > e_\alpha(X_t, \theta^e_0), \\
\phi_t^\text{eff1}(\xi_2)  &= - d_2\log \big( q_\alpha(X_t, \theta^q_0) - \xi_2 \big) \qquad &  \text{for } \xi_2 < q_\alpha(X_t, \theta^q_0),
\end{aligned}
\end{alignat}
for all $t\in \mathbb{N}$ and for some constants $d_1 \ge 0$ and $d_2 > 0$, which we refer to as the \textit{first pseudo-efficient} choices. 
Motivated by the condition 
\begin{align}
\label{eqn:Phi''Condition}
\phi_t''(e_\alpha(X_t,\theta^e_0))
= c\,\Big(\operatorname{Var}_t\big(  Y_t  \big| Y_t \le q_\alpha(X_t,\theta^q_0) \big)  + (1-\alpha) \big( e_\alpha(X_t,\theta^e_0) - q_\alpha(X_t,\theta^q_0) \big)^{2}\Big)^{-1}
\end{align}
for some $c >0$, given in (\ref{QESREfficiencyGapEquation22}) in the proof of Theorem \ref{theorem:QuantileESRegressionEfficiencyGap} and in the two-step efficiency bound of \cite{Barendse2020}, a \emph{second pseudo-efficient} choice, satisfying (\ref{eqn:CondEffGapQESRConditions3})\,--\,(\ref{eqn:CondEffGapQESRConditions5}), is given by
\begin{align}
\begin{aligned}
\label{eqn:QESREfficientChoice2}
g_t^\text{eff2}(\xi_1) &= 0, \qquad \text{ and } \\
\phi_t^\text{eff2}(\xi_2)  
&=  \frac{ d_3(q_t - \xi_2) }{\sqrt{(1-\alpha)v_t}}  \arctan
\left( \frac{\sqrt{1-\alpha}(q_t - \xi_2)}{\sqrt{v_t}} \right)
+\xi_2\frac{\pi d_3(1+d_4)}{2\sqrt{(1-\alpha)v_t}} \\
& - \frac{d_3}{2(1-\alpha)} \log \big( v_t + (1-\alpha)(q_t - \xi_2)^2 \big), \quad \text{for all} \quad \xi_2 < q_t,
\end{aligned}
\end{align}
for constants $d_3 > 0$, $d_4\ge0$, where $v_t = \operatorname{Var}_t \big(  Y_t  \big| Y_t \le q_\alpha(X_t,\theta^q_0) \big)$ and $q_t = q_\alpha(X_t, \theta^q_0)$.
Then,
\[
{\phi_t'}^\text{eff2}(\xi_2)  
= -\frac{d_3}{\sqrt{(1-\alpha)v_t}} \arctan \left( \frac{\sqrt{1-\alpha}(q_t - \xi_2) }{ \sqrt{v_t}} \right)
+\frac{\pi d_3(1+d_4)}{2\sqrt{(1-\alpha)v_t}}>0,
\] and 
${\phi_t''}^\text{eff2}(\xi_2)  
= {d_3}  \big(v_t + (1-\alpha)(q_t - \xi_2)^2\big)^{-1}>0$, 
for all  $\xi_2 < q_t$.

This illustrates that, given that (\ref{eqn:CondEffGapQESRConditions1}) and (\ref{eqn:CondEffGapQESRConditions2}) hold, there exist different efficient M-estimators.
Furthermore, if (\ref{eqn:CondEffGapQESRConditions1}) and  (\ref{eqn:CondEffGapQESRConditions2}) do not hold jointly, Theorem \ref{theorem:QuantileESRegressionEfficiencyGap} \ref{statement:QESREfficientChoice} cannot be employed for a statement on efficiency of different M-estimators and it is generally unclear which choices of $g_t$ and $\phi_t$ result in the most efficient estimator.
We analyze this numerically for location-scale process with heteroskedastic innovations in the simulation study in Section \ref{sec:QESRSimStudy}.
The results there also suggest that there is an efficiency gap in models with joint parameters.

As it is common for efficient semiparametric estimation (cf.\ \citealp{GourierouxMonfortTrognon1984}, \citealp{Komunjer2010a, Komunjer2010b}), the efficient choice depends on the knowledge of the true parameter vector $\theta_0$ and further unknown quantities such as the 
conditional density $f_t$ evaluated at $q_\alpha(X_t,\theta^q_0)$ or the quantile-truncated variance $\operatorname{Var}_t \big( Y_t \big| Y_t \le q_\alpha(X_t,\theta^q_0) \big)$.
In practice, one usually applies a two-step estimation approach where the unknown quantities in the efficient choices are substituted by consistent estimates.
Notably, the pseudo-efficient M-estimators based on the first choices $g_t(\xi_1)=0$ and $\phi_t(\xi_2)$ in (\ref{eqn:QESREfficientChoice1}) are remarkably feasible in the sense that they only require a first-step estimate of the quantile-specific parameters.
This is considerably easier than the required nonparametric first-step estimators of the conditional variance or the conditional distribution function in efficient M-estimation of mean and quantile regressions. 

A further interesting fact arises from a comparison of (\ref{eqn:QESREfficientChoice1}) to the predominantly used loss functions with homogeneous loss differences of degree zero \citep{NoldeZiegel2017}, given by
\begin{align}
\label{eqn:QESRZeroHomogeneousLoss}
g_t(\xi_1) = 0
\qquad \text{ and } \qquad
\phi_t(\xi_2) = - \log(-\xi_2), \quad \text{for } \xi_2<0.
\end{align}
\cite{Patton2019} build their M-estimation approach on these choices and \cite{DimiBayer2019} numerically show that such M-estimators are relatively efficient.

Comparing the choice $g_t(\xi_1) = 0$ to the efficient choice in (\ref{eqn:QESREfficientChoice1}) illustrates the elegance of the parsimonious choice $d_1=0$.
By further comparing the choices of $\phi_t$ in  (\ref{eqn:QESREfficientChoice1}) and (\ref{eqn:QESRZeroHomogeneousLoss}), we see that the zero-homogeneous loss function only deviates from the pseudo-efficient choice in (\ref{eqn:QESREfficientChoice1}) through the translation by $q_\alpha(X_t, \theta^q_0)$.
This justifies the choice of \cite{Patton2019} ex post and theoretically explains the good numerical performance observed by \cite{BayerDimi2019}.
While the zero-homogeneous choice requires strictly negative values for the conditional ES, employing the closely related efficient choice in (\ref{eqn:QESREfficientChoice1}) makes this condition redundant and instead, we only have to impose the natural condition that the conditional ES is smaller than the conditional quantile. 
	Interestingly, when $d_1=0$, 
	\eqref{eqn:QESREfficientChoice1} also constitutes a strictly consistent loss with zero-homogeneous loss differences, when allowing the (itself 1-homogenous) quantile as an input parameter.
	This does not contradict \cite{NoldeZiegel2017} as they naturally do not allow the true quantile as an input parameter.

\subsubsection{Processes Generating an Efficiency Gap in joint Quantile and ES Models}
\label{sec:QESModelExamples}

In this section, we discuss attainability of the process conditions (\ref{eqn:CondEffGapQESRConditions1}) and (\ref{eqn:CondEffGapQESRConditions2}), which are necessary for the M-estimator to match the Z-estimation efficiency bound under the gradient condition $\nabla_{\theta^q} q_\alpha(X_t,\theta_0^q) = \nabla_{\theta^e} e_\alpha(X_t,\theta_0^e)$.
We consider joint quantile and ES models of the form
\begin{align}
Y_t = q_\alpha(X_t, \theta_0^q) + u_t^q 
\qquad \text{and} \qquad
Y_t = e_\alpha(X_t, \theta_0^e) + u_t^e,
\end{align}
where the innovations $(u_t^q)_{t\in\N}$ and $(u_t^e)_{t\in\N}$ satisfy the semiparametric stationarity conditions $Q_\alpha(u_t^q \,|\, X_t ) = 0$ and $\ES_\alpha(u_t^e \,|\, X_t) = 0$,  such that Assumption \eqref{ass:unique model} is satisfied.

Such correctly specified models can for instance be generated through the process	in \eqref{eqn:GeneralLocScaleDGP}, where we---slightly differently from the residual assumption in \eqref{eqn:GeneralLocScaleDGP}---impose that $z_\alpha = F_{\varepsilon_t}^{-1}(\alpha)$ and $s_\alpha = \ES_\alpha(\varepsilon_t)$ are time-independent, such that Assumption  \eqref{ass:unique model} holds. 
Apart from that, the innovations may be heterogeneously distributed.
We then get that $Q_\alpha(Y_t\,|\,X_t) = \zeta(X_t)+ \eta(X_t) z_\alpha$ and $\ES_\alpha(Y_t\,|\,X_t) = \zeta(X_t) + \eta(X_t) s_\alpha$.
E.g., if $\zeta(X_t)$ and $\eta(X_t)$ are linear in $X_t$, as in the simulation setup in Section \ref{sec:QESRSimStudy}, we also get linear models for $q_\alpha(X_t, \theta_0^q)$ and $e_\alpha(X_t, \theta_0^e)$.

It further holds that
$\operatorname{Var}_t \big( Y_t \,|\, Y_t \le q_\alpha(X_t,\theta^q_0) \big)
= \eta(X_t)^2 \operatorname{Var}_t \big( \varepsilon_t \,|\, \varepsilon_t \le z_\alpha \big)$,
and
$f_t \big( q_\alpha(X_t,\theta_0^q) \big)  
=  f_{\varepsilon_t}(z_\alpha) / \eta(X_t)
=  f_{\varepsilon_t}(z_\alpha)(z_\alpha - s_\alpha)/\big(q_\alpha(X_t,\theta^q_0) - e_\alpha(X_t,\theta^e_0) \big)$.
Thus, for stationary innovations $(\varepsilon_t)_{t\in\N}$, 
the quantities $\operatorname{Var}_t \big( \varepsilon_t  \,|\, \varepsilon_t \le z_\alpha \big)$ and  $f_{\varepsilon_t}(z_\alpha)$ are constant, 
which implies that the conditions (\ref{eqn:CondEffGapQESRConditions1}) and (\ref{eqn:CondEffGapQESRConditions2}) are satisfied, and hence, any M-estimator based on choices for $g_t$ and $\phi_t$ satisfying (\ref{eqn:CondEffGapQESRConditions3})\,--\,(\ref{eqn:CondEffGapQESRConditions5}) attains the Z-estimation efficiency bound.

Similarly to Section \ref{sec:DoubleQuantileRegressionModel}, we can easily construct processes which generate an \textit{efficiency gap} by considering time-varying innovation distributions.
E.g.,  we consider independent and Student's $t$-distributed innovations $\varepsilon_t \sim t_{\nu_t}(\mu_t,\sigma_t^2)$ with time-varying degrees of freedom $\nu_t$ and
\begin{align}
\label{eqn:QESRMutSigmat}
\mu_t =  Q_\alpha(t_{\nu_1}) - \sigma_t Q_\alpha(t_{\nu_t})
\qquad \text{ and }\qquad
\sigma_t  = \frac{Q_\alpha(t_{\nu_1}) - \ES_\alpha(t_{\nu_1}) }{Q_\alpha(t_{\nu_t}) - \ES_\alpha(t_{\nu_t}) }.
\end{align}
The conditions in \eqref{eqn:QESRMutSigmat} are such that the quantile-ES stationarity condition is satisfied.
For this process, it still holds that $\operatorname{Var}_t \big( Y_t\,|\,Y_t \le q_\alpha(X_t,\theta_0) \big)  = \eta(X_t)^2 \operatorname{Var} \big( \varepsilon_t\,|\,\varepsilon_t \le z_\alpha \big)$,
as $\varepsilon_t$ is independent of $X_t$.
However, the quantity $\operatorname{Var} \big( \varepsilon_t\,|\,\varepsilon_t \le z_\alpha\big)$ is generally time-varying, and consequently, this violates (\ref{eqn:CondEffGapQESRConditions1}) and hence generates an efficiency gap.

\section{Numerical Illustration of the Efficiency Gap}
\label{sec:Simulation}

In this section, we numerically illustrate the efficiency gap for double quantile and joint quantile and ES models by approximating the expectations (over the covariates) in (\ref{eqn:AsymptoticNormalityZEstimator})\,--\,(\ref{eqn:AsymptoticCovarianceZEstimator2}) in simulations.
We use 1000 simulation replications each consisting of a sample size of $T=2000$.

\subsection{Double Quantile Models}
\label{sec:DQRSimStudy}

For the double quantile models, we simulate according to the process in (\ref{eqn:GeneralLocScaleDGP}), where
$X_t \stackrel{\textrm{iid}}{\sim}   3 \operatorname{Beta}(3, 1.5)$, $\zeta(X_t) = 10 + 0.5  X_t$, and  $\eta(X_t) = 0.5 + 0.5  X_t$.
For the model innovations $\varepsilon_t$, we choose the following three different specifications:
(a) $\varepsilon_t \stackrel{iid}{\sim}  \mathcal{N}(0,1)$;
(b) $\varepsilon_t \sim t_{\nu_t}(\mu_t,\sigma_t^2)$ with time-varying degrees of freedom, $\nu_t = 3 \times \mathds{1}_{\{t \le T/2\}} + 100 \, \mathds{1}_{\{t > T/2\}}$, where $\mu_t$ and $\sigma_t$ are given in (\ref{eqn:DQRChoiceMuSigma}); and
(c) $\varepsilon_t \sim \mathcal{SN} (\mu_t,\sigma_t^2, \gamma_t)$ follows a skewed normal distribution with time-varying skewness, $\gamma_t = 0.9 \, \mathds{1}_{\{t > T/2\}}$, where $\mu_t$ and $\sigma_t$ are given in (\ref{eqn:DQRSkewedNormalChoiceMuSigma}).
%
%

These choices are motivated through the theoretical considerations of Section \ref{sec:DoubleQuantileRegressionModel}  that for models of the form \eqref{eqn:GeneralLocScaleDGP} with i.i.d.\ residuals, the M-estimator is able to attain the Z-estimation efficiency bound, while conversely, it cannot do so for heterogeneously distributed innovations.
The heterogeneously skewed process in (c) is motivated by symmetric prediction intervals where $\alpha = 1- \beta$.
Empirically, scenario (b) (and similarly for (c)) can be motivated by a breakpoint model for the degree of heavy-tailedness of the innovations: A period of stress (first part of the sample) exhibiting heavy tails is followed by a relatively calm period (second part of the sample), which is resembled by an innovation-distribution with considerably lighter tails.

For the considered processes, it holds that $Q_\alpha(Y_t|X_t)=  (10 + 0.5 z_\alpha) + (0.5  + 0.5z_\alpha) X_t$, and $Q_\beta(Y_t|X_t)=  (10 + 0.5 z_\beta) + (0.5  + 0.5z_\beta) X_t$.
We estimate linear models with \textit{separated} parameters,
\begin{align}
\label{eqn:DQR_Models_SepParameters}
q_\alpha(X_t,\theta) =  \theta^{(1)} + \theta^{(2)} X_t,
\qquad \text{and}  \qquad 
q_\beta(X_t,\theta) =  \theta^{(3)} + \theta^{(4)} X_t.
\end{align}
In order to consider models with \textit{joint} parameters, we use a slightly modified parametrization of the process by using $\eta(X_t) = 0.5  X_t$, which implies that $Q_\alpha(Y_t|X_t)=  10  + (0.5  + 0.5z_\alpha) X_t$, and $Q_\beta(Y_t|X_t)=  10 + (0.5  + 0.5z_\beta ) X_t$.
Hence, we use the (correctly specified) \textit{joint intercept models}
\begin{align}
\label{eqn:DQR_Models_JointParameters}
q_\alpha(X_t,\theta) =  \theta^{(1)} + \theta^{(2)} X_t,
\qquad \text{and}  \qquad 
q_\beta(X_t,\theta) =  \theta^{(1)} + \theta^{(3)} X_t.
\end{align}

Table \ref{tab:DQRSimSD} reports the relative standard deviations of the estimated parameters normalized by the corresponding efficiency bound. The row denoted ``Efficiency Bound'' reports the raw standard deviation.
We consider the probability levels $(\alpha, \beta) \in \big\{ (1\%, 2.5\%), (5\%, 95\%), (1\%, 90\%) \big\}$, where the first choice is important for VaR modeling in risk management, while the remaining two consider estimation of a symmetric and an asymmetric prediction interval.
Panels A-C consider the separated parameter models in (\ref{eqn:DQR_Models_SepParameters}) while Panels D-F consider models with joint parameters in (\ref{eqn:DQR_Models_JointParameters}).
We show results for the joint M-estimator using the general loss function in (\ref{eqn:DoubleQuantileGeneralLossFunctions}) paired with the choices $g_{t}(\xi)  = g_{1,t}(\xi)  = g_{2,t}(\xi)$ given in the first column of Table  \ref{tab:DQRSimSD} together with the Z-estimation efficiency bound.
$F_{\text{Log}}$ denotes the distribution function of a standard logistic distribution.
Tables \ref{tab:DQRSimSDTrue} and \ref{tab:DQRSim_Joint_InterceptSDTrue} show results for additional probability levels.

\begin{table}[tb!]
	\caption{Relative Asymptotic Standard Deviations of Double Quantile Models}
	\label{tab:DQRSimSD}
	\tiny
	\centering
	\begin{tabularx}{\linewidth}{X @{\hspace{0.2cm}} lrrrr @{\hspace{0.2cm}} lrrrr @{\hspace{0.2cm}} lrrrr} 
		\addlinespace
		\toprule
		& & \multicolumn{4}{c}{ (a) Homoskedastic } && \multicolumn{4}{c}{ (b) Heteroskedastic $t$ }  && \multicolumn{4}{c}{ (c) Heteroskedastic $\mathcal{SN}$ }    \\
		\cmidrule(lr){3-6} \cmidrule(lr){8-11} \cmidrule(lr){13-16} 
		$g_t(\xi)$ & & $\theta_1$ & $\theta_2$ & $\theta_3$ & $\theta_4$ & &  $\theta_1$ & $\theta_2$ & $\theta_3$ & $\theta_4$ & & $\theta_1$ & $\theta_2$ & $\theta_3$ & $\theta_4$ \\
		\midrule
		\\
		& &\multicolumn{14}{c}{Panel A: Separated Model Parameters and $(\alpha, \beta) =  (1\%, 2.5\%)$} \\
		\cmidrule(lr){3-16}
		$\xi$ &   &  1.044 &  1.042 &  1.044 &  1.042 &   &  1.075 &  1.094 &  1.057 &  1.063 &   &  1.121 &  1.137 &  1.070 &  1.079\\
		$\exp(\xi)$ &   &  1.010 &  1.016 &  1.001 &  1.002 &   &  1.740 &  2.117 &  1.254 &  1.333 &   &  1.312 &  1.361 &  1.136 &  1.152\\
		$\log(\xi)$ &   &  1.063 &  1.061 &  1.057 &  1.054 &   &  1.102 &  1.128 &  1.070 &  1.078 &   &  1.112 &  1.129 &  1.069 &  1.079\\
		$F_{\text{Log}}(\xi)$ &   &  1.044 &  1.042 &  1.044 &  1.042 &   &  1.075 &  1.095 &  1.056 &  1.063 &   &  1.121 &  1.138 &  1.070 &  1.079\\
		$F_t(\xi)$ &   &  1.000 &  1.000 &  1.000 &  1.000 &   &  1.016 &  1.016 &  1.016 &  1.016 &   &  1.011 &  1.012 &  1.012 &  1.012\\
		\cmidrule(lr){3-16}
		Efficiency Bound  &   & 13.619 &  7.546 &  9.745 &  5.399 &   & 47.871 & 22.289 & 29.466 & 14.409 &   & 10.686 &  4.881 &  9.259 &  4.471\\
		\midrule
		\\
		& &\multicolumn{14}{c}{Panel B: Separated Model Parameters and $(\alpha, \beta) = (5\%, 95\%)$} \\
		\cmidrule(lr){3-16}
		$\xi$ &   &  1.043 &  1.042 &  1.043 &  1.042 &   &  1.070 &  1.068 &  1.070 &  1.068 &   &  1.216 &  1.169 &  1.050 &  1.046\\
		$\exp(\xi)$ &   &  1.004 &  1.003 &  1.837 &  1.681 &   &  1.035 &  1.041 &  2.298 &  2.036 &   &  1.126 &  1.091 &  1.759 &  1.631\\
		$\log(\xi)$ &   &  1.051 &  1.049 &  1.025 &  1.024 &   &  1.090 &  1.087 &  1.048 &  1.047 &   &  1.230 &  1.182 &  1.036 &  1.032\\
		$F_{\text{Log}}(\xi)$  &   &  1.043 &  1.042 &  1.043 &  1.042 &   &  1.070 &  1.068 &  1.070 &  1.068 &   &  1.216 &  1.169 &  1.050 &  1.046\\
		$F_t(\xi)$ &   &  1.000 &  1.000 &  1.000 &  1.000 &   &  1.000 &  1.000 &  1.000 &  1.000 &   &  1.000 &  1.000 &  1.001 &  1.000\\
		\cmidrule(lr){3-16}
		Efficiency Bound  &   &  7.672 &  4.271 &  7.672 &  4.271 &   & 13.706 &  7.507 & 13.706 &  7.507 &   &  4.317 &  2.750 &  9.389 &  5.075\\
		\midrule
		\\
		& &\multicolumn{14}{c}{Panel C: Separated Model Parameters and $(\alpha, \beta) = (1\%, 90\%)$} \\
		\cmidrule(lr){3-16}
		$\xi$ &   &  1.044 &  1.042 &  1.044 &  1.042 &   &  1.209 &  1.174 &  1.030 &  1.030 &   &  1.599 &  1.383 &  1.068 &  1.062\\
		$\exp(\xi)$ &   &  1.010 &  1.016 &  1.657 &  1.547 &   &  1.305 &  1.434 &  1.722 &  1.605 &   &  1.260 &  1.144 &  1.648 &  1.549\\
		$\log(\xi)$  &   &  1.063 &  1.061 &  1.027 &  1.026 &   &  1.388 &  1.331 &  1.016 &  1.016 &   &  1.669 &  1.430 &  1.054 &  1.048\\
		$F_{\text{Log}}(\xi)$ &   &  1.044 &  1.042 &  1.044 &  1.042 &   &  1.202 &  1.167 &  1.030 &  1.030 &   &  1.599 &  1.383 &  1.068 &  1.062\\
		$F_t(\xi)$ &   &  1.000 &  1.000 &  1.000 &  1.000 &   &  1.000 &  1.000 &  1.000 &  1.000 &   &  1.000 &  1.000 &  1.000 &  1.000\\
		\cmidrule(lr){3-16}
		Efficiency Bound  &   & 13.650 &  7.579 &  6.250 &  3.471 &   & 31.097 & 19.012 & 10.885 &  5.856 &   &  5.556 &  4.487 &  7.893 &  4.261\\
		\midrule
		\midrule
		\\
		\\
		& &\multicolumn{14}{c}{Panel D: Joint Model Parameters and $(\alpha, \beta) =  (1\%, 2.5\%)$} \\
		\cmidrule(lr){3-16}
		$\xi$ &   &  1.330 &  1.222 &  1.236 &&   &  1.714 &  1.391 &  1.483 & &  &  1.570 &  1.394 &  1.404\\
		$\exp(\xi)$  &   &  1.122 &  1.062 &  1.080 & &  &  1.517 &  1.452 &  1.437 & &  &  1.475 &  1.345 &  1.346\\
		$\log(\xi)$ &   &  1.366 &  1.249 &  1.261 &&   &  1.812 &  1.472 &  1.556 &&   &  1.589 &  1.409 &  1.419\\
		$F_{\text{Log}}(\xi)$  &   &  1.330 &  1.222 &  1.235 &&   &  1.714 &  1.391 &  1.483 &&   &  1.570 &  1.394 &  1.404\\
		$F_t(\xi)$ &   &  1.003 &  1.002 &  1.002 &&   &  1.047 &  1.029 &  1.035 &&   &  1.003 &  1.005 &  1.005\\
		\cmidrule(lr){3-16}
		Efficiency Bound  &   &  3.757 &  2.880 &  2.564 &&   & 14.649 &  9.672 &  8.459 &&   &  4.991 &  2.690 &  2.673\\
		\midrule
		\\
		& &\multicolumn{14}{c}{Panel E: Joint Model Parameters and $(\alpha, \beta) =  (5\%, 95\%)$} \\
		\cmidrule(lr){3-16}
		$\xi$  &   & 1.300 & 1.190 & 1.190 & &  & 1.303 & 1.204 & 1.204 &&   & 1.562 & 1.362 & 1.228\\
		$\exp(\xi)$ &   & 3.475 & 2.523 & 2.537 & &  & 4.675 & 3.635 & 3.217 &  & & 5.973 & 4.223 & 3.083\\
		$\log(\xi)$  &   & 1.292 & 1.198 & 1.170 & &  & 1.312 & 1.230 & 1.188 &&   & 1.513 & 1.349 & 1.191\\
		$F_{\text{Log}}(\xi)$ &   & 1.300 & 1.190 & 1.190 & &  & 1.303 & 1.204 & 1.204 &  & & 1.562 & 1.362 & 1.228\\
		$F_t(\xi)$ &   & 1.000 & 1.000 & 1.000 & &  & 1.000 & 1.000 & 1.000 & &  & 1.001 & 1.001 & 1.001\\
		\cmidrule(lr){3-16}
		Efficiency Bound &   & 2.173 & 1.648 & 1.648 &  & & 3.841 & 2.903 & 2.903 & &  & 1.339 & 1.098 & 1.465\\
		\midrule
		\\
		& &\multicolumn{14}{c}{Panel F: Joint Model Parameters and $(\alpha, \beta) =  (1\%, 90\%)$} \\
		\cmidrule(lr){3-16}
		$\xi$ &   & 1.269 & 1.108 & 1.200 &&   & 2.329 & 1.372 & 1.666 &&   & 2.445 & 1.726 & 1.560\\
		$\exp(\xi)$ &   & 2.593 & 1.689 & 2.084 &&   & 5.484 & 3.164 & 3.334 &&   & 6.719 & 4.224 & 3.315\\
		$\log(\xi)$  &   & 1.231 & 1.114 & 1.167 &&   & 2.331 & 1.484 & 1.648 &&   & 2.296 & 1.706 & 1.486\\
		$F_{\text{Log}}(\xi)$ &   & 1.269 & 1.108 & 1.200 &&   & 2.328 & 1.371 & 1.666 & &  & 2.445 & 1.726 & 1.560\\
		$F_t(\xi)$ &   & 1.065 & 1.022 & 1.046 &&   & 1.577 & 1.392 & 1.259 &&   & 1.638 & 1.535 & 1.195\\
		\cmidrule(lr){3-16}
		Efficiency Bound  &   & 2.272 & 2.292 & 1.578 &&   & 2.301 & 4.575 & 1.960 &&   & 1.102 & 1.310 & 1.225\\
		\bottomrule
		\addlinespace
		\multicolumn{16}{p{.97\linewidth}}{\scriptsize 
			This table presents the (approximated) asymptotic standard deviations for semiparametric double quantile models at different probability levels in the horizontal panels.
			The rows titled ``Efficiency Bound'' report the raw standard deviations whereas the remaining rows report the relative standard deviations compared to the efficiency bound.
			Panels A-C report results for the models with separated parameters given in 	(\ref{eqn:DQR_Models_SepParameters}) while Panels D-F considers the joint intercept models given in (\ref{eqn:DQR_Models_JointParameters}).
			Results for the three residual distributions described in Section \ref{sec:DQRSimStudy} are reported in the three vertical panels of the table.
			We furthermore consider four classical choices of $g_t(\xi)$ together with the (pseudo-) efficient choice $F_t(\xi)$ and the Z-estimation efficiency bound.
		}
	\end{tabularx}
\end{table}

The numerical results generally confirm the conclusions of Section \ref{sec:DoubleQuantileRegressionModel}: the pseudo-efficient M-estimator with $g_{t}(\xi) = F_t(\xi)$ attains the efficiency bound for homoskedastic innovation distributions, while it cannot attain the efficiency bound for both heteroskedastic processes.
Furthermore, as discussed in Section \ref{sec:DoubleQuantileRegressionModel}, for symmetric quantile levels as in Panel B, the symmetrically heteroskedastic process in (b) is not sufficient for generating an efficiency gap, whereas the asymmetric process in (c) is sufficient.
	The latter claim can be seen by the slightly larger standard deviation of $\theta_3$ in Panel B (c), which is not a numerical artifact as it is supported by our theory in Section \ref{sec:DoubleQuantileRegressionModel}.
Remarkably, even for models with separated parameters, where the pseudo-efficient choices are efficient estimators for the individual quantile models \citep{Komunjer2010a, Komunjer2010b}, the corresponding joint M-estimator does not attain the (joint) efficiency bound for the processes with heteroskedastic innovations, see Panel A and Section \ref{sec:GapSepModels}.

We observe that the gap becomes numerically larger for quantile levels in the tails of the conditional distributions (Panel A) and for quantile levels which are close together.
This makes it particularly relevant in the VaR literature, where VaR is often reported for multiple, extreme quantile levels. 
The first observation can be explained by condition (\ref{eqn:CondEffGapDQRConditions1}). For a numerically large efficiency gap, one requires heterogeneity of the conditional densities at the respective quantiles $f_t\big(q_\alpha(X_t,\theta_0^\alpha)\big)$ and $f_t\big(q_\beta(X_t,\theta_0^\beta)\big)$, which is more common in the tails of the conditional distributions than in their central regions.
The second observation can be explained by noting that the efficiency gap is driven by the non-zero term $\alpha (1-\beta)$ in the off-diagonal entries of $S_t(X_t,\theta_0)$ in (\ref{eqn:DQREfficientSMainText}). It is particularly large for $\alpha \approx \beta \approx 1/2$ and particularly small for $\alpha <\hspace{-0.3em}< \beta$.

Panels D--F in Table \ref{tab:DQRSimSD} present results for the models with a joint intercept parameter.
We find that the general Z-estimation efficiency bound is still valid, which substantiates the statement of Theorem \ref{theorem:GeneralEfficientA}.
Differently from models with separated parameters, the pseudo-efficient choices $g_{t}(\xi) = F_t(\xi)$ generally cannot attain the efficiency bound, even in the homoskedastic residual case, 
which indicates that the efficiency gap applies to an even wider class of processes in joint parameter models.
For both heteroskedastic innovation distributions, the efficiency gap exists and is larger in magnitude.
Furthermore, the efficiency gap becomes substantially larger, especially in the example of Panel F, while the pseudo-efficient choices still result in the most efficient estimator among the considered choices of M-estimators.
These results show that the efficiency gap is present for a large class of double quantile models and data generating processes, which goes beyond the theoretically considered models of Theorem \ref{theorem:DoubleQuantileRegressionEfficiencyBound}.

\subsection{Joint Quantile and ES Models}
\label{sec:QESRSimStudy}

For joint quantile and ES models with separated model parameters, we use the process in (\ref{eqn:GeneralLocScaleDGP}) and utilize parametric choices which result in strictly negative ES values,
$X_t \stackrel{\textrm{iid}}{\sim}  3 \times \operatorname{Beta}(3, 1.5)$, $\zeta(X_t) = -1 -0.5  X_t$, and  $\eta(X_t) = 0.5 + 0.5  X_t$.
For the model innovations $\varepsilon_t$, we choose the following two specifications:
(a) $\varepsilon_t \stackrel{iid}{\sim}  \mathcal{N}(0,1)$; and
(b) $\varepsilon_t \sim t_{\nu_t}(\mu_t,\sigma_t^2)$ with time-varying degrees of freedom, $\nu_t = 3 \, \mathds{1}_{\{t \le T/2\}} + 100 \, \mathds{1}_{\{t > T/2\}}$, where $\mu_t$ and $\sigma_t$ are given in (\ref{eqn:QESRMutSigmat}).
%
These choices are motivated through the theoretical considerations of Section \ref{sec:QuantileESRegressionModel} that for location-scale models with i.i.d.\ residuals, the M-estimator is able to attain the Z-estimation efficiency bound, while conversely, it cannot do so for heterogeneously distributed innovations.
Also recall the empirical motivation of breakpoint models from Section \ref{sec:DQRSimStudy}

For the considered process, it holds that $Q_\alpha(Y_t|X_t)=  (-1 + 0.5z_\alpha) + (0.5z_\alpha  -0.5) X_t$ and $\ES_\alpha(Y_t|X_t) =  (-1 + 0.5 s_\alpha) + (0.5s_\alpha  -0.5) X_t$.
We estimate the following linear models with separated parameters,
\begin{align}
\label{eqn:QESR_Models_SepParameters}
q_\alpha(X_t,\theta) =  \theta^{(1)} + \theta^{(2)} X_t,
\qquad \text{and}  \qquad 
e_\alpha(X_t,\theta) =  \theta^{(3)} + \theta^{(4)} X_t,
\end{align}
which satisfy the conditions of Theorem \ref{theorem:QuantileESRegressionEfficiencyGap}.
We further consider linear models  with \textit{joint} model parameters where the conditions of Theorem \ref{theorem:QuantileESRegressionEfficiencyGap} do not hold in order to assess efficient estimation of quantile--ES models beyond the model classes considered in Theorem \ref{theorem:QuantileESRegressionEfficiencyGap}.
For this, we use a slightly modified parameterisation of the process by using $\eta(X_t) = 0.5  X_t$, which implies that $Q_\alpha(Y_t|X_t)=  -1 + (0.5z_\alpha  -0.5) X_t$ and $\ES_\alpha(Y_t|X_t) =  -1 + (0.5s_\alpha  -0.5) X_t$.
We use the (correctly specified) \textit{joint intercept models}
\begin{align}
\label{eqn:QESR_Models_JointParameters}
q_\alpha(X_t,\theta) =  \theta^{(1)} + \theta^{(2)} X_t,
\quad \text{and}  \quad 
e_\alpha(X_t,\theta) =  \theta^{(1)} + \theta^{(3)} X_t.
\end{align}

We consider the quantile and ES at joint probability levels $\alpha \in \{1\%,  2.5\%, 10\%\}$.
For $g_t$, we use the two (pseudo-efficient) choices $g_t(\xi_1)=0$ and $g_t(\xi_1)=F_t(\xi_1)$ coupled with several choices of $\phi_t$, see the first two columns of Table \ref{tab:QESRSim25Percent} for a detailed list.
The first two choices of $\phi_t$ correspond to sub-optimal choices as already noticed by \cite{DimiBayer2019}, 
whereas the next choice $\phi_t(\xi_2) = -\log(-\xi_2)$ coincides with the ubiquitous zero-homogeneous loss.
The latter two choices $\phi_t^{\text{eff1}}$ and $\phi_t^{\text{eff2}}$ are the pseudo-efficient choices given in  (\ref{eqn:QESREfficientChoice1}) and  (\ref{eqn:QESREfficientChoice2}).

\begin{table}[tb]
	\scriptsize
	\caption{Relative Asymptotic Standard Deviations of Joint Quantile and ES Models}
	\label{tab:QESRSim25Percent}
	\centering
	\begin{tabularx}{\linewidth}{XX @{\hspace{1cm}} lrrrr @{\hspace{0.5cm}} lrrrr}
		\addlinespace
		\toprule
		& & & \multicolumn{4}{c}{(a) Homoskedastic} && \multicolumn{4}{c}{(b) Heteroskedastic}  \\
		\cmidrule(lr){4-7} \cmidrule(lr){9-12} 
		$g_t(\xi_1)$ & $\phi_t(\xi_2)$& & $\theta_1$ & $\theta_2$ & $\theta_3$ & $\theta_4$ & &  $\theta_1$ & $\theta_2$ & $\theta_3$ & $\theta_4$  \\
		\midrule
		\\
		& & & \multicolumn{9}{c}{Panel A: Models with Separated Parameters} \\
		\cmidrule(lr){3-12}	
		$0$ & $\exp(\xi_2)$ &   &  1.232 &  1.386 &  1.100 &  1.163 &   &  2.696 &  4.186 &  1.900 &  2.499\\
		$F_t(\xi_1)$ & $\exp(\xi_2)$ &   &  1.184 &  1.283 &  1.100 &  1.163 &   &  1.782 &  1.866 &  1.900 &  2.499\\
		$0$ & $F_{\text{Log}}(\xi_2)$ &   &  1.220 &  1.367 &  1.088 &  1.145 &   &  2.684 &  4.164 &  1.890 &  2.479\\
		$F_t(\xi_1)$ & $F_{\text{Log}}(\xi_2)$ &   &  1.173 &  1.267 &  1.088 &  1.145 &   &  1.774 &  1.857 &  1.890 &  2.479\\
		$0$ & $-\log(-\xi_2)$ &   &  1.001 &  1.001 &  1.003 &  1.003 &   &  1.160 &  1.162 &  1.212 &  1.167\\
		$F_t(\xi_1)$ & $-\log(-\xi_2)$ &   &  1.001 &  1.001 &  1.003 &  1.003 &   &  1.160 &  1.162 &  1.212 &  1.167\\
		$0$ & $\phi_t^{\text{eff1}}(\xi_2)$ &   &  1.000 &  1.000 &  1.000 &  1.000 &   &  1.166 &  1.168 &  1.211 &  1.166\\
		$F_t(\xi_1)$ & $\phi_t^{\text{eff1}}(\xi_2)$ &   &  1.000 &  1.000 &  1.000 &  1.000 &   &  1.160 &  1.162 &  1.211 &  1.166\\
		$0$ & $\phi_t^{\text{eff2}}(\xi_2)$ &   &  1.000 &  1.000 &  1.000 &  1.000 &   &  1.211 &  1.200 &  1.142 &  1.100\\
		\multicolumn{2}{l}{Barendse Bound}  &   &  1.043 &  1.041 &  1.000 &  1.000 &   &  1.133 &  1.134 &  1.142 &  1.100\\
		\cmidrule(lr){3-12}	
		\multicolumn{2}{l}{Efficiency Bound} &   &  9.942 &  5.476 & 11.907 &  6.559 &   & 31.248 & 15.290 & 59.332 & 33.933\\
		\midrule
		\\
		&  & &\multicolumn{9}{c}{Panel B: Models with Joint Parameters} \\
		\cmidrule(lr){3-12}
		$0$ & $\exp(\xi_2)$ &   &  1.069 &  1.129 &  1.089 &&    &  2.126 &  2.825 &  2.075\\
		$F_t(\xi_1)$ & $\exp(\xi_2)$ &   &  1.060 &  1.107 &  1.081 & &   &  2.021 &  2.366 &  1.967\\
		$0$ & $F_{\text{Log}}(\xi_2)$ &   &  1.070 &  1.117 &  1.083 & &   &  2.103 &  2.737 &  2.025\\
		$F_t(\xi_1)$ & $F_{\text{Log}}(\xi_2)$ &   &  1.061 &  1.096 &  1.075 & &   &  1.995 &  2.287 &  1.916\\
		$0$ & $-\log(-\xi_2)$ &   &  1.098 &  1.070 &  1.062 & &   &  1.695 &  1.414 &  1.322\\
		$F_t(\xi_1)$ & $-\log(-\xi_2)$ &   &  1.095 &  1.068 &  1.061 & &   &  1.691 &  1.412 &  1.321\\
		$0$ & $\phi_t^{\text{eff1}}(\xi_2)$ &   &  1.042 &  1.031 &  1.027 & &   &  1.737 &  1.500 &  1.372\\
		$F_t(\xi_1)$ & $\phi_t^{\text{eff1}}(\xi_2)$ &   &  1.041 &  1.030 &  1.027 & &   &  1.722 &  1.489 &  1.365\\
		$0$ & $\phi_t^{\text{eff2}}(\xi_2)$ &   &  1.022 &  1.016 &  1.014 & &   &  1.773 &  1.488 &  1.344\\
		\cmidrule(lr){3-12}	
		\multicolumn{2}{l}{Efficiency Bound} &   &  3.676 &  2.529 &  2.680 & &   & 12.586 &  8.144 & 11.456\\
		\bottomrule
		\addlinespace
		\multicolumn{12}{p{.97\linewidth}}{
			This table presents the (approximated) relative asymptotic standard deviations for semiparametric joint quantile and ES models at joint probability level of $2.5\%$ for various choices of M-estimators together with the Z-estimation efficiency bound.
			The rows titled ``Efficiency Bound'' report the raw standard deviations whereas the remaining rows report the relative standard deviations compared to the efficiency bound.
			Panel A reports results for the models with separated parameters given in (\ref{eqn:QESR_Models_SepParameters}) while Panel B considers the joint intercept models given in (\ref{eqn:QESR_Models_JointParameters}).
			The two considered residual distributions are presented in the two vertical panels of the table.
			The line ``Barendse Bound'' in Panel A refers to the two-step efficiency bound of \cite{Barendse2020} discussed in Section \ref{sec:TwoStepEfficiency} and is reported here for completeness.
		}
	\end{tabularx}
\end{table}

Panel A of Table \ref{tab:QESRSim25Percent} presents the approximated parameter standard deviations for $\alpha = 2.5\%$ and for the separated parameter models in (\ref{eqn:QESR_Models_SepParameters}).
The results confirm the theoretical considerations of Section \ref{sec:QuantileESRegressionModel}: 
the M-estimator based on either of the pseudo-efficient choices, $\phi_t^\text{eff1}$ and $\phi_t^\text{eff2}$, attains the Z-estimation efficiency bound for location-scale models with homoskedastic innovations, while there is an efficiency gap for heteroskedastic innovation distributions with a magnitude of up to $15\%$.
Table \ref{tab:QESRSimSDTrue} in Section \ref{app:AdditionalTables} reports additional results for $\alpha = 1\%$ and $\alpha = 10\%$, which shows that the efficiency gap is more pronounced for small(er) probability levels, corresponding to the most important cases for the risk measures VaR and ES.

As indicated by (\ref{eqn:QESREfficientChoice1}), our simulation results confirm that, though counterintuitive at first sight, efficient M-estimation in the homoskedastic case can be accomplished by both, the traditional efficient choice of quantile regression, $g_t(\xi_1) = F_t(\xi_1)$, and by the zero-function $g_t(\xi_1) = 0$.
Furthermore, both pseudo-efficient choices $\phi^\text{eff1}$ and $\phi^\text{eff2}$ are able to attain the efficiency bound in the homoskedastic setting for separated parameter models.
However, their performance differs for heteroskedastic models, where the choice $\phi^\text{eff2}$ delivers more efficient ES estimates but at the same time slightly less efficient quantile estimates.
The function $\phi_t(\xi_2) = - \log(-\xi_2)$ performs almost as well as the pseudo-efficient choices throughout all considered designs, which is not surprising given its similar form to $\phi^\text{eff1}$.

Panel B of Table \ref{tab:QESRSim25Percent} presents results for the models with joint parameters given in  (\ref{eqn:QESR_Models_JointParameters}). 
While the Z-estimation efficiency bound is still valid, it cannot be attained by any of the M-estimators utilized in the simulation study, even in the homoskedastic case, for any of the chosen pseudo-efficient choices.
This implies that the efficiency gap for joint quantile and ES models goes beyond the model class considered in Theorem \ref{theorem:QuantileESRegressionEfficiencyGap}.
This holds similarly for the heteroskedastic case, where the efficiency gap becomes quantitatively much larger: 
the standard deviations of the pseudo-efficient choices are between $34\%$ and $77\%$ larger than the efficiency bound.

As in the heteroskedastic case of Panel A, the second pseudo-efficient choice slightly outperforms the first one also for this example of  joint parameter models.
Finally, among the considered M-estimators, the ubiquitous zero-homogeneous choice $g_t(\xi_1)=0$ and $\phi_t(\xi_2) = - \log(-\xi_2)$ performs relatively well and even outperforms both pseudo-efficient choices for the heteroskedastic innovation distributions and the joint parameter models of Panel B. 
This is especially remarkable given that, in contrast to the pseudo-efficient choices, it does not require any pre-estimates in practice.

\section{Conclusion}
\label{sec:Conclusion}

The results of this paper have important consequences.
On the theoretical side, they motivate the consideration of semiparametric M-estimation efficiency bounds, which will generally not coincide with the semiparametric efficiency bound of \cite{Stein1956} for multivariate functionals.
On the practical side, they suggest the use of a new pseudo-efficient and feasible loss function for M-estimation of semiparametric VaR and ES models \citep{Patton2019, Taylor2019}, which recently attain a lot of attraction.
We anticipate that similar results can be derived for multiple expectiles or the interquantile expectation (RVaR) \citep{Barendse2020}.

If interest is particularly on \emph{efficient} estimation for general functionals, our findings suggest the following practical recommendations:
If the M-estimator attains the Z-estimation efficiency bound, it seems advisable to use the former due to its often superior numerical stability, and as \citet{DFZ_CharMest} guarantees its consistency, which can be problematic for (efficient) Z- and GMM-estimation due to missing global identification.
On the other hand, in the presence of an efficiency gap, the efficient Z-estimator is more attractive as its efficiency dominates all M-estimators.
Its possibly lacking \emph{global} identification could be remedied by using a consistent pre-estimate---e.g., from (pseudo-efficient) M-estimation---and restricting the numerical optimization algorithm to a local search around the pre-estimate, and by relying on \emph{local} identification as e.g., suggested by \citet[p.\ 2127]{NeweyMcFadden1994}.


\begin{appendix}
	
\renewcommand{\theass}{A\arabic{ass}}   
\setcounter{ass}{2}

\section{Additional Assumptions}
\label{sec:AddAssumptions}

This section restates assumptions from other papers we use in our theory:
Assumption \eqref{ass:CharMest} combines Assumptions 2 and 3 of \citet{DFZ_CharMest}. 
Assumptions \eqref{ass:GPLclass} and \eqref{ass:FZclass} provide sufficient conditions for the respective characterization results of loss functions for the double quantile and joint quantile and ES models, respectively \citep{FisslerZiegel2016}.

\begin{ass}
	\label{ass:CharMest}
	\begin{enumerate}
	\item[(i)]
	For all random variables $Z=(Y,X) \sim F_Z \in \F_{\Z}$, assume that the map $m(X,\cdot)\colon\Theta\to\Xi$ is surjective almost surely. 
	Moreover, the conditional expectation $\E\big[\rho\big(Y,m(X,\theta)\big)\big|X\big]$ is continuous in $\theta$ almost surely. 
	\item[(ii)]
	For all random variables $Z=(Y,X) \sim F_Z \in \F_{\Z}$ and for any event $A\in \sigma(X)$ with positive probability $\P(A)>0$ the conditional distribution $F_{Z|A}$ is also in $\F_{\Z}$.
	\end{enumerate}
\end{ass}

\begin{ass}
	\label{ass:GPLclass}
	Let $\F_{\Y|\X}$ contain all continuously differentiable distribution functions with positive derivatives (densities) such that the double quantile maps surjectively on $\Xi$, which is assumed to be an open and path connected subset of $\R^2$.
	\end{ass}

\begin{ass}
	\label{ass:FZclass}
	Let $\F_{\Y|\X}$ contain all continuously differentiable distribution functions with positive derivatives (densities) and integrable lower tail such that $(Q_\alpha, \ES_\alpha)$ maps surjectively on $\Xi\subseteq\R^2$, which is assumed to be an open and path connected subset of $\R^2$.
\end{ass}

	
\end{appendix}




\singlespacing
\setlength{\bibsep}{2pt plus 0.3ex}
\def\bibfont{\small}

\bibliographystyle{apalike}

\newpage
\onehalfspacing
\setcounter{page}{1}
\setcounter{footnote}{0}
\begin{center}
	SUPPLEMENTARY MATERIAL FOR   \vspace{10pt} \\
	{\Large\bf {The Efficiency Gap} \vspace{10pt} }\\
	Timo Dimitriadis, Tobias Fissler and Johanna Ziegel\\
	This Version: \today \\ 
\end{center}

\renewcommand{\thesection}{S.\arabic{section}}   
\renewcommand{\thepage}{S.\arabic{page}}  
\renewcommand{\thetable}{S.\arabic{table}}   
\renewcommand{\thefigure}{S.\arabic{figure}}   

\setcounter{section}{0}
\setcounter{table}{0}
\setcounter{figure}{0}

\renewcommand{\theass}{S\arabic{ass}}   
\setcounter{ass}{5}

\section{Proofs for the Results of the Main Paper}
\label{app:Proofs}


\begin{proof}[Proof of Theorem \ref{theorem:GeneralEfficientA}]
	For $A_{t,C}^\ast(X_t,\theta_0) = C D_t(X_t,\theta_0)^\intercal S_t(X_t,\theta_0)^{-1}$, one obtains that
	$\Delta_{T,\mathbb{A}^\ast} = \frac{1}{T} \sum_{t=1}^T C \, \mathbb{E} \left[ D_t(X_t,\theta_0)^\intercal  S_t(X_t,\theta_0)^{-1} D_t(X_t,\theta_0) \right]$
	and $\Sigma_{T,\mathbb{A}^\ast} = \Delta_{T,\mathbb{A}^\ast}C^\intercal$.
	Thus, the asymptotic covariance of the Z-estimator based on the choice $A_{t,C}^\ast(X,\theta_0)$ is the limit of
	\[
	\Delta_{T,\mathbb{A}^\ast}^{-1} \Sigma_{T,\mathbb{A}^\ast}\left(\Delta_{T,\mathbb{A}^\ast}^{-1}\right)^\intercal
	= \Lambda_T^{-1} 
	= \left( \frac{1}{T} \sum_{t=1}^T \mathbb{E} \big[ D_t(X_t,\theta_0)^\intercal S_t(X_t,\theta_0)^{-1} D_t(X_t,\theta_0) \big] \right)^{-1}
	\] 
	for all deterministic and non-singular choices of $C$, which shows part \ref{theorem:GeneralEfficientAPart1} of Theorem \ref{theorem:GeneralEfficientA}.
	
	As the asymptotic covariance is independent of the choice of $C$, without loss of generality we continue with $C = I_q$ for the proof of part \ref{theorem:GeneralEfficientAPart2} and henceforth use the notation $A_{t}^\ast = A_{t, I_q}^\ast$.
	We define the random vector $\chi_{t,T} =  \big( \Delta_{T,\mathbb{A}}^{-1} A_t(X_t,\theta_0) - \Lambda_{T}^{-1} A_{t}^\ast(X_t,\theta_0) \big) \varphi\big(Y_t,m(X_t,\theta_0) \big)$ for all $t, 1\le t \le T, T \ge 1$.
	Straight-forward calculations yield that
	\begin{align*}
		&\frac{1}{T} \sum_{t=1}^T \mathbb{E} \left[ \chi_{t,T} \chi_{t,T}^\intercal \right] = \\
		&\Delta_{T,\mathbb{A}}^{-1} \left( \frac{1}{T} \sum_{t=1}^T \mathbb{E} \left[ A_t(X_t,\theta_0) \varphi\big(Y_t, m(X_t,\theta_0) \big) \varphi\big(Y_t,m(X_t,\theta_0)\big)^\intercal A_t(X_t,\theta_0)^\intercal \right] \right) (\Delta_{T,\mathbb{A}}^\intercal)^{-1} \\
		+ \, &\Lambda_T^{-1} \left( \frac{1}{T} \sum_{t=1}^T \mathbb{E} \left[ A_t^\ast(X_t,\theta_0) \varphi\big(Y_t, m(X_t,\theta_0) \big) \varphi\big(Y_t, m(X_t,\theta_0) \big)^\intercal A_t^\ast(X_t,\theta_0)^\intercal \right] \right) (\Lambda_T^{-1})^\intercal \\
		- \, &\Delta_{T,\mathbb{A}}^{-1} \left( \frac{1}{T} \sum_{t=1}^T \mathbb{E} \left[ A_t(X_t,\theta_0) \varphi\big(Y_t, m(X_t,\theta_0) \big) \varphi\big(Y_t, m(X_t,\theta_0) \big)^\intercal A_t^\ast(X_t,\theta_0)^\intercal \right] \right) (\Lambda_T^{-1})^\intercal \\
		- \, &\Lambda_T^{-1} \left( \frac{1}{T} \sum_{t=1}^T \mathbb{E} \left[ A_t^\ast(X_t,\theta_0) \varphi\big(Y_t, m(X_t,\theta_0) \big) \varphi\big(Y_t, m(X_t,\theta_0) \big)^\intercal A_t(X_t,\theta_0)^\intercal \right] \right) (\Delta_{T,\mathbb{A}}^\intercal)^{-1} \\
		= \, &\Delta_{T,\mathbb{A}}^{-1} \left( \frac{1}{T} \sum_{t=1}^T \mathbb{E} \left[ A_t(X_t,\theta_0) S_t(X_t,\theta_0) A_t(X_t,\theta_0)^\intercal \right] \right) (\Delta_{T,\mathbb{A}}^\intercal)^{-1} \\
		+ \, &\Lambda_T^{-1} \left( \frac{1}{T} \sum_{t=1}^T \mathbb{E} \left[ D_t(X_t,\theta_0)^\intercal S_t(X_t,\theta_0)^{-1} D_t(X_t,\theta_0) \right] \right) \Lambda_T^{-1} \\
		- \, &\Delta_{T,\mathbb{A}}^{-1} \left( \frac{1}{T} \sum_{t=1}^T \mathbb{E} \left[ A_t(X_t,\theta_0) D_t(X_t,\theta_0) \right] \right) \Lambda_T^{-1} \\
		- \, &
		\Lambda_T \left( \frac{1}{T} \sum_{t=1}^T \mathbb{E} \left[ D_t(X_t,\theta_0)^\intercal A_t(X_t,\theta_0)^\intercal \right] \right) (\Delta_{T,\mathbb{A}}^\intercal)^{-1} \\
		= \, &\Delta_{T,\mathbb{A}}^{-1} \Sigma_{T,\mathbb{A}} \Delta_{T,\mathbb{A}}^{-1} - \Lambda_T^{-1},
	\end{align*}
	which is positive semi-definite for all $T \ge 1$ as the sum of outer products, which concludes the proof of part \ref{theorem:GeneralEfficientAPart2}. 
	
	For the proof of part \ref{theorem:GeneralEfficientAPart3}, assume that for some $t = 1,\dots,T$, the matrix $A_t(X_t,\theta)$ is such that $A_t(X_t,\theta_0) \not= A_{t,C}^\ast(X_t,\theta_0)$ for any non-singular and deterministic matrix $C$  with positive probability.
	Then, for some $t = 1,\dots,T$, the matrix
	\(
	M_{T,\mathbb{A}} (X_t,\theta_0) := \Delta_{T,\mathbb{A}}^{-1} A_t(X_t,\theta_0) - \Lambda_{T}^{-1} A_{t,C}^\ast(X_t,\theta_0) 
	\)
	is nonzero with positive probability, as otherwise $A_t(X_t,\theta_0) =  A_{t,  \tilde C}^\ast(X_t,\theta_0)$ almost surely with $\tilde C = \Delta_{T,\mathbb{A}} \Lambda_{T}^{-1} C$.
	This implies that the matrix $M_{T,\mathbb{A}} (X_t,\theta_0)$ has positive rank with positive probability.
	Furthermore, the matrix $S_t(X_t,\theta_0)$ defined in \eqref{eqn:DefSMatrix}
	is positive definite with probability one for all $t = 1,\dots,T$ by assumption.
	Consequently, we can apply the Cholesky decomposition and get that there exists a lower triangular matrix $G_t(X_t,\theta_0)$ with strictly positive diagonal entries such that $S_t(X_t,\theta_0) = G_t(X_t,\theta_0) G_t(X_t,\theta_0)^\intercal$ almost surely, i.e., the matrix $G_t(X_t,\theta_0)$ has full rank almost surely.
	Thus, the matrix
	\(
	B_{T,\mathbb{A},t}(X_t,\theta_0) := M_{T,\mathbb{A}}(X_t,\theta_0) G_t(X_t,\theta_0)
	\)
	has positive rank for some $t = 1,\dots,T$ with positive probability by Sylvester's rank inequality as it is the product of matrices with strictly positive rank (with positive probability) and full rank (almost surely), respectively.
	Consequently, there exists a $j \in \{ 1,\dots,k\}$ such that
	\begin{align}\label{eq:positive prob}
		\mathbb{P} \big( B_{T,\mathbb{A},t}(X_t,\theta_0)^\intercal e_j \not= 0 \big) > 0, \qquad \text{for some }t = 1,\dots,T,
	\end{align}
	where $e_j$ is the $j$-th standard basis vector of $\mathbb{R}^k$.
	Thus, 
	\begin{align*}
		&e_j^\intercal \left( \frac{1}{T} \sum_{t=1}^T \mathbb{E} \left[ \chi_{t,T} \chi_{t,T}^\intercal \right] \right) e_j 
		= \frac{1}{T} \sum_{t=1}^T \mathbb{E} \left[e_j^\intercal  M_{T,\mathbb{A}}(X_t,\theta_0) S_t(X_t,\theta_0) M_{T,\mathbb{A}}(X_t,\theta_0)^\intercal e_j \right] \\
		&= \frac{1}{T} \sum_{t=1}^T \mathbb{E} \left[e_j^\intercal  B_{T,\mathbb{A},t}(X_t,\theta_0) B_{T,\mathbb{A},t}(X_t,\theta_0)^\intercal e_j \right] 
		= \frac{1}{T} \sum_{t=1}^T \mathbb{E} \left[ \left\| B_{T,\mathbb{A},t}(X_t,\theta_0)^\intercal e_j \right\|^2 \right]
		> 0,
	\end{align*}
	for all $T \ge 1$, since all summands are non-negative and, invoking \eqref{eq:positive prob}, at least one summand must be strictly positive,
	which shows that the matrix
	\(
	\Delta_{T,\mathbb{A}}^{-1} \Sigma_{T,\mathbb{A}} \Delta_{T,\mathbb{A}}^{-1} - \Lambda_T^{-1}
	\)
	has at least one strictly positive eigenvalue, which concludes the proof of the theorem.
\end{proof}

\begin{proof}[Proof of Theorem \ref{theorem:DoubleQuantileRegressionEfficiencyBound}]		
	Under Assumptions \eqref{ass:unique model}, \eqref{ass:CharMest} and \eqref{ass:GPLclass}, \citet[Theorem 1]{DFZ_CharMest} and \citet[Proposition 4.2]{FisslerZiegel2016} yield that any consistent M-estimator of semiparametric double quantile models is based on classical (strictly) consistent loss functions for the pair of two quantiles, given in (\ref{eqn:DoubleQuantileGeneralLossFunctions}).	
	Furthermore, the M- and Z-estimator have identical asymptotic covariance if and only if the moment conditions of the Z- and derivative of the loss of the M-estimator coincide, or, respectively, their conditional expectations coincide, see the discussion after \eqref{eqn:DefDMatrix}.
	Thus, in the following we compare whether the derivatives of any strictly consistent loss function given in (\ref{eqn:DoubleQuantileGeneralLossFunctions})  can attain the efficient moment conditions of the Z-estimator almost surely.
	

	We get that all identification functions which correspond to an M-estimator (in the form of a derivative of the conditional expectation almost surely) are given by
	\begin{align*}
			\psi_{g_{1,t},g_{2,t}}(Y_t,X_t,\theta) = 
			\begin{pmatrix}
				\nabla_{\theta^\alpha} q_\alpha(X_t,\theta^\alpha) g_{1,t}'\big(q_\alpha(X_t,\theta^\alpha)\big) \big( \mathds{1}_{\{ Y_t \le q_\alpha(X_t,\theta^\alpha)\}} - \alpha \big) \\
				\nabla_{\theta^\beta} q_\beta(X_t,\theta^\beta) g_{2,t}'\big(q_\beta(X_t,\theta^\beta)\big) \big( \mathds{1}_{\{ Y_t \le q_\beta(X_t,\theta^\beta)\}} - \beta \big)
			\end{pmatrix},
	\end{align*}
	which can be written as
	$\psi_{g_{1,t},g_{2,t}}(Y_t,X_t,\theta) = A_{g_{1,t},g_{2,t}} (X_t,\theta)  \varphi\big(Y_t, m(X_t,\theta) \big)$,
	where
	\begin{align}
		\label{eqn:DQRGeneralInstrumentMatrixMestimator}
		A_{g_{1,t},g_{2,t}} (X_t,\theta) = 		
		\begin{pmatrix}
			\nabla_{\theta^\alpha} q_\alpha(X_t,\theta^\alpha) g_{1,t}'\big(q_\alpha(X_t,\theta^\alpha)\big) & 0 \\
			0 & \nabla_{\theta^\beta} q_\beta(X_t,\theta^\beta) g_{2,t}'\big(q_\beta(X_t,\theta^\beta)\big)
		\end{pmatrix}.
	\end{align}
	We start by showing statement \ref{statement:DQREfficiencyGap}, assuming that the Z-estimation efficiency bound is attained by the M-estimator.
	From Theorem \ref{theorem:GeneralEfficientA}, part \ref{theorem:GeneralEfficientAPart1} and  \ref{theorem:GeneralEfficientAPart2}, we get that the efficient instrument choice is given by $A^\ast_{t,C}(X_t,\theta_0) = C D_t(X_t,\theta_0)^\intercal S_t(X_t,\theta_0)^{-1}$, at the true parameter $\theta_0$, where $C$ is some deterministic and nonsingular matrix and where $D_t(X_t,\theta_0)$ and $S_t(X_t,\theta_0)$ are given in \eqref{eqn:DQREfficientSMainText}.
	Furthermore, Theorem \ref{theorem:GeneralEfficientA} part \ref{theorem:GeneralEfficientAPart3} shows that any choice of $A_t(X_t,\theta_0)$  which deviates from $A^\ast_{t,C}(X_t,\theta_0)$ (at the true parameter $\theta_0$) with positive probability for some $t \in \mathbb{N}$, cannot attain the efficiency bound.
	Thus, in the following we show by contradiction that the general instrument matrix of the M-estimator, $A_{g_{1,t},g_{2,t}}(X_t,\theta_0)$, given in (\ref{eqn:DQRGeneralInstrumentMatrixMestimator}), cannot attain the necessary form $A^\ast_{t,C}(X_t,\theta_0)$ at the true parameter $\theta_0$ with probability one for any deterministic matrix $C$.
	
	For this, we assume that there exists a deterministic and non-singular $q \times q$ matrix $C$ and functions $g_{1,t}$ and $g_{2,t}$ such that $A^\ast_{t,C}(X_t,\theta_0) = A_{g_{1,t},g_{2,t}}(X_t,\theta_0)$ almost surely for all $t \in \mathbb{N}$.
	We split $C = \left( \begin{smallmatrix} C_{11} & C_{12} \\ C_{21} & C_{22} \end{smallmatrix} \right)$ in its respective parts, where $C_{11} \in \mathbb{R}^{q_1 \times q_1}$, $C_{22} \in \mathbb{R}^{q_2 \times q_2}$, and $C_{12}, C_{21}^\intercal \in \mathbb{R}^{q_1 \times q_2}$.
	Then, the equation $A^\ast_{t,C}(X_t,\theta_0) = A_{g_{1,t},g_{2,t}}(X_t,\theta_0)$ is equivalent to
	\begin{align}
		\begin{aligned}
			\label{eqn:DQRModelAMatrixEquality}
			&\begin{pmatrix}
				\alpha (1-\alpha) g_{1,t}'\big(q_\alpha(X_t,\theta_0^\alpha)\big) \nabla_{\theta^\alpha} q_\alpha(X_t,\theta_0^\alpha)  &
				\alpha (1-\beta) g_{1,t}'\big(q_\alpha(X_t,\theta_0^\alpha)\big) \nabla_{\theta^\alpha} q_\alpha(X_t,\theta_0^\alpha) \\
				\alpha (1-\beta) g_{2,t}'\big(q_\beta(X_t,\theta_0^\beta)\big) \nabla_{\theta^\beta} q_\beta(X_t,\theta_0^\beta) &
				\beta (1-\beta) g_{2,t}'\big(q_\beta(X_t,\theta_0^\beta)\big) \nabla_{\theta^\beta} q_\beta(X_t,\theta_0^\beta) 
			\end{pmatrix} \\
			& = \begin{pmatrix}
				f_t\big(q_\alpha(X_t,\theta_0^\alpha)\big)  C_{11} \nabla_{\theta^\alpha} q_\alpha(X_t,\theta_0^\alpha)  &
				f_t\big(q_\beta(X_t,\theta_0^\beta)\big)  C_{12} \nabla_{\theta^\beta} q_\beta(X_t,\theta_0^\beta) \\
				f_t\big(q_\alpha(X_t,\theta_0^\alpha)\big)  C_{21} \nabla_{\theta^\alpha} q_\alpha(X_t,\theta_0^\alpha) &
				f_t\big(q_\beta(X_t,\theta_0^\beta)\big)  C_{22} \nabla_{\theta^\beta} q_\beta(X_t,\theta_0^\beta) 
			\end{pmatrix},
		\end{aligned}
	\end{align}
	which must hold element-wise for all four sub-components.
	Equality of the upper left component yields that there is some $A\in \mathcal A$ with $\P(A)=1$ such that 
	\begin{align}
		\label{eqn:RandomeigenvalueEquation}
		\xi_t(\omega) \cdot \nabla_{\theta^\alpha} q_\alpha\big(X_t(\omega),\theta_0^\alpha\big) =  C_{11} \cdot \nabla_{\theta^\alpha} q_\alpha\big(X_t(\omega),\theta_0^\alpha\big), \quad \forall \omega\in A
	\end{align}
	for the scalar random variable $\xi_t := \alpha (1-\alpha) g_{1,t}'\big(q_\alpha(X_t,\theta_0^\alpha)\big)/f_t\big(q_\alpha(X_t,\theta_0^\alpha)\big)$.
	Equation (\ref{eqn:RandomeigenvalueEquation}) is an eigenvalue problem for the deterministic matrix $C_{11}$ with stochastic eigenvalues $\xi_t(\omega)$  and eigenvectors $ \nabla_{\theta^\alpha} q_\alpha(X_t(\omega),\theta_0^\alpha)$, $\omega\in A$. 
	We now show that this equation only holds if $\xi_t$ is constant on $A$.
	
	By Assumption \ref{cond:LinearIndependentImageValues}, there are $\omega_1, \ldots, \omega_{q_1+1}\in A$ such that for $v_\ell:= \nabla_{\theta^\alpha} q_\alpha\big(X_t(\omega_\ell),\theta_0^\alpha\big)$, $\ell\in\{1, \ldots, q_1+1\}$, any subset of cardinality $q_1$ of $\{v_1, \ldots, v_{q_1+1}\}$ is linearly independent. 
	As $C_{11}$ is a deterministic $q_1 \times q_1$ matrix, it can have at most $q_1$ different eigenvalues.
	Let $\lambda_1, \dots, \lambda_{q_1}$ be the eigenvalues of $C_{11}$ (not necessarily different, thus counted multiple times for higher algebraic multiplicities) ordered such that $v_i$ is an eigenvector for eigenvalue $\lambda_i$ for all $i = 1,\dots,q_1$.
	Invoking that $v_1,\dots,v_{q_1}$ are linearly independent, it holds that
		$
		\sum_{ \lambda \in \{ \lambda_1,\dots,\lambda_{q_1} \} } \operatorname{dim}(E_\lambda) = q_1,
		$
	where the summation ignores repetitions in the set $\{ \lambda_1,\dots,\lambda_{q_1} \}$ and where $E_\lambda$ denotes the eigenspace corresponding to eigenvalue $\lambda$.
	The eigenvector $v_{q_1+1}$ must be contained in $E_{\lambda_i}$ for some $i=1,\dots,q_1$ as otherwise, the sum of the geometric multiplicities would exceed $q_1$.
	If $\operatorname{dim}(E_{\lambda_i}) = l < q_1$, then $E_{\lambda_i}$ is spanned by $l$ elements of $\{ v_1,\dots,v_{q_1}\}$, and as $v_{q_1+1}$ is contained in $E_{\lambda_i}$, these $l$ elements of $\{ v_1,\dots,v_{q_1}\}$ then must be linearly dependent together with $v_{q_1+1}$.
	This contradicts Assumption \ref{cond:LinearIndependentImageValues}.
	Thus, $\operatorname{dim}(E_{\lambda_i}) = q_1$ and consequently, the geometric multiplicity of $\lambda_i$ is $q_1$, which then must equal the algebraic multiplicity. 
	Hence, all eigenvalues of $C_{11}$ are equal, $\lambda_1 = \cdots = \lambda_{q_1}$, and consequently, $\xi_t$ is constant on $A$, implying that it is constant almost surely.
	This implies that $g_{1,t}'\big(q_\alpha(X_t,\theta_0^\alpha) \big) = c_2 f_t\big(q_\alpha(X_t,\theta_0^\alpha)\big)$ almost surely for some constant $c_2 > 0$ and for all $t \in \mathbb{N}$, i.e., (\ref{eqn:CondEffGapDQRConditions2}).
	An analogous proof for the lower right entry of (\ref{eqn:DQRModelAMatrixEquality}) shows (\ref{eqn:CondEffGapDQRConditions3}), which concludes the proof of \ref{statement:DQREfficiencyGap}.
	
	For \ref{statement:DQREfficientChoice} we start with the ``only if'' direction assuming that the M-estimator attains the efficiency bound. 
	From part \ref{statement:DQREfficiencyGap}, we already obtain that 
	(\ref{eqn:CondEffGapDQRConditions2}) and (\ref{eqn:CondEffGapDQRConditions3}) must hold. 
	Exploiting 
	$\nabla_{\theta^\alpha} q_\alpha(X_t,\theta_0^\alpha)  = \nabla_{\theta^\beta} q_\beta(X_t,\theta_0^\beta)$ and $g_{1,t}'\big(q_\alpha(X_t,\theta_0^\alpha)\big) = c_2 f_t\big(q_\alpha(X_t,\theta_0^\alpha)\big)$, the upper right component of (\ref{eqn:DQRModelAMatrixEquality}) implies that
	\begin{align}
		\label{eqn:RandomeigenvalueEquationOffDiagonal}
		\frac{ \alpha (1-\beta) c_2 f_t\big(q_\alpha(X_t,\theta_0^\alpha)\big)}{f_t\big(q_\beta(X_t,\theta_0^\beta)\big)} \cdot \nabla_{\theta^\alpha}  q_\alpha(X_t,\theta_0^\alpha) =  C_{12} \cdot \nabla_{\theta^\alpha}  q_\alpha(X_t,\theta_0^\alpha),  
	\end{align}
	almost surely.
	Applying the same eigenvalue argument to (\ref{eqn:RandomeigenvalueEquationOffDiagonal}) (recalling that $\nabla_{\theta^\alpha} q_\alpha(X_t,\theta_0^\alpha)  = \nabla_{\theta^\beta} q_\beta(X_t,\theta_0^\beta)$ implies that $q_1=q_2$ such that $C_{12}$ is quadratic) yields (\ref{eqn:CondEffGapDQRConditions1}).	
	
	
	For the ``if'' implication in \ref{statement:DQREfficientChoice}, we assume that (\ref{eqn:CondEffGapDQRConditions1}), (\ref{eqn:CondEffGapDQRConditions2}) and (\ref{eqn:CondEffGapDQRConditions3}) hold.
	We choose  $C_{11} = \alpha (1-\alpha )c_2 I_{q_1 \times q_1}$, $C_{12} = \alpha (1-\beta) c_1 c_2 I_{q_1 \times q_2}$,  $C_{21} = \alpha (1-\beta) c_3/c_1 I_{q_2 \times q_1}$ and $C_{22} = \beta (1-\beta) c_3 I_{q_2 \times q_2}$, where $\operatorname{det}(C) \not= 0$ follows from $0<\alpha<\beta<1$.
	Thus, straightforward calculations yield that $A_{g_{1,t},g_{2,t}} (X_t,\theta_0) = A^\ast_{t,C}(X_t,\theta_0)$ holds almost surely for all $t \in \mathbb{N}$.
	Applying Theorem \ref{theorem:GeneralEfficientA}
	yields the claim.
\end{proof}

\begin{proof}[Proof of Theorem \ref{theorem:QuantileESRegressionEfficiencyGap}]
	This proof follows the general ideas of the proof of Theorem \ref{theorem:DoubleQuantileRegressionEfficiencyBound}.
	Under Assumptions \eqref{ass:unique model}, \eqref{ass:CharMest} and \eqref{ass:FZclass}, \citet[Theorem 1]{DFZ_CharMest}  and \citet[Theorem 5.2, Corollary 5.5]{FisslerZiegel2016} yield that any consistent M-estimator of joint quantile and ES models is based on classical (strictly) consistent loss functions given in (\ref{eqn:QESRGeneralLossFunctions}).
	Thus, in the following we analyze whether the derivatives of any consistent loss function are able to match the efficient moment conditions of Theorem \ref{theorem:GeneralEfficientA} almost surely.
	
	We get that all identification functions which correspond to an M-estimator (in the form of a derivative almost surely) are given by $\psi_{g_t,\phi_t} \big( Y_t, X_t, \theta \big)$ equalling
			\[
			\begin{pmatrix}
				\nabla_{\theta^q} q_\alpha(X_t,\theta^q) \Big( g_t'\big(q_\alpha(X_t,\theta^q)\big) + \phi_t'\big(e_\alpha(X_t,\theta^e)\big)/\alpha \Big)  \left( \mathds{1}_{\{Y_t \le q_\alpha(X_t,\theta^q)\}} - \alpha \right)  \\[0.3em]
				\nabla_{\theta^e} e_\alpha(X_t,\theta^e) \phi_t''\big(e_\alpha(X_t,\theta^e)\big) 
				\left( e_\alpha(X_t,\theta^e) - q_\alpha(X_t,\theta^q) + \frac{1}{\alpha}(q_\alpha(X_t,\theta^q) - Y_t) \mathds{1}_{\{Y_t \le q_\alpha(X_t,\theta^q)\}}  \right)
			\end{pmatrix}.
		\]
		This implies that the moment conditions corresponding to an M-estimator can be written as $\psi_{g_t,\phi_t} \big( Y_t, X_t, \theta \big) =	A_{g_t,\phi_t}(X_t,\theta)  \varphi \big( Y_t, m(X_t,\theta) \big)$, where $\varphi\big( Y_t, m(X_t,\theta) \big)$ is given in (\ref{eqn:QESIdFunction}), and $A_{g_t,\phi_t}(X_t,\theta)$ is a diagonal matrix with entries 
		$\Big( g_t'\big(q_\alpha(X_t,\theta^q)\big) + \phi_t'\big(e_\alpha(X_t,\theta^e)\big)/\alpha \Big) 
		\times \nabla_{\theta^q} q_\alpha(X_t,\theta^q)$
		and 
		$ \phi_t''\big(e_\alpha(X_t,\theta^e)\big) \nabla_{\theta^e} e_\alpha(X_t,\theta^e)$.
		To show \ref{statement:QESREfficiencyGap}, we assume that the Z-estimation efficiency bound is attained by the M-estimator.
		From Theorem \ref{theorem:GeneralEfficientA}, we get that the efficient estimator has to fulfill the condition $A^\ast_{t,C}(X_t,\theta_0) = C D_t(X_t,\theta_0)^\intercal S_t(X_t,\theta_0)^{-1}$ for some deterministic and nonsingular matrix $C$, where $D_t(X_t,\theta_0)$ and $S_t(X_t,\theta_0)$ are given in \eqref{eqn:QESREfficientMatrixD} and \eqref{eqn:QESREfficientMatrixS}.
		%
		Thus, we verify whether there exists a deterministic and non-singular $q \times q$ matrix $C$ (and appropriate functions $g_t$ and $\phi_t$) such that $A^\ast_{t,C}(X_t,\theta_0) = A_{g_t,\phi_t}(X_t,\theta_0)$ almost surely, i.e., whether $C D_t(X_t,\theta_0)^\intercal = A_{g_t,\phi_t}(X_t,\theta_0) S_t(X_t,\theta_0)$ holds almost surely.
		By splitting the matrix $C = \begin{pmatrix} C_{11} & C_{12} \\ C_{21} & C_{22} \end{pmatrix}$ in its respective parts, where $C_{11} \in \mathbb{R}^{q_1 \times q_1}$, $C_{22} \in \mathbb{R}^{q_2 \times q_2}$, and $C_{12}, C_{21}^\intercal \in \mathbb{R}^{q_1 \times q_2}$, this simplifies to the following four equalities,
		\begin{align}
			\label{QESREfficiencyGapEquation11}
			C_{11} \nabla_{\theta^q} q_\alpha(X_t,\theta^q_0) = (1-\alpha)  \frac{\alpha g_t'\big(q_\alpha(X_t,\theta^q_0)\big) + \phi_t'\big(e_\alpha(X_t,\theta^e_0)\big)}{ f_t\big(q_\alpha(X_t,\theta^q_0)\big)}  \nabla_{\theta^q} q_\alpha(X_t,\theta^q_0), 
			\\
			\begin{aligned}
				\label{QESREfficiencyGapEquation12}
				C_{12}	\nabla_{\theta^e} e_\alpha(X_t,\theta^e_0)  &= (1-\alpha)  \big( q_\alpha(X_t,\theta^q_0) - e_\alpha(X_t,\theta^e_0) \big) \\
				&\qquad \times \left(  g_t'\big(q_\alpha(X_t,\theta^q_0)\big) + \phi_t'\big(e_\alpha(X_t,\theta^e_0)\big)/\alpha \right) \nabla_{\theta^q} q_\alpha(X_t,\theta^q_0),
			\end{aligned}
			\\
			\label{QESREfficiencyGapEquation21}
			C_{21} \nabla_{\theta^q} q_\alpha(X_t,\theta^q_0) = \frac{(1-\alpha) \big( q_\alpha(X_t,\theta^q_0) - e_\alpha(X_t,\theta^e_0) \big) \phi_t''\big(e_\alpha(X_t,\theta^e_0)\big) }{f_t\big(q_\alpha(X,\theta^q_0)\big) } \nabla_{\theta^e} e_\alpha(X_t,\theta^e_0),
			\\
			\begin{aligned}
				\label{QESREfficiencyGapEquation22}
				&C_{22} \nabla_{\theta^e} e_\alpha(X_t,\theta^e_0) = \phi_t''\big(e_\alpha(X_t,\theta^e_0)\big) \nabla_{\theta^e} e_\alpha(X_t,\theta^e_0) \\
				&\quad \times
				\left( \frac{1}{\alpha} \operatorname{Var}_t \big(  Y_t  \big| Y_t \le q_\alpha(X_t,\theta^q_0) \big) 
				+ \frac{1-\alpha}{\alpha} \big( e_\alpha(X_t,\theta^e_0) - q_\alpha(X_t,\theta^q_0) \big)^2\right),
			\end{aligned}
		\end{align}
		which have to hold almost surely.
		Using the same eigenvalue argument as in the proof of Theorem \ref{theorem:DoubleQuantileRegressionEfficiencyBound}, equation (\ref{QESREfficiencyGapEquation11}) implies that 
		\begin{align}
			\label{QESREfficiencyGapEquation11Scalar}
			(1 - \alpha) \big( \alpha g_t'\big(q_\alpha(X_t,\theta^q_0)\big) + \phi_t'\big(e_\alpha(X_t,\theta^e_0)\big) \big) = \tilde c_1 f_t\big(q_\alpha(X_t,\theta^q_0)\big)
		\end{align}
		almost surely for some constant $\tilde c_1 > 0$.
		Equation (\ref{eqn:CondEffGapQESRConditionAdditional}) follows by setting $c_6 = \tilde c_1/(\alpha(1-\alpha))$.
		Similarly, (\ref{QESREfficiencyGapEquation22}) implies that 
		\begin{align}
			\label{QESREfficiencyGapEquation22Scalar}
			\frac{\tilde c_2}{\phi_t''\big(e_\alpha(X_t,\theta^e_0)\big)} = \frac{1}{\alpha} \operatorname{Var}_t \big(  Y_t  \big| Y_t \le q_\alpha(X_t,\theta^q_0) \big) 
			+ \frac{1-\alpha}{\alpha} \big( e_\alpha(X_t,\theta^e_0) - q_\alpha(X_t,\theta^q_0) \big)^2
		\end{align}
		almost surely for some constant $\tilde c_2 >0$.
		Furthermore, combining (\ref{QESREfficiencyGapEquation12}) and (\ref{QESREfficiencyGapEquation21}) implies
		\begin{align*}
				&C_{12} C_{21} \nabla q_\alpha(X_t,\theta^q_0) = \nabla q_\alpha(X_t,\theta^q_0) (1-\alpha)^2 \big/ \alpha \\
				&\times \frac{\big( q_\alpha(X_t,\theta^q_0) - e_\alpha(X_t,\theta^e_0) \big)^2 \phi_t''\big(e_\alpha(X_t,\theta^e_0)\big) \left( \alpha g_t'\big(q_\alpha(X_t,\theta^q_0)\big) + \phi_t'\big(e_\alpha(X_t,\theta^e_0)\big) \right)}{ f_t\big(q_\alpha(X_t,\theta^q_0)\big)} 
		\end{align*}
		almost surely and employing the same eigenvalue argument again yields that
		\begin{align}
			\begin{aligned}
				\label{QESREfficiencyGapEquation1221Scalar}
				&\big( q_\alpha(X_t,\theta^q_0) - e_\alpha(X_t,\theta^e_0) \big)^2 \phi_t''\big(e_\alpha(X_t,\theta^e_0)\big) \left( \alpha g_t'\big(q_\alpha(X_t,\theta^q_0)\big) + \phi_t'\big(e_\alpha(X_t,\theta^e_0)\big) \right) \\
				&= \frac{\tilde c_3 \alpha}{(1-\alpha)^2} f_t\big(q_\alpha(X_t,\theta^q_0)\big)
			\end{aligned}
		\end{align}
		almost surely for some constant $\tilde c_3 > 0$.
		Substituting (\ref{QESREfficiencyGapEquation11Scalar}) and (\ref{QESREfficiencyGapEquation22Scalar}) into (\ref{QESREfficiencyGapEquation1221Scalar}) finally yields
		\begin{align}
			\label{QESREfficiencyGapEquationSubstituted1Scalar}
			\operatorname{Var}_t \big(  Y_t  \big| Y_t \le q_\alpha(X_t,\theta^q_0) \big) 
			= (1-\alpha) \left( \frac{ \tilde c_1 \tilde c_2}{ \tilde c_3} - 1 \right) \big( q_\alpha(X_t,\theta^q_0) - e_\alpha(X_t,\theta^e_0) \big)^2.
		\end{align}
		almost surely.
		By defining the constant $c_1 :=  (1-\alpha) \left( \frac{ \tilde c_1 \tilde c_2}{ \tilde c_3} - 1 \right)$, we obtain (\ref{eqn:CondEffGapQESRConditions1}), where the positivity of $c_1$ follows from that fact that both sides of \eqref{QESREfficiencyGapEquationSubstituted1Scalar} are positive.
		Substituting (\ref{QESREfficiencyGapEquationSubstituted1Scalar}) into (\ref{QESREfficiencyGapEquation22Scalar}) yields (\ref{eqn:CondEffGapQESRConditions3}), which concludes the proof of statement \ref{statement:QESREfficiencyGap}.
		
		For \ref{statement:QESREfficientChoice} we start with the ``only if'' direction, assuming that the M-estimator attains the efficiency bound.
		From part \ref{statement:QESREfficiencyGap} we obtain that \eqref{eqn:CondEffGapQESRConditions1},  \eqref{eqn:CondEffGapQESRConditions3} and \eqref{eqn:CondEffGapQESRConditionAdditional} must hold.
		Employing the same eigenvalue argument as before, we obtain from (\ref{QESREfficiencyGapEquation21}) that 
		$
		\tilde c_4 f_t\big(q_\alpha(X_t,\theta^q_0)\big)/(1-\alpha) = \big(q_\alpha(X_t,\theta^q_0) - e_\alpha(X_t,\theta^e_0) \big)\phi_t''\big(e_\alpha(X_t,\theta^e_0)\big)
		$
		for some constant $\tilde c_4>0$, where we additionally exploited  that $\nabla_{\theta^q} q_\alpha(X_t,\theta_0^q) = \nabla_{\theta^e} e_\alpha(X_t,\theta_0^e)$ almost surely.
		Combining \eqref{eqn:CondEffGapQESRConditions1} and  \eqref{eqn:CondEffGapQESRConditions3} yields $\phi_t''\big(e_\alpha(X_t,\theta^e_0)\big) = c_3/c_1 \big( q_\alpha(X_t,\theta^q_0) - e_\alpha(X_t,\theta^e_0) \big)^{-2}$, which in turn leads us to
		\begin{align}
			\label{eqn:QESREfficiencyGapFtCondition}
			f_t\big(q_\alpha(X_t,\theta^q_0)\big) = \frac{c_2}{q_\alpha(X_t,\theta^q_0) - e_\alpha(X_t,\theta^e_0)},
		\end{align}
		almost surely for where $c_2 =( 1-\alpha) c_3/(c_1 \tilde c_4) > 0$, establishing \eqref{eqn:CondEffGapQESRConditions2}.
		Using again that $\phi_t''\big(e_\alpha(X_t,\theta^e_0)\big) = c_3/c_1 \big( q_\alpha(X_t,\theta^q_0) - e_\alpha(X_t,\theta^e_0) \big)^{-2}$ and since the support of $e_\alpha(X_t,\theta^e_0)$ is a non-degenerate interval by assumption, it must hold that
		\begin{align}
			\label{eqn:PhiPRime}
			\phi_t'\big(e_\alpha(X_t,\theta^e_0)\big) = \frac{c_3/c_1}{ q_\alpha(X_t,\theta^q_0) - e_\alpha(X_t,\theta^e_0)} + \tilde c_{5,t},
		\end{align}
		almost surely for some deterministic, but possibly time-varying constant $\tilde c_{5,t} \in \mathbb{R}$ for all $t \in \mathbb{N}$.
		Combining (\ref{QESREfficiencyGapEquation11Scalar}), (\ref{eqn:QESREfficiencyGapFtCondition}) and (\ref{eqn:PhiPRime})
		yields that
		\begin{align}
			\label{eqn:QESREfficiencyGapG'Condition}
			g_t'\big(q_\alpha(X_t,\theta^q_0)\big) = c_4	f_t\big(q_\alpha(X_t,\theta^q_0)\big) + c_{5,t},
		\end{align}
		\color{black}
		where $c_4 :=  \left( \frac{\tilde c_1}{\alpha (1 - \alpha) } - \frac{c_3}{\alpha c_1 c_2} \right) \in \mathbb{R}$ and $c_{5,t} := - \tilde c_{5,t}/\alpha$, which establishes \eqref{eqn:CondEffGapQESRConditions4}.
		Eventually, employing (\ref{QESREfficiencyGapEquation11Scalar}) and (\ref{eqn:QESREfficiencyGapG'Condition}) yields that
		$
		\phi_t'\big(e_\alpha(X_t,\theta^e_0)\big)/\alpha = \tilde c_1f_t\big(q_\alpha(X_t,\theta^q_0)\big)/(\alpha (1-\alpha)) - c_4 f_t\big(q_\alpha(X_t,\theta^q_0)\big) - c_{5,t},
		$
		and hence
		$\phi_t'\big(e_\alpha(X_t,\theta^e_0)\big) = c_3 f_t\big(q_\alpha(X_t,\theta^q_0)\big)/(c_1 c_2) - \alpha c_{5,t},
		$ 
		which shows \eqref{eqn:CondEffGapQESRConditions5} and concludes this direction.
		
		
		For the ``if'' implication in statement \ref{statement:QESREfficientChoice}, we assume the conditions (\ref{eqn:CondEffGapQESRConditions1})\,--\,(\ref{eqn:CondEffGapQESRConditions5}).
		%
		Choosing
		$C_{11} =  \alpha (1-\alpha) \left( c_4 + \frac{c_3}{\alpha c_1 c_2} \right) I_{q_1 \times q_1}$, 
		$C_{12} =  \frac{(1-\alpha)}{\alpha c_1 c_3^2} \big( \alpha c_1 c_2 c_4 + c_3 \big) I_{q_1 \times q_2}$,  
		$C_{21} = \frac{(1-\alpha) c_3}{c_1 c_2} I_{q_2 \times q_1}$ and 
		$C_{22} = \frac{c_1 + 1 - \alpha}{\alpha c_1 c_3}  I_{q_2 \times q_2}$, 
		automatically yields $\operatorname{det}(C) \not= 0$, and straight-forward calculations yield that (\ref{QESREfficiencyGapEquation11})\,--\,(\ref{QESREfficiencyGapEquation22}) are satisfied and thus, the identity $A_{g_t, \phi_t} (X_t,\theta_0) = A^\ast_{t,C}(X_t,\theta_0)$ holds almost surely for all $t \in \mathbb{N}$.
		Applying Theorem \ref{theorem:GeneralEfficientA} yields the claim.
	\end{proof}

\section{Details on the connection between loss and identification functions}
\label{sec:Details}

\subsection{Multivariate functionals}

Section \ref{sec:LossIdFunctions} in the main paper provided the arguments why, roughly speaking, there are ``more'' strict identification functions than strictly consistent loss functions for multivariate functionals.
Interestingly,
in extreme cases, e.g., for prediction intervals symmetric around the mean or the median, multivariate functionals may admit strict identification functions, but may even fail to be elicitable at all, due to the said integrability conditions \cite[Proposition 4.12]{FisslerHlavinovaRudloff2019Theory}.
By virtue of the arguments \cite{DFZ_CharMest}, this gap in turn induces a gap between the classes of consistent M- and Z-estimators, which is illustrated by the following Remark \ref{exmp:derivative loss}. 
Note that in order to use Osband's principle in Remark \ref{exmp:derivative loss}, and in line with the discussion in \citet[Chapter 7]{NeweyMcFadden1994}, it is sufficient to assume that the conditional expectation $\E\big[\rho\big(Y_t,m(X_t, \theta)\big)\big|X_t\big]$ is differentiable in $\theta$ almost surely. 
This allows us to treat also losses that are \emph{per se} not differentiable, such as the pinball loss.

\begin{rem}\label{exmp:derivative loss}
	Let Assumption \eqref{ass:unique model} hold for some $k$-dimensional functional $\Gamma\colon \F_{\Y|\X}\to\Xi$ with a strict $\F_{\Y|\X}$-identification function $\ph\colon \R\times\Xi\to\R^k$ and a strictly $\F_{\Y|\X}$-consistent loss function $\rho \colon \R\times \Xi\to\R$. Suppose that $\E\big[\rho\big(Y_t,m(X_t, \theta)\big)\big|X_t\big]$ is differentiable in $\theta$ almost surely.
	Under the richness conditions on $\F_{\Y|\X}$ of \citet[Theorem 3.2]{FisslerZiegel2016}, we then have
	\be{eq:derivative loss}
	\nabla_\theta \E\big[\rho\big(Y_t,m(X_t, \theta)\big)\big|X_t\big] = \nabla_\theta m(X_t, \theta)^\intercal \cdot h\big(m(X_t,\theta)\big)\cdot\E\big[\ph\big(Y_t,m(X_t, \theta)\big)\big|X_t\big] \,,
	\ee
	where $h$ takes values in $\R^{k\times k}$ and the gradient $\nabla_\theta m(X_t,\theta)$ is in $\R^{k\times q}$.
	Comparing \eqref{eq:derivative loss} with \eqref{eqn:UnconditionalIdentificationGeneralForm}, one obtains the identity $A(X_t,\theta) = \nabla_\theta m(X_t, \theta)^\intercal \cdot  h\big(m(X_t,\theta)\big)$ for the instrument matrix. 
	The presence of $h$ and the limitations of the choice of $h$ discussed above yield that there are considerably fewer (strict) model-consistent losses than (strict) moment functions.
\end{rem}

\subsection{Univariate functionals}

If $\Gamma$ is univariate, its mixture-continuity implies that every strictly consistent loss $\rho$ is (strictly) \emph{order-sensitive} meaning that $\xi\mapsto \bar \rho(F,\xi)$ is (strictly) decreasing (increasing) for $\xi\le \Gamma(F)$ (for $\xi\ge \Gamma(F)$); see \citet[Proposition 3]{Nau1985}, \citet[Proposition 1]{Lambert2013} and \citet[Proposition 3.4]{BelliniBignozzi2015}.
The functional $\Gamma$ is mixture-continuous if for all $F_0, F_1\in\F$ such that $(1-\lambda)F_0 + \lambda F_1\in\F$ for all $\lambda \in[0,1]$ the function $[0,1]\ni \lambda\mapsto \Gamma((1-\lambda)F + \lambda G)$ is continuous.
Therefore, the derivative of $\rho$ is an \emph{oriented} identification function in the sense that $\nabla_\xi \bar \rho(F,\xi)\le 0$ ($\ge0$) if $\xi\le \Gamma(F)$ ($\ge\Gamma(F)$).
Intuitively, this excludes the existence of additional local minima of the expected loss, while possible saddle points still remain an issue.
Moreover, Osband's principle in dimension one \eqref{eq:Osband}.
Even if $\rho$ is strictly consistent, its derivative is not necessarily a strict identification function due to possible saddle points of the expected loss. That means even if $\ph$ is strict, $h$ might vanish at some points, see \cite{SteinwartPasinETAL2014} and \citet[p.\ 2117]{NeweyMcFadden1994} for further details and examples.
If $\ph$ is oriented and strict, then $h$ is non-negative.
This means, on the other hand, that we can also start with an oriented strict identification function, multiply it with any positive $h$,
and integrate it. This results in a strictly order-sensitive, and therefore, strictly consistent loss \cite[Theorem 7]{SteinwartPasinETAL2014}. If $\ph$ is not strict and $h$ simply non-negative, the resulting loss is merely consistent. This leads to the fact that there is a one-to-one relation between consistent losses and \emph{oriented} identification functions for $\Gamma$.

\begin{rem}
	If the identification function fails to be oriented, it can still be integrated, but does not yield a consistent score. E.g., $\ph(y,\xi) = (\one\{\xi\ge0\} - \one\{\xi<0\})(\xi-y)$ is a strict identification function for the mean which fails to be oriented. It is easy to check that the integral $\rho(y,\xi) = \int_0^\xi \ph(y,z)\dint z $ is not a strictly consistent loss function for the mean. This identification function constitutes a counterexample to  \citet[Lemma 6]{SteinwartPasinETAL2014}.
\end{rem}

\section{Semiparametric Models for Multiple Moments}
\label{sec:DoubleMomentRegressionModel}

We consider joint semiparametric models for the first and second moments, denoted by $\Gamma_{\textrm{mom}}$, and closely related, joint models for mean and variance, $\Gamma_{(\E, \Var)}$.
Since mean and variance are considered as the most important functionals in classical statistics, the related class of ARMA-GARCH models \citep{Bollerslev1986} is omnipresent in the econometric literature, and is often estimated through M- or Z-estimation.
See also \cite{Spady2018} for joint mean--variance regression models.

\begin{ass}
	\label{ass:Bregmanclass}
	Let $\F_{\Y|\X}$ contain all square integrable distributions such that $(\E, \Var)$ maps surjectively on $\Xi\subseteq\R^2$, which is assumed to be an open and path connected subset of $\R^2$.
\end{ass}
Assumption \eqref{ass:Bregmanclass} is required to characterize the classes of consistent loss and identification functions.
Recalling that $\Gamma_{\textrm{mom}}$ and $\Gamma_{(\E, \Var)}$ are in bijection, we can invoke the \emph{revelation principle} \citep{Osband1985, Gneiting2011} to relate the corresponding strict $\F_{\Y|\X}$-identification and strictly $\F_{\Y|\X}$-consistent loss functions. 
Strict identification functions are given by
\begin{align}
	\label{eqn:DMRPhi0Def}
	\ph_{\textrm{mom}}(y,\xi_1,\xi_2) = \begin{pmatrix}
		\xi_1 - y  \\
		\xi_2- y^2
	\end{pmatrix}, \quad \text{and} \quad
	\ph_{(\E, \Var)}(y,\xi_1,\xi_2) = \begin{pmatrix}
		\xi_1 - y  \\
		\xi_2 + \xi_1^2- y^2
	\end{pmatrix}.
\end{align}

Given Assumptions \eqref{ass:unique model} and \eqref{ass:CharMest}, \citet[Theorem 1]{DFZ_CharMest} yields that the full class of consistent M-estimators at \eqref{eq:M-est} is determined by the full class of (strictly) $\F_{\Y|\X}$-consistent loss functions.
	Using the revelation principle and  following \citet[Proposition 4.4]{FisslerZiegel2016}, under Assumption \eqref{ass:Bregmanclass}, the class of all differentiable (strictly) $\F_{\Y|\X}$-consistent loss functions is given by
	\begin{align}
		\begin{aligned}
			\label{eqn:DoubleMomentModelGeneralLoss1}
			\rho_{\textrm{mom},t} \big(y, \xi_1, \xi_2 \big) &= - \phi_t \big( \xi_1, \xi_2 \big) 
			+ \nabla \phi_t \big( \xi_1, \xi_2 \big)^\intercal
			\begin{pmatrix}
				\xi_1 - y  \\
				\xi_2- y^2
			\end{pmatrix} 
			+ \kappa_t(y),\\
			\rho_{(\E, \Var),t} \big(y, \xi_1, \xi_2 \big) &= - \phi_t \big( \xi_1, \xi_2+\xi_1^2 \big) 
			+ \nabla \phi_t \big( \xi_1, \xi_2 +\xi_1^2\big)^\intercal
			\begin{pmatrix}
				\xi_1 - y  \\
				\xi_2+\xi_1^2- y^2
			\end{pmatrix} 
			+ \kappa_t(y)
		\end{aligned}
	\end{align}
	where $\phi_t: \{(\xi_1, \xi_2)\in\R^2\,|\, \xi_1^2 \le \xi_2\} \to \mathbb{R}$ are (strictly) convex and twice differentiable functions with gradient $\nabla \phi_t$, and $\kappa_t$ is an $\F_{\Y|\X}$-integrable function.
	For any sequence $\Phi = (\phi_t)_{t\in\N}$ of such functions, we denote the corresponding M-estimators defined via \eqref{eq:M-est} by $\widehat \theta^{\textrm{mom}}_{M,T,\Phi}$ and $\widehat \theta^{(\E, \Var)}_{M,T,\Phi}$.
	\begin{proposition}
		\label{prop:JointMomentRegressionEfficiencyBound}
		Under Assumptions \eqref{ass:unique model}, \eqref{ass:PsiUncorrelated} together with Assumptions \eqref{ass:CharMest} and \eqref{ass:Bregmanclass} in Appendix \ref{sec:AddAssumptions}, suppose that the M-estimators $\widehat \theta^{\textrm{mom}}_{M,T,\Phi}$ and $\widehat \theta^{(\E, \Var)}_{M,T,\Phi}$ for the first two moments and for $(\E, \Var)$ are asymptotically normal.
		If almost surely
		\begin{align}\label{eq:eff phi}
			\phi_t(z) = \frac{1}{2} z^\intercal \boldsymbol{\operatorname{Var}}_t \left( \big( Y_t, Y_t^2 \big) \right)^{-1} z \qquad  \forall t\in\N,
		\end{align}	
		then these M-estimators attain the corresponding Z-estimation efficiency bounds in \eqref{eq:Lambda}.
	\end{proposition}

\begin{proof}[Proof of Proposition \ref{prop:JointMomentRegressionEfficiencyBound}]
	We first consider the case of the double moment functional.
	Straight-forward calculations yield that the class of identification functions corresponding to the M-estimators based on loss functions given in (\ref{eqn:DoubleMomentModelGeneralLoss1}) is given by
	\begin{align}\label{eq:psiphi}
		\psi_{\phi_t}(Y_t,X_t,\theta) =  A_{\phi_t}(X_t,\theta) \cdot \varphi_{\textrm{mom}}\big(Y_t, m(X_t,\theta) \big),
	\end{align}
	where $\varphi$ is given in (\ref{eqn:DMRPhi0Def}) and 
	$
	A_{\phi_t}(X_t,\theta) =
	\begin{pmatrix}
		\nabla_{\theta} m_1(X_t,\theta) &
		\nabla_{\theta} m_2(X_t,\theta)
	\end{pmatrix}
	\cdot 
	\nabla^2 \phi_t\big( m(X_t,\theta) \big).$
	Applying Theorem \ref{theorem:GeneralEfficientA} yields that the efficiency bound can be attained by a Z-estimator (and for equivalent M-estimators) if and only if  $A^\ast_{t,C}(X_t,\theta) = C D_t(X_t,\theta)^\intercal S_t(X_t,\theta)^{-1}$ almost surely, where $C$ is some deterministic and non-singular matrix, and where 
	\begin{align*}
		S_t(X_t,\theta_0) = \boldsymbol{\operatorname{Var}}_t \big( (Y_t, Y_t^2 ) \big), \qquad \text{ and } \qquad 
		D_t(X_t,\theta_0) =
		\begin{pmatrix}
			\nabla_{\theta} m_1(X_t,\theta_0)^\intercal \\
			\nabla_{\theta} m_2(X_t,\theta_0)^\intercal
		\end{pmatrix}.
	\end{align*}
	By choosing $C = I_q$ and the strictly convex quadratic form
	$\phi_t(z) = \frac{1}{2} z^\intercal \boldsymbol{\operatorname{Var}}_t  \big( (Y_t, Y_t^2 ) \big)^{-1} z$
	for all $t \in \mathbb{N}$ and for all $z \in \mathbb{R}^2$, this yields that $\nabla^2 \phi_t \big( m(X_t,\theta_0) \big) = \boldsymbol{\operatorname{Var}}_t  \big( (Y_t, Y_t^2 ) \big)^{-1}$ almost surely.
	Consequently, the M-estimator for the double moment regression is able to attain the efficient instrument matrix $A^\ast_{t,C}(X_t,\theta_0)$ (at $\theta_0$) and consequently the Z-estimation efficiency bound.

	For the situation of mean and variance, \eqref{eq:psiphi} takes the form
	\(
	\psi_{\phi}(Y_t,X_t,\theta) =  \tilde A_{t,\phi}(X_t,\theta) \cdot \varphi_{(\E, \Var)}\big(Y_t, m(X_t,\theta) \big),
	\) 
	where $\varphi$ is given in (\ref{eqn:DMRPhi0Def}) and where
	\begin{align*}
		\tilde A_{t,\phi}(X_t,\theta) =
		\begin{pmatrix} \nabla_\theta m_1(X_t,\theta)^\intercal \\ \nabla_\theta v(X_t,\theta)^\intercal + 2 m_1(X_t,\theta) \nabla_\theta m_1(X_t,\theta)^\intercal \end{pmatrix}^\intercal 
		\cdot 
		\nabla^2 \phi_t \begin{pmatrix} m_1(X_t,\theta) \\ v(X_t,\theta) + m_1^2(X_t,\theta) \end{pmatrix}.
	\end{align*}	
	Straight-forward calculations yield that $S_t(X_t,\theta_0)
	= \boldsymbol{\operatorname{Var}}_t \big( (Y_t, Y_t^2 ) \big)$ and 
		\[
		D_t(X_t,\theta_0)
		= \begin{pmatrix}
			\nabla_{\theta} m_1(X_t,\theta_0)^\intercal \\
			\nabla_{\theta}  v (X_t,\theta_0)^\intercal + 2 m_1(X_t,\theta_0) \nabla_{\theta} m_1(X_t,\theta_0)^\intercal
		\end{pmatrix}.
		\]
		Thus, upon using $\R^2\ni z \mapsto \phi_t(z) = \frac{1}{2} z^\intercal \boldsymbol{\operatorname{Var}}_t \big( (Y_t, Y_t^2 ) \big)^{-1} z$, the efficient choice can be attained.	
	\end{proof}
	
	This result is in line with the classical univariate mean regression, where both, M- and Z-estimators are able to attain the Z-estimation efficiency bound and the most efficient Bregman loss is given by the squared loss, weighted with the inverse of the conditional variance.
	Intuitively, this attainability can be explained by the fact that the classes of strictly consistent joint loss functions given in \eqref{eqn:DoubleMomentModelGeneralLoss1} are \textit{relatively large} due to the presence of the general convex function $\phi_t$, being a function in two arguments.

	For the first two moments, this can be illustrated by comparing it to a minimal subclass in this context, namely the class only consisting of the sum of (strictly) consistent loss functions for the individual components, the first and second moment. This arises from \eqref{eqn:DoubleMomentModelGeneralLoss1} when $\phi_t$ takes the additive form $\phi_{\textrm{add}, t}(\xi_1,\xi_2) = \phi_{1,t}(\xi_1)+\phi_{2,t}(\xi_2)$, where $\phi_{i,t}$ are both (strictly) convex.
	Since the Hessian $\nabla^2\phi_{\textrm{add}, t}$ is diagonal, $\phi_{\textrm{add}, t}$ can only take the form in \eqref{eq:eff phi} for the special situation when $Y_t$ and $Y_t^2$ are conditionally uncorrelated.
	Since the class of convex functions on $\R^2$ is far larger than the sum of two convex functions in the individual components, the efficiency bound can be attained.
	For the pair of mean and variance, note that one cannot decompose the loss into a sum of strictly consistent losses for each component, due to the variance failing to be elicitable in general. 
	In particular, this also shows the importance of modelling the variance \emph{jointly} with the mean. 
	However, an additive decomposition of $\phi_t$ as discussed above is also possible for mean and variance.
	
	These results are in stark contrast to the double quantile (DQ) regression framework considered Section \ref{sec:DoubleQuantileRegressionModel}, where the gap arises since the class of strictly consistent losses is relatively small, coinciding with the described minimal additive class. 

\section{Further implications of the gap: Equivariance properties}
\label{sec:implications}

\cite{Patton2011} and \cite{NoldeZiegel2017} provide arguments for the usage of homogeneous loss functions for forecast comparison and ranking. More generally, \cite{FisslerZiegel2019} advocate for loss functions that respect equivariance properties of the functional of interest. Besides homogeneity, a major equivariance property of interest is translation equivariance, or---more generally speaking---linear equivariance; see \cite{FisslerZiegel2019}. Again, we focus on two interesting pairs of functionals, (mean, variance) and $(Q_\a, \ES_\a)$, $\a\in(0,1)$. For any random variable $Y$ with finite second moment and any scalar $c\in\R$, the following identities hold
\begin{align}\label{eq:equivariance}
	\begin{split}
		\big((Q_\a(Y+c), \ES_\a(Y+c)\big) &= \big((Q_\a(Y)+c, \ES_\a(Y)+c\big),\\
		\big(\E[Y+c], \Var(Y+c)\big) &= \big(\E[Y]+c,\Var(Y)\big)\,. 
	\end{split}
\end{align}
Suppose one is to model the functional $(Q_\a, \ES_\a)$ with a parametric model (possibly with joint model parameters) of the form $m(X,\theta)=\big(q_\alpha(X,\theta), e_\alpha(X,\theta)\big)$, where $\theta= \big(\theta^{(1)}, \ldots, \theta^{(q)}\big)\in\Theta\subseteq\R^q$, with intercept parameters, say
\[
\begin{pmatrix}
	q_\alpha(X,\theta) \\
	e_\alpha(X,\theta)
\end{pmatrix}
= 
\begin{pmatrix}
	\theta^{(1)} + \widetilde q_\alpha (X,\theta^{(3)},\ldots, \theta^{(q)})\\
	\theta^{(2)} + \widetilde e_\alpha (X,\theta^{(3)},\ldots, \theta^{(q)})
\end{pmatrix}.
\]
Then, under Assumption \eqref{ass:unique model}, the correctly specified parameter $\theta_0$ has the following equivariance property for $(Y,X)\in\Z$ and $c\in\R$ such that $(Y+c,X)\in\Z$: 
\begin{align}\label{eq:equivariance2}
	\theta_0^{(j)}(F_{(Y+c,X)}) = 
	\begin{cases}
		\theta_0^{(j)}(F_{(Y,X)}) + c, & \text{for } j=1,2, \\
		\theta_0^{(j)}(F_{(Y,X)}), & \text{for } j=3, \ldots, q.
	\end{cases}
\end{align}
Similar results apply to the pair (mean, variance), where, of course, the intercept transformation only appears in the mean-component.

Similarly, given data $(\boldsymbol{Y}, \boldsymbol{X}) = (Y_t, X_t)_{t=1, \ldots, T}$, it would be desirable to find a similar translation equivariance property for an estimator $\widehat \theta_T = \widehat \theta_T (\boldsymbol{Y}, \boldsymbol{X} )$:
\begin{align}\label{eq:equiv}
	\widehat \theta_T^{(j)}(\boldsymbol{Y} + c, \boldsymbol{X}) = 
	\begin{cases}
		\widehat \theta_T^{(j)}(\boldsymbol{Y}, \boldsymbol{X}) + c, & \text{for } j=1,2, \\
		\widehat \theta_T^{(j)}(\boldsymbol{Y}, \boldsymbol{X}), & \text{for } j=3, \ldots, q.
	\end{cases}
\end{align}
Under Assumption \eqref{ass:unique model} of a correctly specified model, \eqref{eq:equiv} holds for the probability limit of any consistent estimator. However, in finite samples or under model misspecification, it may well fail unless there is some additional structure in the estimator.
For example, the OLS-estimator clearly satisfies \eqref{eq:equivariance2} and \eqref{eq:equiv}, relying on the fact that the squared loss $\rho(y,\xi) = \frac12(y-\xi)^2$ is translation invariant. Also, the corresponding Z-estimator is translation equivariant, since the standard identification function $\ph(y,\xi) = y-\xi$ is translation invariant and the instrument matrix $A(X,\theta) = X$ is independent of $\theta$.

It turns out that both two-dimensional functionals in \eqref{eq:equivariance} possess strict identification functions that respect the respective equivariance properties described there, namely
\begin{align*}
	\ph_{(Q_\alpha, \ES_\alpha)}(y,\xi_1, \xi_2) &=
	\begin{pmatrix}
		\one\{y\le \xi_1\} -\a \\ \xi_2 +\tfrac{1}{\a}\one\{y\le \xi_1\} (\xi_1- y) - \xi_1
	\end{pmatrix},
	\\
	\ph_{(\E, \Var)}(y,\xi_1, \xi_2) &= 
	\begin{pmatrix}
		\xi_1 - y \\ \xi_2 - (y-\xi_1)^2
	\end{pmatrix}.
\end{align*}
		Using instrument matrices which are independent of $\theta$, they induce Z-estimators which obey the translation equivariance in their intercept components.
		However, Propositions 4.9 and 4.10 in \cite{FisslerZiegel2019} ascertain that for both functional pairs, there are no strictly consistent loss function with these equivariance properties---at least under general and realistic assumptions.
		This rules out the existence of corresponding M-estimators with this property---another manifestation of the gap between these two classes of estimators.
		
		%

\section{Details on the Efficiency in Double Quantile Models}
		\label{sec:GapSepModels}
		
		Theorem \ref{theorem:DoubleQuantileRegressionEfficiencyBound} part \ref{statement:DQREfficientChoice}, which is based on Theorem \ref{theorem:GeneralEfficientA} part \ref{theorem:GeneralEfficientAPart3}, merely implies that the difference of the asymptotic covariances between any M-estimator and the joint efficient Z-estimator is positive semi-definite with at least one positive eigenvalue.
		One could plausibly suspect that this is purely caused by differing off-diagonal ``covariance'' terms, and that the diagonal entries---i.e., the estimation ``variances'' of the parameters---coincide (at least for the pseudo-efficient M-estimator).
		The following example illustrates that the efficiency gap also arises for the diagonal entries.

		
		To simplify the exposition, we consider a stationary process $(Y_t,X_t)_{t\in\mathbb N}$ with a univariate $X_t \in \mathbb{R}$, and two linear ``slope only'' models $q_\alpha(X_t,\theta^\alpha) = X_t\theta^\alpha$ and $q_\beta(X_t,\theta^\beta) = X_t\theta^\beta$.
		Using the shorthands $f_\alpha:= f_t\big(q_\alpha(X_t,\theta_0^\alpha)\big)$ and $f_\beta:= f_t\big(q_\beta(X_t,\theta_0^\beta)\big)$, the asymptotic variance of the individual efficient Z-estimator for the $\theta^\alpha$ component is $\alpha(1-\alpha)/\E[f_\alpha^2 X_t^2]$.
		On the other hand, the asymptotic variance of the joint efficient Z-estimator for the $\theta^\alpha$ component is 
		\[
		\frac{\alpha(1-\alpha)}{\E[f_\alpha^2X_t^2]} \; \times \; 
		\frac{\alpha(1-\alpha)\beta(1-\beta) - \alpha^2(1-\beta)^2}{\alpha(1-\alpha)\beta(1-\beta) - \alpha^2(1-\beta)^2 \E[f_\alpha f_\beta X_t^2]^2 /\big(\E[f_\alpha^2  X_t^2]\E[ f_\beta^2 X_t^2]\big)}\,,
		\] 
		which is generally smaller than  
		\(
		\alpha(1-\alpha)/\E[f_\alpha^2X_t^2],
		\)
		since the Cauchy--Schwartz inequality implies that $\E[f_\alpha f_\beta X_t^2]^2 /\big(\E[f_\alpha^2  X_t^2]\E[ f_\beta^2 X_t^2]\big)\le 1$.
		The latter holds with equality if and only if $f_\alpha  X_t$ and $f_\beta X_t$ are colinear almost surely, once again stressing the importance of condition \eqref{eqn:CondEffGapDQRConditions1}.
		This effect can numerically be observed in our simulations in panels A of Table \ref{tab:DQRSimSD} by comparing the efficiency bound to the pseudo-efficient choices based on the choices $F_t(\xi)$.

\section{The Two-Step Estimation Efficiency Bound}
\label{sec:TwoStepEfficiency}	

In related work, \cite{Barendse2020} considers efficiency among the class of \textit{two-step} estimators of semiparametric models for the quantile and ES with separated parameters.
These two-step estimators utilize a quantile regression to estimate the quantile parameters in the first step and a restricted and weighted least squares estimator in the second step for the model parameters of the conditional ES.
The author considers efficiency among the possible estimation weights from the second step weighted least squares estimator, see \cite{Barendse2020} for details.
This procedure amounts to efficiency of the ES parameters \textit{in isolation}, which generally results in more restrictive efficiency bounds than efficiency of the joint model parameters considered in this article.

In our notation, the class of two-step estimators can be characterized by the general form \eqref{eqn:UnconditionalIdentificationGeneralForm}, the identification functions in \eqref{eqn:QESIdFunction} and the class of instrument matrices
\begin{align}
	\label{eqn:EfficientAtMatrixBarendse}
	A_t^\dagger(X_t, \theta_0) = 
	\begin{pmatrix}
		\nabla_{\theta^q} q_\alpha(X_t,\theta_0^q) & 0   \\
		0 & \phi_t''\big(e_\alpha(X_t,\theta^e_0)\big)  \nabla_{\theta^e} e_\alpha(X_t,\theta_0^e)
	\end{pmatrix}.
\end{align}
For these estimators, the theory of \cite{ProkhorovSchmidt2009}, \cite{Bartalotti2013} can be used to establish that the asymptotic distribution of the joint Z- and the two-step estimators coincide.
Consequently, the family of two-step estimators of \cite{Barendse2020} form a subclass of the general class of Z-estimators we consider in this article.
Hence, it follows that the resulting two-step estimation efficiency bound is no smaller than the general Z-estimation efficiency bound of Theorem \ref{theorem:GeneralEfficientA}.
While these two bounds can coincide in special situations, they generally do not as illustrated in the following.

For the special case of location-scale models with stationary innovations $(\varepsilon_t)_{t \in \mathbb{N}}$ discussed in Section \ref{sec:QESModelExamples}, 
the efficient weights of \cite{Barendse2020} coincide with the choice of $\phi_t''$ implied by a combination of (\ref{eqn:CondEffGapQESRConditions1}) and (\ref{eqn:CondEffGapQESRConditions3}) in Theorem \ref{theorem:QuantileESRegressionEfficiencyGap}.
This illustrates that for this special case, and in terms of the ES parameters, $\theta^e$, considered in isolation, the efficient two-step and the efficient M-estimator are equally efficient; see \citet[Section 4.4]{Barendse2020}.
However, if efficiency is considered for the full parameter vector, $\theta$, the two-step estimator using the instrument matrix (\ref{eqn:EfficientAtMatrixBarendse}) is generally less efficient, which is caused by the inefficient choice of the first-step standard quantile regression.
In this special case, joint efficiency could be guaranteed by employing an efficient quantile regression estimator in the first step, see e.g., \cite{Komunjer2010a, Komunjer2010b}.
We refer to the simulation results of Sections \ref{sec:QESRSimStudy} and \ref{app:AdditionalTables} and in particular to Tables \ref{tab:QESRSim25Percent} and \ref{tab:QESRSimSDTrue} for a numerical illustration.

More generally, \citet[Section 4.4]{Barendse2020}  illustrates that, taken in isolation, the ES specific asymptotic sub-covariance matrix of the M-estimator $\widehat \theta^e$ is subject to his two-step efficiency bound.
However, this does not hold if one considers the entire covariance matrix of the joint model parameters for the quantile and ES.
This can be observed by comparing (\ref{eqn:EfficientAtMatrixBarendse}) with the efficient instrument matrix $A_t^\ast$ given in (\ref{eqn:QESREfficientMatrixD}) and (\ref{eqn:QESREfficientMatrixS}):
while $A_{t,C}^\ast$ generally requires non-zero off-diagonal blocks, the matrix $A_t^\dagger$ is restricted to a block diagonal matrix with zero off-diagonal blocks.

Recall that under the gradient condition that $\nabla_{\theta^q} q_\alpha(X_t,\theta_0^q) = \nabla_{\theta^e} e_\alpha(X_t,\theta_0^e)$ for all $t \in \mathbb{N}$ almost surely, part \ref{statement:QESREfficientChoice} of Theorem \ref{theorem:QuantileESRegressionEfficiencyGap} implies that if conditions (\ref{eqn:CondEffGapQESRConditions1}) or (\ref{eqn:CondEffGapQESRConditions2}) fail to hold, the M-estimator cannot attain the Z-estimation efficiency bound.
As  \citet[Section 4.4]{Barendse2020} informally shows that the two-step estimators are equivalent to the class of M-estimators in terms of the efficiency of the ES parameters, this illustrates that the two-step estimators also cannot attain the Z-estimation efficiency bound in this setting. (Formally, relating $A_t^\dagger(X_t, \theta_0)$ in (\ref{eqn:EfficientAtMatrixBarendse}) to the efficient choice $A_t^\ast(X_t, \theta_0)$ and employing Theorem \ref{theorem:GeneralEfficientA} as in the proof of Theorem \ref{theorem:QuantileESRegressionEfficiencyGap} yields the desired result.)	
Besides supporting our claim of an existing \textit{efficiency gap} for the joint quantile and ES models, this illustrates that the two-step efficiency bound of \cite{Barendse2020} does generally not coincide with the general Z-estimation efficiency bound of \cite{Hansen1985}, \cite{Chamberlain1987}, and \cite{Newey1993}.

We illustrate the theoretical considerations of this section numerically through the simulation setup of Sections \ref{sec:QESRSimStudy} and \ref{app:AdditionalTables}.
In Panel A of Table \ref{tab:QESRSim25Percent} and in Panels A and B of Table \ref{tab:QESRSimSDTrue}, 
we additionally report the two-step efficiency bound in the line denoted ``Barendse Bound". 
For the homoskedastic innovations and for the ES specific parameters, the two-step efficiency bound coincides with the Z-estimation efficiency bound, while it does not for the quantile parameters.
This is primarily caused by the inefficient first-step quantile estimation---using an efficient quantile estimator (based on $g_t(\xi_1) = F_t(\xi_1)$) would equate both efficiency bounds in the homoskedastic case.
In contrast, in the heteroskedastic case, the two-step efficiency bound is considerably larger than the Z-estimation efficiency bound for all four considered parameters.
Interestingly, the choice of $\phi^\text{eff2}$ motivated by this two-step estimation efficiency bound exhibits equally efficient ES parameters while the quantile parameters show larger standard deviations.

\section{Connections to the semiparametric efficiency bound}
	\label{sec:SemipEffBound}
		The main focus of this article lies on the Z-estimation efficiency bound for conditional moment restrictions of \cite{Hansen1985}, \cite{Chamberlain1987} and \cite{Newey1993}.
		In the context of i.i.d.\ processes and differentiable moment conditions, \cite{Chamberlain1987} shows that this bound coincides with the general semiparametric efficiency bound in the sense of a least favorable submodel of \cite{Stein1956}; c.f.\ \cite{Newey1990} and \cite{BKRW1998book} for surveys on this matter and \cite{Ackerberg2014}, \cite{Jankova2018}, \cite{Hristache2016},  \cite{Komunjer2010a} for some recent progress.
		
		The definition of the semiparametric efficiency bound builds on the idea that the data stems from a \textit{parametric submodel}, i.e., a parametric model which completely specifies the full distribution, contains the correctly specified model, and satisfies the semiparametric model assumption. 
		E.g., if we consider a semiparametric model for the conditional mean, we do not make any assumptions about the exact conditional distribution beyond the mean assumption.
		Any model which parametrizes the full conditional distribution (e.g., a normal distribution with parameterized variance) is such a \textit{parametric submodel}.
		Estimation of any parametric submodel is subject to the classical Cram\'{e}r--Rao efficiency bound, which can be attained, e.g., by maximum likelihood estimation using the true parametric distribution, dispensing with a discussion of superefficient estimators.
		For any parametric submodel, a consistent and asymptotically normal semiparametric estimator is contained in the class of estimators for this parametric submodel and thus, it is subject to the parametric Cram\'{e}r--Rao efficiency bound.
		Consequently, any semiparametric estimator has an asymptotic variance which is no smaller than the Cram\'{e}r--Rao bound for \textit{any} parametric submodel.
		Hence, the \textit{semiparametric efficiency bound} is defined as the supremum of the  Cram\'{e}r--Rao bounds of all parametric submodels.
		
		The results of this paper concerning efficient estimation are derived with respect to the Z-estimation efficiency bound of \cite{Hansen1985}, \cite{Chamberlain1987}, and \cite{Newey1993}.
		In applications to smooth objective functions and i.i.d.\ processes, the result of \cite{Chamberlain1987} can be used to equate these two bounds.
		However, as we are not aware of a general relation of these bounds for non-i.i.d.\ processes, we cannot preclude that the semiparametric efficiency bound is strictly smaller (in the Loewner order) than the Z-estimation efficiency bound in certain situations.
		Consequently, all following assertions are stated in relation to the Z-estimation efficiency bound.
		This does not affect our main conclusion in terms of efficient estimation: When the M-estimator cannot attain the Z-estimation efficiency bound, it also cannot attain the semiparametric efficiency bound, irrespectively of whether these quantities coincide.

\section{Identification of the Efficient Z-estimator for double quantile models}
\label{app:IDZestDQRModels}

Following \citet{DFZ_CharMest}, strict model consistency can directly be obtained by employing strictly consistent loss functions and a no-perfect collinearity condition of the model gradient.
In contrast, this is more involved for the Z-estimator.
Thus, the following proposition shows strict model identification for an efficient Z-estimator and for a large class of models.

Note that Theorem \ref{theorem:GeneralEfficientA} asserts that the Z-estimator is efficient based on any choice $A^*_{t,C}(X_t,\theta)$ of instrument matrix such that $A^*_{t,C}(X_t,\theta_0) = CD_t(X_t,\theta_0)^\intercal S_t(X_t,\theta_0)^{-1}$, see \eqref{eqn::GeneralEfficientAMatrix}. This means we only have a condition on $A^*_{t,C}(X_t,\theta)$ for $\theta= \theta_0$. 
To come up with such a matrix, there are two straight forward ways how to guarantee this. First, we might set $A^*_{t,C}(X_t,\theta) = CD_t(X_t,\theta)^\intercal S_t(X_t,\theta)^{-1}$, and second, we might choose $A^*_{t,C}(X_t,\theta)$ to be constant in $\theta$ and equal to $CD_t(X_t,\theta_0)^\intercal S_t(X_t,\theta_0)^{-1}$. For practical purposes, the latter situation is often hard or infeasible to implement, since it usually requires knowledge of the unknown true parameter $\theta_0$ (and additional quantities of the conditional distribution $F_t$). 

For the particular situation of the double quantile model, using the canonical identification function $\ph$ given in Running Example \eqref{RunningExmp3}, $S_t(X_t,\theta_0)$ takes the form \eqref{eqn:DQREfficientSMainText}, which means it is entirely independent of \emph{any} knowledge on the underlying DGP whatsoever. This makes the latter choice attractive and reasonably feasible.

\begin{proposition}
	\label{prop:UniqueIdentificationDQR}
	We assume that 
	(a) the double quantile model is linear with separated parameters, i.e., $Q_\alpha(Y_t|X_t) = q_\alpha(X_t,\theta_0^\alpha) = X_t^\intercal \theta_0^\alpha$ and $Q_\beta(Y_t|X_t) = q_\alpha(X_t,\theta_0^\beta)= X_t^\intercal \theta_0^\beta$,
	such that $\theta_0 = (\theta_0^\alpha, \theta_0^\beta)\in\interior(\Theta)$,
	(b) for all $t\in\N$, $F_t$ is differentiable with a strictly positive derivative  $f_t$,
	and (c), there exists a possibly time-dependent deterministic constant $c_t > 0$, such that $f_t\big(q_\alpha(X_t,\theta_0^\alpha)\big) = c_t  f_t\big(q_\beta(X_t,\theta_0^\beta)\big)$ almost surely.
	Then, the moment function of the efficient Z-estimator of the DQR model is a strict $\F_\Z$-identification function for $\theta_0$, i.e., it holds that
	\begin{align*}
		\mathbb{E} \left[ A_t^\ast(X_t, \theta_0) \varphi \big( Y_t, m(X_t, \theta) \big)  \right]	= 0 
		\quad \Longleftrightarrow \quad
		\theta = \theta_0,
	\end{align*}
	where $A_t^\ast(X_t, \theta_0)$ is given in \eqref{eqn::GeneralEfficientAMatrix}.
\end{proposition}
At the cost of some more tedious notation, Proposition \ref{prop:UniqueIdentificationDQR} can be generalized to the situation of linear models with not necessarily separated parameters, so long as there is at least one component that is used for modelling one quantile only, respectively. E.g., in a simple linear regression model, the two quantile models might have the same slope, but a different intercept, or vice versa, they might have the same intercept, but a different slope. 
Generalising the assertion much beyond linear models seems to be difficult due to the application of the mean value theorem in the proof.

\begin{proof}[Proof of Proposition \ref{prop:UniqueIdentificationDQR}]
	It holds that 
	\(
	\mathbb{E} \left[ A_t^\ast(X_t, \theta_0) \varphi \big( Y_t, m(X_t, \theta_0) \big)  \right]	= 0
	\)
	since we have that  $\E\left[\varphi \big( Y_t, m(X_t, \theta_0) \big)\big|X_t\right]=0$.
	The reverse direction is a little more involved.
	For this, straight-forward calculations yield that for any $\theta\in\Theta$
	\begin{align*}
		\mathbb{E} \left[ A_t^\ast(X_t, \theta_0) \varphi \big( Y_t, m(X_t, \theta) \big)  \right]
		= \mathbb{E} \left[ U_1 \nabla_\theta q_\alpha(X_t,\theta_0^\alpha)  + U_2 \nabla_\theta q_\beta(X_t,\theta_0^\beta) \right], 
	\end{align*}
	where the scalar and $\sigma(X_t)$-measurable random variables $U_1$ and $U_2$ are given by
	\begin{align*}
		U_1 &= \frac{ f_t\big(q_\alpha(X_t,\theta_0^\alpha)\big) }{\alpha (1-\alpha) \beta - \alpha^2 (1-\beta)} \big( \beta a - \alpha b \big) \qquad \text { and } \\ 
		U_2 &= \frac{ f_t\big(q_\beta(X_t,\theta_0^\beta)\big) }{\beta (1-\alpha) (1-\beta) - \alpha (1-\beta)^2} \big( -(1-\beta) a + (1- \alpha) b \big),
	\end{align*}
	with $a = F_t\big(q_\alpha(X_t,\theta^\alpha)\big) - \alpha$ and $b = F_t\big(q_\beta(X_t,\theta^\beta)\big) - \beta$.
	As $\nabla_\theta q_\alpha(X_t,\theta_0^\alpha) = \begin{pmatrix}
		X_t \\0
	\end{pmatrix}$ and $\nabla_\theta q_\beta(X_t,\theta_0^\beta) = \begin{pmatrix}
		0\\X_t 
	\end{pmatrix}$,
	it holds that $\mathbb{E} \left[ A_t^\ast(X_t, \theta_0) \varphi \big( Y_t, m(X_t, \theta) \big)  \right] = 0$ if and only if
	\begin{align}\label{eqn:AppendixC}
		\mathbb{E} \left[  f_t\big(q_\alpha(X_t,\theta_0^\alpha)\big) 
		\big( \beta a - \alpha b \big) X_t \right] = 0 \ \text{ and } \
		\mathbb{E} \left[  f_t\big(q_\beta(X_t,\theta_0^\beta)\big) \big( (1-\beta) a - (1- \alpha) b  \big) X_t \right] = 0.
	\end{align} 
	As $f_t\big(q_\alpha(X_t,\theta_0^\alpha)\big) = c_t  f_t\big(q_\beta(X_t,\theta_0^\beta)\big)$ almost surely by assumption (where $c_t$ is deterministic), this implies that 
	\begin{align*}
		\beta \mathbb{E} \big[  f_t\big(q_\alpha(X_t,\theta_0^\alpha)\big)  a X_t  \big] - \alpha  \mathbb{E} \big[  f_t\big(q_\alpha(X_t,\theta_0^\alpha)\big) b X_t \big] &= 0 \qquad \text{ and } \\
		c_t (1-\beta) \mathbb{E} \big[  f_t\big(q_\alpha(X_t,\theta_0^\alpha)\big) a X_t \big]  -  c_t (1-\alpha) \mathbb{E} \big[  f_t\big(q_\alpha(X_t,\theta_0^\alpha)\big) b X_t  \big] &= 0.
	\end{align*} 	
	Solving this system of equations, where we exploit that $c_t\neq0$, and combining it with \eqref{eqn:AppendixC} and the fact that $\alpha\neq \beta$, we arrive at 
	\begin{align}
		\label{eqn:ExpectationZeroContradiction}
		\mathbb{E} \left[  f_t\big(q_\alpha(X_t,\theta_0^\alpha)\big) a X_t \right] = 0 \quad \text{ and } \quad 
		\mathbb{E} \left[  f_t\big(q_\alpha(X_t,\theta_0^\alpha)\big) b X_t \right] = 0.
	\end{align} 
	We now proceed by a proof through contradiction with an argument similar as in \cite{DimiBayer2019}.
	For this, assume that $\theta \not= \theta_0$.
	Using the zero-condition in (\ref{eqn:ExpectationZeroContradiction}), we get
	\begin{align*}
		0 &= \mathbb{E} \left[  f_t\big(q_\alpha(X_t,\theta_0^\alpha)\big) a X_t^\intercal \right] \big( \theta^\alpha - \theta_0^\alpha \big) \\
		&= \mathbb{E} \left[  f_t\big(q_\alpha(X_t,\theta_0^\alpha)\big)  
		X_t^\intercal \big(  \theta^\alpha - \theta_0^\alpha \big)  
		\big(   F_t\big(q_\alpha(X_t,\theta^\alpha)\big) - F_t\big(q_\alpha(X_t,\theta_0^\alpha)\big)  \big) \right] \\
		&= \mathbb{E} \left[  f_t\big(q_\alpha(X_t,\theta_0^\alpha)\big) 
		f_t\big(q_\alpha(X_t,\tilde \theta^\alpha)\big) 
		\Big(X_t^\intercal \big(  \theta^\alpha - \theta_0^\alpha \big)  \Big)^2
		\right],
	\end{align*}
	where we have used the mean value theorem and the linearity of the model to obtain the last identity and where $\tilde \theta^\alpha = (1-\lambda) \theta_0^\alpha + \lambda \theta^\alpha$ for some $\lambda\in[0,1]$. 
	By assumption, the density is strictly positive such that we can conclude that $\P\big(X_t^\intercal  (  \theta^\alpha - \theta_0^\alpha ) =0\big)=1$. 
	Then, due to Assumption \eqref{ass:unique model}, it must hold that $\theta^\alpha = \theta_0^\alpha$. 
	Employing a similar argument to $\theta^\beta$ yields that $\theta^\beta= \theta_0^\beta$, which concludes this proof.
\end{proof}

\section{Additional Simulation Results}
\label{app:AdditionalTables}

In this section, we report simulation results for additional probability levels for the double quantile, and the joint quantile and ES models discussed in Sections \ref{sec:DoubleQuantileRegressionModel} and \ref{sec:QuantileESRegressionModel}.
Specifically, Tables \ref{tab:DQRSimSDTrue} 
and \ref{tab:DQRSim_Joint_InterceptSDTrue} present results for the double quantile model and Table \ref{tab:QESRSimSDTrue} for the joint quantile and ES model.
The format of these tables follows Tables \ref{tab:DQRSimSD} and \ref{tab:QESRSim25Percent} from the main document.

\begin{table}[p]
	\caption{Asymptotic Standard Deviations of Separated Parameter Double Quantile Models}
	\label{tab:DQRSimSDTrue}
	\tiny
	\centering
	\begin{tabularx}{\linewidth}{X @{\hspace{0.5cm}} lrrrr @{\hspace{0.1cm}} lrrrr @{\hspace{0.1cm}} lrrrr} 
		\addlinespace
		\toprule
		& & \multicolumn{4}{c}{ (a) Homoskedastic } && \multicolumn{4}{c}{ (b) Heteroskedastic $t$ }  && \multicolumn{4}{c}{ (c) Heteroskedastic $\mathcal{SN}$ }    \\
		\cmidrule(lr){3-6} \cmidrule(lr){8-11} \cmidrule(lr){13-16} 
		$g_t(\xi)$ & & $\theta_1$ & $\theta_2$ & $\theta_3$ & $\theta_4$ & &  $\theta_1$ & $\theta_2$ & $\theta_3$ & $\theta_4$ & & $\theta_1$ & $\theta_2$ & $\theta_3$ & $\theta_4$ \\
		\midrule
		\\
		& &\multicolumn{14}{c}{Panel A: $(\alpha, \beta) = (0.5\%, 1\%)$} \\
		\cmidrule(lr){3-16}
		$\xi$ &   &  1.043 &  1.041 &  1.043 &  1.041 &   &  1.134 &  1.155 &  1.089 &  1.106 &   &  1.179 &  1.197 &  1.120 &  1.136\\
		$\exp(\xi)$ &   &  1.024 &  1.034 &  1.011 &  1.016 &   &  2.964 &  4.122 &  1.937 &  2.350 &   &  1.503 &  1.587 &  1.309 &  1.357\\
		$\log(\xi)$ &   &  1.066 &  1.064 &  1.061 &  1.059 &   & 64.828 & 70.114 &  1.099 &  1.119 &   &  1.161 &  1.178 &  1.112 &  1.128\\
		$F_{\text{Log}}(\xi)$ &   &  1.043 &  1.041 &  1.043 &  1.041 &   &  1.168 &  1.202 &  1.091 &  1.107 &   &  1.179 &  1.197 &  1.120 &  1.136\\
		$F_t(\xi)$  &   &  1.000 &  1.000 &  1.000 &  1.000 &   &  1.015 &  1.014 &  1.015 &  1.014 &   &  1.009 &  1.009 &  1.009 &  1.009\\
		\cmidrule(lr){3-16}
		Eff.~Bound &   & 18.321 & 10.104 & 14.022 &  7.733 &   & 75.811 & 33.848 & 55.029 & 25.315 &   & 11.893 &  5.249 & 10.762 &  4.909\\
		\midrule
		\\
		& &\multicolumn{14}{c}{Panel B: $(\alpha, \beta) = (5\%, 10\%)$} \\
		\cmidrule(lr){3-16}
		$\xi$ &   &  1.043 &  1.042 &  1.043 &  1.042 &   &  1.034 &  1.040 &  1.040 &  1.040 &   &  1.040 &  1.049 &  1.040 &  1.042\\
		$\exp(\xi)$ &   &  1.004 &  1.003 &  1.020 &  1.019 &   &  1.050 &  1.065 &  1.016 &  1.016 &   &  1.044 &  1.052 &  1.027 &  1.029\\
		$\log(\xi)$ &   &  1.051 &  1.049 &  1.047 &  1.045 &   &  1.047 &  1.053 &  1.046 &  1.046 &   &  1.043 &  1.052 &  1.042 &  1.044\\
		$F_{\text{Log}}(\xi)$ &   &  1.043 &  1.042 &  1.043 &  1.042 &   &  1.034 &  1.040 &  1.040 &  1.040 &   &  1.040 &  1.049 &  1.040 &  1.042\\
		$F_t(\xi)$  &   &  1.000 &  1.000 &  1.000 &  1.000 &   &  1.009 &  1.010 &  1.009 &  1.010 &   &  1.010 &  1.013 &  1.011 &  1.014\\
		\cmidrule(lr){3-16}
		Eff.~Bound &   &  7.895 &  4.337 &  6.387 &  3.508 &   & 16.453 &  8.585 & 11.536 &  6.145 &   &  7.569 &  3.821 &  6.784 &  3.577\\
		\midrule
		\\
		& &\multicolumn{14}{c}{Panel C: $(\alpha, \beta) =  (25\%, 50\%)$} \\
		\cmidrule(lr){3-16}
		$\xi$ &   & 1.042 & 1.040 & 1.042 & 1.040 &   & 1.046 & 1.044 & 1.044 & 1.042 &   & 1.054 & 1.053 & 1.058 & 1.057\\
		$\exp(\xi)$ &   & 1.080 & 1.076 & 1.206 & 1.187 &   & 1.073 & 1.069 & 1.213 & 1.194 &   & 1.092 & 1.089 & 1.227 & 1.208\\
		$\log(\xi)$ &   & 1.039 & 1.037 & 1.034 & 1.032 &   & 1.043 & 1.041 & 1.035 & 1.034 &   & 1.051 & 1.050 & 1.049 & 1.048\\
		$F_{\text{Log}}(\xi)$ &   & 1.042 & 1.040 & 1.042 & 1.040 &   & 1.046 & 1.044 & 1.044 & 1.042 &   & 1.054 & 1.053 & 1.058 & 1.057\\
		$F_t(\xi)$  &   & 1.000 & 1.000 & 1.000 & 1.000 &   & 1.001 & 1.001 & 1.001 & 1.001 &   & 1.009 & 1.009 & 1.009 & 1.009\\
		\cmidrule(lr){3-16}
		Eff.~Bound &   & 5.091 & 2.797 & 4.683 & 2.573 &   & 6.031 & 3.306 & 5.206 & 2.854 &   & 4.491 & 2.473 & 4.973 & 2.759\\
		\midrule
		\\
		& &\multicolumn{14}{c}{Panel D: $(\alpha, \beta) =  (1\%, 99\%)$} \\
		\cmidrule(lr){3-16}
		$\xi$ &   &  1.042 &  1.041 &  1.042 &  1.041 &   &  1.118 &  1.116 &  1.118 &  1.116 &   &  2.075 &  1.590 &  1.047 &  1.039\\
		$\exp(\xi)$ &   &  1.010 &  1.015 &  2.206 &  1.956 &   &  1.354 &  1.555 &  4.280 &  3.441 &   &  1.488 &  1.285 &  2.006 &  1.824\\
		$\log(\xi)$ &   &  1.061 &  1.058 &  1.022 &  1.021 &   &  1.245 &  1.238 &  1.089 &  1.088 &   &  2.189 &  1.645 &  1.036 &  1.029\\
		$F_{\text{Log}}(\xi)$ &   &  1.042 &  1.040 &  1.042 &  1.041 &   &  1.114 &  1.111 &  1.118 &  1.116 &   &  2.074 &  1.590 &  1.047 &  1.039\\
		$F_t(\xi)$  &   &  1.000 &  1.000 &  1.000 &  1.000 &   &  1.000 &  1.000 &  1.000 &  1.000 &   &  1.000 &  1.000 &  1.000 &  1.000\\
		\cmidrule(lr){3-16}
		Eff.~Bound &   & 13.918 &  7.674 & 13.918 &  7.674 &   & 35.871 & 19.775 & 35.871 & 19.775 &   &  3.799 &  4.523 & 16.316 &  8.576\\
		\midrule
		\\
		& &\multicolumn{14}{c}{Panel E: $(\alpha, \beta) =  (10\%, 99\%)$} \\
		\cmidrule(lr){3-16}
		$\xi$ &   &  1.044 &  1.042 &  1.044 &  1.042 &   &  1.029 &  1.030 &  1.238 &  1.197 &   &  1.163 &  1.141 &  1.039 &  1.036\\
		$\exp(\xi)$ &   &  1.021 &  1.020 &  2.167 &  1.930 &   &  1.003 &  1.003 &  5.754 &  4.199 &   &  1.125 &  1.107 &  2.058 &  1.864\\
		$\log(\xi)$ &   &  1.047 &  1.045 &  1.023 &  1.022 &   &  1.036 &  1.036 &  1.183 &  1.149 &   &  1.167 &  1.145 &  1.024 &  1.021\\
		$F_{\text{Log}}(\xi)$ &   &  1.044 &  1.042 &  1.044 &  1.042 &   &  1.029 &  1.030 &  1.238 &  1.197 &   &  1.163 &  1.141 &  1.039 &  1.036\\
		$F_t(\xi)$  &   &  1.000 &  1.000 &  1.000 &  1.000 &   &  1.000 &  1.000 &  1.000 &  1.000 &   &  1.000 &  1.000 &  1.000 &  1.000\\
		\cmidrule(lr){3-16}
		Eff.~Bound &   &  6.493 &  3.559 & 14.180 &  7.772 &   & 10.978 &  5.874 & 29.809 & 18.299 &   &  3.713 &  2.186 & 16.410 &  8.699\\
		\bottomrule
		\addlinespace
		\multicolumn{16}{p{.97\linewidth}}{\scriptsize 
			This table presents the (approximated) asymptotic standard deviations for semiparametric double quantile models with separated model parameters given in (\ref{eqn:DQR_Models_SepParameters}) in the main article at different probability levels in the horizontal panels.
			The rows titled ``Eff.\ Bound'' report the raw standard deviations whereas the remaining rows report the relative standard deviations compared to the efficiency bound.
			Results for the three residual distributions described in Section \ref{sec:DQRSimStudy} are reported in the three vertical panels of the table.
			We furthermore consider four classical choices of $g_t(\xi)$ together with the (pseudo-)efficient choice $F_t(\xi)$ and the Z-estimation efficiency bound.
		}
	\end{tabularx}
\end{table}

\begin{table}[p]
	\caption{Asymptotic Standard Deviations of Joint Parameter Double Quantile Models}
	\label{tab:DQRSim_Joint_InterceptSDTrue}
	\scriptsize
	\centering
	\begin{tabularx}{\linewidth}{X @{\hspace{1cm}} lrrr @{\hspace{0.5cm}} lrrr @{\hspace{0.5cm}} lrrr} 
		\addlinespace
		\toprule
		& & \multicolumn{3}{c}{ (a) Homoskedastic } && \multicolumn{3}{c}{ (b) Heteroskedastic $t$ }  && \multicolumn{3}{c}{ (c) Heteroskedastic $\mathcal{SN}$ }    \\
		\cmidrule(lr){3-5} \cmidrule(lr){7-9} \cmidrule(lr){11-13} 
		$g_t(\xi)$ & & $\theta_1$ & $\theta_2$ & $\theta_3$ & &  $\theta_1$ & $\theta_2$ & $\theta_3$ & & $\theta_1$ & $\theta_2$ & $\theta_3$  \\
		\midrule
		\\
		& &\multicolumn{11}{c}{Panel A: $(\alpha, \beta) =  (0.5\%, 1\%)$} \\
		\cmidrule(lr){3-13}
		$\xi$ &   &  1.346 &  1.241 &  1.249 &   &  1.726 &  1.464 &  1.523 &   &  1.599 &  1.442 &  1.450\\
		$\exp(\xi)$ &   &  1.105 &  1.056 &  1.070 &   &  1.789 &  2.147 &  2.019 &   &  1.492 &  1.408 &  1.413\\
		$\log(\xi)$ &   &  1.394 &  1.277 &  1.284 &   &  1.870 &  1.594 &  1.633 &   &  1.624 &  1.461 &  1.469\\
		$F_{\text{Log}}(\xi)$ &   &  1.346 &  1.241 &  1.249 &   &  1.724 &  1.462 &  1.522 &   &  1.599 &  1.442 &  1.450\\
		$F_t(\xi)$  &   &  1.008 &  1.005 &  1.006 &   &  1.059 &  1.037 &  1.043 &   &  1.007 &  1.006 &  1.006\\
		\cmidrule(lr){3-13}
		Eff.~Bound &   &  5.216 &  3.903 &  3.570 &   & 30.869 & 17.564 & 16.438 &   &  6.491 &  3.267 &  3.254\\
		\midrule
		\\
		& &\multicolumn{11}{c}{Panel B: $(\alpha, \beta) = (5\%, 10\%)$} \\
		\cmidrule(lr){3-13}
		$\xi$ &   & 1.325 & 1.229 & 1.235 &   & 1.465 & 1.302 & 1.331 &   & 1.463 & 1.315 & 1.322\\
		$\exp(\xi)$ &   & 1.221 & 1.142 & 1.161 &   & 1.328 & 1.180 & 1.226 &   & 1.409 & 1.268 & 1.279\\
		$\log(\xi)$ &   & 1.339 & 1.240 & 1.244 &   & 1.491 & 1.325 & 1.351 &   & 1.471 & 1.321 & 1.328\\
		$F_{\text{Log}}(\xi)$ &   & 1.325 & 1.229 & 1.235 &   & 1.465 & 1.302 & 1.331 &   & 1.463 & 1.315 & 1.322\\
		$F_t(\xi)$  &   & 1.002 & 1.002 & 1.002 &   & 1.041 & 1.027 & 1.032 &   & 1.004 & 1.007 & 1.007\\
		\cmidrule(lr){3-13}
		Eff.~Bound &   & 2.360 & 1.737 & 1.622 &   & 5.417 & 3.637 & 3.318 &   & 3.338 & 1.952 & 1.952\\
		\midrule
		\\
		& &\multicolumn{11}{c}{Panel C: $(\alpha, \beta) =  (25\%, 50\%)$} \\
		\cmidrule(lr){3-13}
		$\xi$ &   & 1.280 & 1.198 & 1.201 &   & 1.335 & 1.228 & 1.233 &   & 1.234 & 1.183 & 1.181\\
		$\exp(\xi)$ &   & 1.555 & 1.360 & 1.414 &   & 1.613 & 1.383 & 1.448 &   & 1.480 & 1.347 & 1.381\\
		$\log(\xi)$ &   & 1.264 & 1.188 & 1.187 &   & 1.319 & 1.219 & 1.220 &   & 1.221 & 1.173 & 1.169\\
		$F_{\text{Log}}(\xi)$ &   & 1.280 & 1.198 & 1.201 &   & 1.335 & 1.228 & 1.233 &   & 1.234 & 1.183 & 1.181\\
		$F_t(\xi)$  &   & 1.000 & 1.000 & 1.000 &   & 1.001 & 1.001 & 1.001 &   & 1.023 & 1.019 & 1.019\\
		\cmidrule(lr){3-13}
		Eff.~Bound &   & 1.630 & 1.172 & 1.141 &   & 1.757 & 1.310 & 1.244 &   & 1.883 & 1.222 & 1.280\\
		\midrule
		\\
		& &\multicolumn{11}{c}{Panel D: $(\alpha, \beta) = (1\%, 99\%)$} \\
		\cmidrule(lr){3-13}
		$\xi$ &   &  1.305 &  1.191 &  1.191 &   &  1.350 &  1.248 &  1.248 &   &  2.277 &  2.088 &  1.217\\
		$\exp(\xi)$ &   &  4.798 &  3.638 &  3.236 &   & 10.343 & 10.625 &  6.131 &   & 14.938 & 11.490 &  4.276\\
		$\log(\xi)$ &   &  1.315 &  1.217 &  1.174 &   &  1.485 &  1.382 &  1.285 &   &  2.305 &  2.143 &  1.202\\
		$F_{\text{Log}}(\xi)$ &   &  1.305 &  1.191 &  1.191 &   &  1.350 &  1.247 &  1.248 &   &  2.277 &  2.088 &  1.217\\
		$F_t(\xi)$  &   &  1.000 &  1.000 &  1.000 &   &  1.000 &  1.000 &  1.000 &   &  1.000 &  1.000 &  1.000\\
		\cmidrule(lr){3-13}
		Eff.~Bound &   &  3.728 &  2.864 &  2.864 &   &  9.913 &  7.560 &  7.560 &   &  1.310 &  1.269 &  2.329\\
		\midrule
		\\
		& &\multicolumn{11}{c}{Panel E: $(\alpha, \beta) = (10\%, 99\%)$} \\
		\cmidrule(lr){3-13}
		$\xi$ &   &  1.272 &  1.202 &  1.108 &   &  2.387 &  1.674 &  1.390 &   &  1.328 &  1.265 &  1.083\\
		$\exp(\xi)$ &   &  5.218 &  3.907 &  2.997 &   & 35.481 & 20.956 &  8.469 &   &  7.024 &  5.280 &  2.772\\
		$\log(\xi)$ &   &  1.281 &  1.211 &  1.096 &   &  2.444 &  1.710 &  1.346 &   &  1.325 &  1.266 &  1.067\\
		$F_{\text{Log}}(\xi)$ &   &  1.272 &  1.202 &  1.108 &   &  2.387 &  1.674 &  1.390 &   &  1.328 &  1.265 &  1.083\\
		$F_t(\xi)$  &   &  1.065 &  1.046 &  1.022 &   &  1.565 &  1.244 &  1.373 &   &  1.014 &  1.011 &  1.002\\
		\cmidrule(lr){3-13}
		Eff.~Bound &   &  2.262 &  1.572 &  2.288 &   &  2.219 &  1.941 &  4.497 &   &  1.433 &  1.016 &  2.282\\
		\bottomrule
		\addlinespace
		\multicolumn{13}{p{.97\linewidth}}{
			This table presents the (approximated) asymptotic standard deviations for semiparametric double quantile models with joint model parameters given in (\ref{eqn:DQR_Models_JointParameters}) in the main article at different probability levels in the horizontal panels.
			The rows titled ``Eff.\ Bound'' report the raw standard deviations whereas the remaining rows report the relative standard deviations compared to the efficiency bound.
			Results for the three residual distributions described in Section \ref{sec:DQRSimStudy} are reported in the three vertical panels of the table.
			We furthermore consider four classical choices of $g_t(\xi)$ together with the (pseudo-) efficient choice $F_t(\xi)$ and the Z-estimation efficiency bound.
		}
	\end{tabularx}
\end{table}

\begin{table}[p]
	\scriptsize
	\caption{Asymptotic Standard Deviations of Quantile and ES Models}
	\label{tab:QESRSimSDTrue}
	\centering
	\begin{tabularx}{\linewidth}{XX @{\hspace{1cm}} lrrrr @{\hspace{0.5cm}} lrrrr}
		\addlinespace
		\toprule
		& & & \multicolumn{4}{c}{(a) Homoskedastic} && \multicolumn{4}{c}{(b) Heteroskedastic}  \\
		\cmidrule(lr){4-7} \cmidrule(lr){9-12} 
		$g_t(\xi_1)$ & $\phi_t(\xi_2)$& & $\theta_1$ & $\theta_2$ & $\theta_3$ & $\theta_4$ & &  $\theta_1$ & $\theta_2$ & $\theta_3$ & $\theta_4$  \\
		\midrule
		\\
		&  & &\multicolumn{9}{c}{Panel A: $\alpha=1\%$ and Models with Separated Parameters} \\
		\cmidrule(lr){3-12}
		$0$ & $\exp(\xi_2)$ &   &   1.289 &   1.496 &   1.140 &   1.232 &   &   4.541 &   9.016 &   2.551 &   4.120\\
		$F_t(\xi_1)$ & $\exp(\xi_2)$ &   &   1.270 &   1.451 &   1.140 &   1.232 &   &   2.619 &   2.759 &   2.551 &   4.120\\
		$0$ & $F_{\text{Log}}(\xi_2)$  &   &   1.278 &   1.477 &   1.127 &   1.213 &   &   4.534 &   8.997 &   2.545 &   4.108\\
		$F_t(\xi_1)$ & $F_{\text{Log}}(\xi_2)$ &   &   1.258 &   1.433 &   1.127 &   1.213 &   &   2.615 &   2.754 &   2.545 &   4.108\\
		$0$ & $-\log(-\xi_2)$ &   &   1.001 &   1.001 &   1.002 &   1.002 &   &   1.280 &   1.316 &   1.243 &   1.194\\
		$F_t(\xi_1)$ & $-\log(-\xi_2)$ &   &   1.001 &   1.001 &   1.002 &   1.002 &   &   1.280 &   1.316 &   1.243 &   1.194\\
		$0$ & $\phi_t^{\text{eff1}}(\xi_2)$  &   &   1.000 &   1.000 &   1.000 &   1.000 &   &   1.287 &   1.324 &   1.242 &   1.193\\
		$F_t(\xi_1)$ & $\phi_t^{\text{eff1}}(\xi_2)$ &   &   1.000 &   1.000 &   1.000 &   1.000 &   &   1.281 &   1.317 &   1.242 &   1.193\\
		$0$ & $\phi_t^{\text{eff2}}(\xi_2)$ &   &   1.000 &   1.000 &   1.000 &   1.000 &   &   1.330 &   1.354 &   1.168 &   1.122\\
		\multicolumn{2}{l}{Barendse Bound} &   &   1.043 &   1.041 &   1.000 &   1.000 &   &   1.209 &   1.238 &   1.168 &   1.122\\
		\cmidrule(lr){3-12}	
		\multicolumn{2}{l}{Efficiency Bound} &   &  13.879 &   7.649 &  17.058 &   9.401 &   &  55.070 &  24.918 & 123.549 &  70.916\\
		\midrule
		\\
		& & & \multicolumn{9}{c}{Panel B: $\alpha=10\%$  and Models with Separated Parameters} \\
		\cmidrule(lr){3-12}
		$0$ & $\exp(\xi_2)$&   &  1.143 &  1.226 &  1.045 &  1.071 &   &  1.513 &  1.801 &  1.397 &  1.514\\
		$F_t(\xi_1)$ & $\exp(\xi_2)$ &   &  1.053 &  1.067 &  1.045 &  1.071 &   &  1.153 &  1.155 &  1.397 &  1.514\\
		$0$ & $F_{\text{Log}}(\xi_2)$  &   &  1.131 &  1.207 &  1.035 &  1.057 &   &  1.499 &  1.777 &  1.382 &  1.490\\
		$F_t(\xi_1)$ & $F_{\text{Log}}(\xi_2)$ &   &  1.046 &  1.059 &  1.035 &  1.057 &   &  1.146 &  1.147 &  1.382 &  1.490\\
		$0$ & $-\log(-\xi_2)$ &   &  1.001 &  1.001 &  1.004 &  1.004 &   &  1.057 &  1.045 &  1.194 &  1.173\\
		$F_t(\xi_1)$ & $-\log(-\xi_2)$ &   &  1.001 &  1.001 &  1.004 &  1.004 &   &  1.055 &  1.043 &  1.194 &  1.173\\
		$0$ & $\phi_t^{\text{eff1}}(\xi_2)$ &   &  1.000 &  1.000 &  1.000 &  1.000 &   &  1.060 &  1.048 &  1.191 &  1.171\\
		$F_t(\xi_1)$ & $\phi_t^{\text{eff1}}(\xi_2)$ &   &  1.000 &  1.000 &  1.000 &  1.000 &   &  1.053 &  1.042 &  1.191 &  1.171\\
		$0$ & $\phi_t^{\text{eff2}}(\xi_2)$ &   &  1.000 &  1.000 &  1.000 &  1.000 &   &  1.104 &  1.085 &  1.117 &  1.098\\
		\multicolumn{2}{l}{Barendse Bound} &   &  1.043 &  1.041 &  1.000 &  1.000 &   &  1.076 &  1.064 &  1.117 &  1.098\\
		\cmidrule(lr){3-12}	
		\multicolumn{2}{l}{Efficiency Bound} &   &  6.353 &  3.500 &  7.157 &  3.943 &   & 12.629 &  6.742 & 20.111 & 11.274\\
		\midrule
		\midrule
		\\
		& & & \multicolumn{9}{c}{Panel C: $\alpha=1\%$ and Models with Joint Parameters} \\
		\cmidrule(lr){3-12}
		$0$ & $\exp(\xi_2)$ &   &  1.082 &  1.179 &  1.124 & &  &  2.571 &  4.824 &  3.019\\
		$F_t(\xi_1)$ & $\exp(\xi_2)$&   &  1.080 &  1.171 &  1.121 &&   &  2.466 &  4.203 &  2.888\\
		$0$ & $F_{\text{Log}}(\xi_2)$  &   &  1.081 &  1.162 &  1.114 & &  &  2.529 &  4.653 &  2.931\\
		$F_t(\xi_1)$ & $F_{\text{Log}}(\xi_2)$ &   &  1.079 &  1.154 &  1.111 &   &&  2.422 &  4.038 &  2.799\\
		$0$ & $-\log(-\xi_2)$ &   &  1.089 &  1.064 &  1.057 &&   &  1.711 &  1.481 &  1.326\\
		$F_t(\xi_1)$ & $-\log(-\xi_2)$ &   &  1.089 &  1.063 &  1.057 & &  &  1.711 &  1.481 &  1.326\\
		$0$ & $\phi_t^{\text{eff1}}(\xi_2)$ &   &  1.052 &  1.038 &  1.033 &  & &  1.793 &  1.612 &  1.401\\
		$F_t(\xi_1)$ & $\phi_t^{\text{eff1}}(\xi_2)$ &   &  1.052 &  1.038 &  1.033 &  & &  1.786 &  1.606 &  1.398\\
		$0$ & $\phi_t^{\text{eff2}}(\xi_2)$ &   &  1.029 &  1.021 &  1.018 &&   &  1.827 &  1.584 &  1.364\\
		\cmidrule(lr){3-12}	
		\multicolumn{2}{l}{Efficiency Bound} &   &  5.153 &  3.542 &  3.786 & &  & 25.453 & 15.221 & 23.534\\
		\midrule
		\\
		& & &\multicolumn{9}{c}{Panel D: $\alpha=10\%$ and Models with Joint Parameters} \\
		\cmidrule(lr){3-12}
		$0$ & $\exp(\xi_2)$ &   & 1.062 & 1.070 & 1.053 & &  & 1.766 & 1.716 & 1.501\\
		$F_t(\xi_1)$ & $\exp(\xi_2)$ &   & 1.033 & 1.029 & 1.032 & &  & 1.547 & 1.390 & 1.359\\
		$0$ & $F_{\text{Log}}(\xi_2)$ &   & 1.070 & 1.068 & 1.055 &  & & 1.764 & 1.688 & 1.485\\
		$F_t(\xi_1)$ & $F_{\text{Log}}(\xi_2)$ &   & 1.035 & 1.028 & 1.030 &  & & 1.527 & 1.366 & 1.338\\
		$0$ & $-\log(-\xi_2)$ &   & 1.124 & 1.088 & 1.079 &&   & 1.641 & 1.375 & 1.343\\
		$F_t(\xi_1)$ & $-\log(-\xi_2)$ &   & 1.102 & 1.073 & 1.066 & &  & 1.597 & 1.349 & 1.323\\
		$0$ & $\phi_t^{\text{eff1}}(\xi_2)$ &   & 1.023 & 1.016 & 1.015 &  & & 1.636 & 1.405 & 1.356\\
		$F_t(\xi_1)$ & $\phi_t^{\text{eff1}}(\xi_2)$ &   & 1.019 & 1.014 & 1.013 & &  & 1.595 & 1.378 & 1.334\\
		$0$ & $\phi_t^{\text{eff2}}(\xi_2)$ &   & 1.009 & 1.006 & 1.006 &&   & 1.667 & 1.408 & 1.331\\
		\cmidrule(lr){3-12}	
		\multicolumn{2}{l}{Efficiency Bound} &   & 2.310 & 1.598 & 1.659 & &  & 4.515 & 3.108 & 3.929\\
		\bottomrule
		\addlinespace
		\multicolumn{12}{p{.97\linewidth}}{\scriptsize
			This table presents the (approximated) asymptotic standard deviations for semiparametric joint quantile and ES models at joint probability level of $1\%$ and $10\%$  for various choices of M-estimators together with the Z-estimation efficiency bound and in Panel A and B, the two-step efficiency bound of \cite{Barendse2020} discussed in Section \ref{sec:TwoStepEfficiency}.
			The rows titled ``Efficiency Bound'' report the raw standard deviations whereas the remaining rows report the relative standard deviations compared to the efficiency bound.
			Panels A and B report results for the models with separated parameters given in 	(\ref{eqn:QESR_Models_SepParameters}) while Panel C and D consider the joint intercept models given in (\ref{eqn:QESR_Models_JointParameters}).
			The two considered residual distributions are presented in the two vertical panels of the table.
		}
	\end{tabularx}
\end{table}

\FloatBarrier

\end{document}